\documentclass[aap]{imsart}

\RequirePackage{amsthm,amsmath,amsfonts,amssymb}
\RequirePackage[numbers,sort&compress]{natbib}
\usepackage[bb=boondox]{mathalfa}
\usepackage{xcolor}
\usepackage{microtype}

\RequirePackage[colorlinks,citecolor=blue,urlcolor=blue]{hyperref}
\startlocaldefs
\hypersetup{colorlinks=true,allcolors=blue}      
\newtheorem{theorem}{Theorem}[section] 
\newtheorem{lemma}{Lemma}[section]
\theoremstyle{definition}
\newtheorem{definition}{Definition}[section]

\newtheorem{proposition}{Proposition}[section]
\newtheorem{corollary}{Corollary}[section]
\newtheorem{remark}{Remark}[section]
\newtheorem{assumption}{Assumption}[section]

\newcommand{\tmin}{t_{\textrm{min}}}
\newcommand{\tminA}{t_{\textrm{min},A}}
\newcommand{\Var}{\mathrm{Var}}
\newcommand{\dd}{\mathrm{d}}

\newcommand{\e}{\mathrm{e}}
\newcommand{\ps}{p_{\textrm{s}}}
\newcommand{\js}{j_{\textrm{s}}}
\newcommand{\vs}{v_{\textrm{s}}}

\DeclareMathOperator{\E}{\mathbb{E}}

\endlocaldefs
\begin{document}

\begin{frontmatter}
\title{Concentration of additive 
functionals of Stratonovich-type}
\runtitle{Concentration of Stratonovich-type functionals}

\begin{aug}
\author[A]{\fnms{Rick}~\snm{Bebon}\ead[label=e1]{rick.bebon@physik.uni-freiburg.de}
\orcid{0000-0003-2187-0008}}
\author[A]{\fnms{Alja\v{z} }~\snm{Godec}\ead[label=e2]{agodec@physik.uni-freiburg.de}\thanks{[\textbf{Corresponding author}]}\orcid{0000-0003-1888-6666}}
\author[A]{\fnms{Angelika}~\snm{Rohde}\ead[label=e3]{angelika.rohde@stochastik.uni-freiburg.de}}
\address[A]{Faculty of Mathematics and Physics, Albert-Ludwigs University of Freiburg\printead[presep={,\ }]{e1,e2,e3}}
\end{aug}

\begin{abstract}
Additive functionals $\overline{J}_t=\frac{1}{t}\int_0^tU(X_s)\circ \dd X_s$ of  Stratonovich-type 
recently attracted much attention 
in the context of inference of thermodynamic properties of complex systems from observations $U$ of individual fluctuating paths $(X_s)_{0\le s\le t}$, whereby $X_0$ is 
initiated from some general measure. 
Concentration results on $\overline{J}_t$, albeit desirable, are virtually nonexistent. They turn out to be significantly more challenging 
to prove than for classical Lebesgue-type functionals
$\overline{\rho}_t=\frac{1}{t}\int_0^t V(X_s) \dd s$ because
the tilt deforms the full second-order structure of the Feynman-Kac generator instead of  
contributing an additive potential.
This renders the generator generally non-self-adjoint even under detailed balance.
We overcome this by working with a symmetrized
Dirichlet form with a new effective potential that now couples
the observable to the non-equilibrium character of the dynamics.
We prove concentration inequalities for $\overline{J}_t$ for any bounded, sufficiently smooth vector-valued function $U$ of a general geometrically ergodic diffusion process $X_s$, including explicit sub-gamma and Bernstein-type inequalities, and we obtain explicit upper bounds on ${\Var}(\overline{J}_t)$. 
Strikingly, 
under detailed balance the concentration of $\overline{J}_t$
is distinctively sub-Gaussian at all times and all deviations, with a variance proxy
fixed by the noise alone and independent of the spectral gap, which has
no analog for $\overline{\rho}_t$.
\end{abstract}



\end{frontmatter}

\section{Introduction}
\subsection{Motivation: Uncertainty in time-series analysis with small data}
Additive functionals of ergodic Markov processes are 
fundamental objects in probability theory, and their fluctuations are described 
asymptotically via ergodic theorems, central limit theorems, and large deviation theory, 
and non-asymptotically through concentration inequalities.
The mathematical literature on this topic is vast but focuses almost exclusively on  
classical functionals of \emph{density} (i.e., Lebesgue) type, $t^{-1}\int_0^t V(X_s)\dd s$,
for which a well-developed non-asymptotic theory is available \cite{wu2000deviation,cattiaux2008deviation,lezaud2001chernoff,paulin2015concentration}.
Meanwhile, over the past decade, additive functionals of \emph{Stratonovich} type $t^{-1}\int_0^tU(X_s)\circ \dd X_s$---so-called
\emph{generalized currents}---have moved to the center of attention in non-equilibrium statistical mechanics 
and stochastic thermodynamics, where they define energy-like quantities (such as heat, work, and entropy production) of mesoscopic systems along individual 
stochastic trajectories \cite{sekimoto2010,seifert2012stochastic,seifert2025stochastic,peliti2021stochastic}.
The study of Stratonovich integrals is further motivated as they (in contrast to It\^o integrals) generalize the notion of a continuity equation for additive functionals to individual sample paths \cite{dieball2022coarse, dieball2022mathematical}
(see Remark~\ref{rem:strato_motivation} 
for further motivation and discussion).

From a path-wise perspective on thermodynamics, their
finite-time fluctuations carry intrinsic, nontrivial physical meaning, 
and should be understood not only as statistical uncertainty but as 
\emph{physical observables in their own right},
as exemplified by 
the so-called \emph{thermodynamic uncertainty relation} (TUR) which bounds the fluctuations of any such 
current from below by the mean entropy production
\emph{of the underlying dynamical process} \cite{seifert2012stochastic, seifert2025stochastic}.
Controlling Stratonovich-type current fluctuations is moreover indispensable for
interpreting state-of-the-art experiments (e.g., single-molecule and particle-tracking
measurements~\cite{neuman2008single,camunas2016elastic,saxton2008,ernst2014})---where they 
are evaluated directly from measured time series without
knowledge of the underlying dynamics---and in controlling uncertainty 
in thermodynamic inference~\cite{Seifert2019,horowitz2020thermodynamic,Dieball_2025_P,dieball2023direct}.
The latter critically hinges on the feasibility of
adequately sampling the second 
or higher-order~\cite{dechant2020,Dechant2021,Wampler2021,Manikandan2022,Ray2023} moments,
or the full tails of the distribution~\cite{collin2005verification,alemany2015free,ribezzi2014free,Pohorille2010,Hummer2001,Zuckerman2002,Jarzynski2006,Gore2003,engel2009asymptotics}.
In practice, owing to experimental constraints and intrinsically small sample sizes
one typically observes only a handful (up to $\sim100$) of independent trajectories,
each of finite duration, frequently short relative to the mixing time of the underlying dynamics.
Therefore, neither stationarity nor standard asymptotic error analysis can be invoked.
In this regime, central-limit and bootstrapping methods
implicitly assume the sample to be representative of the target statistic and tend to underestimate uncertainty~\cite{schenker1985qualms,Davison1997,Shao2003,hogg1995introduction}, while Bayesian credible intervals~\cite{gelman1995bayesian}, though not asymptotic, are sensitive to, and generally biased by, the choice of prior~\cite{Smid2019,Mostofian2019}.

Yet, in contrast to the density-type case, the mathematical literature addressing
a non-asymptotic control of the fluctuations of Stratonovich functionals of diffusion processes
is essentially absent. Importantly, it turns out that proving concentration inequalities for Stratonovich functionals is substantially more challenging due to the higher complexity of the Feynman-Kac semigroup which carries through the entire proof. In contrast to the standard Lebesgue-type additive functionals, the generator of the Feynman-Kac semigroup is also almost never (except for highly specialized settings) self-adjoint. Therefore, the proof strategies developed for Lebesgue-type additive functionals do not directly generalize and a different approach is required. 

Here we fill this gap by developing a non-asymptotic theory of the concentration of additive Stratonovich-type functionals.
This in turn yields the missing uncertainty quantification framework for
finite-time and small-sample control,
and makes precise the meaning of ``sufficiently long trajectory'' and ``sufficiently many given the length'' in
practice. Furthermore, physical implications are developed in a companion manuscript.

\subsection{Preview of proof strategy and main results}
Here and in the following we consider ergodic It\^o diffusions $(X_t)_{t\geq 0}$  
on $\mathbb{R}^d$ with additive noise, Markov generator
$\mathcal{L}=b\cdot\nabla+\nabla\cdot D\nabla$, invariant probability
measure $\mu$ (with density $\dd\mu/\dd x=\ps$), and initial law $\nu\ll\mu$ with $\dd\nu/\dd\mu\in
L^2(\mu)$.
For our explicit results (Section~\ref{sec:conc_poincare}) we further assume that the symmetric
part $\mathcal{L}_S$ of $\mathcal{L}$ in $L^2(\mu)$ has a spectral gap
$\lambda_{\mathrm{gap}}>0$, which is equivalent to the process being
geometrically ergodic (see Section~\ref{sec:setup} for the precise
assumptions). 
We study two structurally distinct classes of additive functionals of
$(X_t)_{t\geq0}$, obtained by time-averaging a single path over $[0,t]$.
Functionals of \emph{density type},
\begin{align}
  \overline{\rho}_t(V) = \frac{1}{t}\int_0^t V(X_s)\dd s,
  \label{eq:density_intro}
\end{align}
with $V$ a scalar observable, are ordinary Lebesgue integrals along the
path and have zero quadratic variation, whereas functionals of
\emph{current (or Stratonovich) type},
\begin{align}
  \overline{J}_t(U) = \frac{1}{t}\int_0^t U(X_s)\circ\dd X_s,
  \label{eq:current_intro}
\end{align}
with $U$ being a vector field and $U(X_s)\circ\dd X_s=\sum_i U_i(X_s)\circ\dd X^i_s$, are Stratonovich stochastic integrals that carry
a non-trivial quadratic variation. For suitable choices of $U$ the latter
realize the \emph{generalized currents} of stochastic thermodynamics
(their precise definitions, physical interpretation, and the associated literature are found in
Section~\ref{sec:setup_observables}).

Our central goal is to prove non-asymptotic tail bounds of the form
\begin{align}
  \mathbb{P}^\nu\left(\overline{A}_t - \E^\mu[\overline{A}_t] \geq a\right)
  \leq
  N_\nu\e^{-t I^A\bigl(\E^\mu[\overline{A}_t]+a\bigr)},
  \qquad
  N_\nu \equiv \left\|\frac{\dd\nu}{\dd\mu}\right\|_{L^2(\mu)},
  \label{eq:conc_sketch_intro}
\end{align}
that are valid for \emph{all} $t>0$ and $a>0$, for both $\overline{A}_t\in\{\overline{\rho}_t(V),\overline{J}_t(U)\}$, with an explicit rate $I^A$ and a prefactor $N_\nu$ that encodes the
initial condition. From this result we obtain confidence intervals on estimates from single trajectories as well as (small-)sample means. 

For density-type the route to such bounds is by now classical \cite{wu2000deviation,cattiaux2008deviation,lezaud2001chernoff,Birrell2025}.
The key strategy involves Chernoff's inequality, which reduces the tail probability
to the moment generating function $\E^\nu[\e^{kA_t}]$, which itself is encoded via a Feynman-Kac
semigroup $P_t^{k,\rho}$ whose corresponding \emph{tilted generator}
is given by a well-known Schr\"odinger-type operator that takes the simple
form $\mathcal{L}_k^\rho=\mathcal{L}+kV$, i.e. , the tilt that enters 
only as an additive multiplication operator that leaves the original 
differential structure of $\mathcal{L}$ untouched.
In further steps, its quadratic form only requires
the symmetric $\mathcal{L}_S+kV$ part, whose principal eigenvalue $\Lambda^\rho(k)$
turns into the norm bound $\|P_t^{k,\rho}\|_{L^2(\mu)}\leq\e^{t\Lambda^\rho(k)}$ via a Lumer-Phillips argument. Lastly, a Poincar\'e inequality gives explicit control of $\Lambda^\rho(k)$
and yields sub-gamma and Bernstein-type bounds.

While our approach conceptually relies on similar ideas developed in those works, 
we emphasize that the difficulty and novelty 
in extending non-asymptotic results
to additive functionals of the Stratonovich-type 
requires additional and conceptually new input at each step of the proof.
Due to the higher complexity of the corresponding Feynman-Kac 
semigroup---whose tilt is no longer a simple additive potential, precisely because of the non-trivial quadratic variation of $\overline{J}_t$---the 
analysis of Stratonovich functionals turns out to be substantially more delicate than their density-type counterpart.
Non-asymptotic concentration results for $\overline{J}_t(U)$, although
highly desirable, have to the best of our knowledge remained essentially non-existent.

We next describe the main challenges and how we solve them. To highlight the newly appearing mathematical challenges, we contrast the proof for Stratonovich functionals with the 
density-case. In particular, we re-derive corresponding density-type bounds adapted
to our setting alongside (mainly following \cite{wu2000deviation,cattiaux2008deviation})
to make the comparison explicit.

\subsubsection*{The current tilt is a differential operator}
Because the exponent of the Feynman-Kac semigroup $P_t^{k,J}$ contains a stochastic Stratonovich integral with non-trivial quadratic variation, its associated tilted generator (Proposition~\ref{prop:tilted_gen_current}) is obtained from the underlying $\mathcal{L}$
by a gradient shift $\nabla\to\nabla+kU$ and reads
\begin{align}
  \mathcal{L}_k^J
  = b\cdot(\nabla+kU) + (\nabla+kU)\cdot D(\nabla+kU).
  \label{eq:tilted_contrast_intro}
\end{align}
Compared to the density case, $\mathcal{L}_k^\rho = \mathcal{L} + kV$, the current tilt
thus \emph{deforms the full second-order structure} of the Markov generator.
In particular, it gives rise to a new first-order term $2kU\cdot D\nabla$
and (non-negative) zeroth-order term $k^2(U\cdot DU)$, neither of which appears for
density-type observables (see Remark~\ref{rem:structural_comparison}).
As a result, the classical Feynman-Kac approach of the density (i.e., Lebesgue-type) functionals no longer applies.

\subsubsection*{Loss of self-adjointness under detailed balance}
In the established density case the tilted generator $\mathcal{L}^\rho_k$ fails to be self-adjoint
in $L^2(\mu)$ for all $k$ \emph{only} when the underlying dynamics is non-reversible, i.e., only when
$\mathcal{L}$ itself is not self-adjoint.
Detailed balance thus gives a particularly easy case where
spectral control is inherited directly ``for free''.
For current observables, however, this is \emph{not} the case anymore.
Due to the new first-order term  $2kU\cdot D\nabla$ the tilted operator $\mathcal{L}_k^J$
is now non-self-adjoint for every $k\neq0$ and non-trivial $U$, 
\emph{even if the underlying dynamics obeys detailed balance}
(Section~\ref{sec:selfadjoint_tilted}).
We solve this issue by bounding the semigroup norm via the quadratic form $\langle\mathcal{L}_k^Jf,f\rangle_\mu$ that becomes independent of the antisymmetric part of the tilted operator 
(Lemma~\ref{lem:quad_identity}). 
This allows us to compute the explicit $L^2(\mu)$-adjoint in terms of the 
reversible-irreversible decomposition of the drift $b=D\nabla\log\ps+\vs$ 
and express the new symmetrization again in a now different 
Schr\"odinger from (Proposition~\ref{prop:sym_tilted_explicit}),
\begin{align}
\widetilde{\mathcal{L}}_k^J
  = \mathcal{L}_S + kW + k^2 Q,
  \label{eq:sym_contrast_intro}
\end{align}
where $W \equiv \vs\cdot U$ and $Q \equiv U\cdot DU \geq 0$.
This crucial step reveals, only on the level of the quadratic forms, 
a new \emph{effective potential}
$kW+k^2Q$ for the Stratonvich-type functionals that, however, differs 
fundamentally from the density counterpart in two essential ways.

First, the linear part now not only couples to the observable but \emph{also to
the non-equilibrium structure of the dynamics} through the local mean velocity $\vs$ (i.e., $k \vs \cdot U$ compared to 
simply $kV$) and vanishes identically under detailed balance.
Second, the quadratic part has its origin in the non-vanishing quadratic variation of the stochastic
integral and couples to the noise $D$ of the underlying dynamics
but does \emph{not} vanish at equilibrium.
Both new contributions are specific to Stratonovich-type functionals and have far-reaching consequences
for all results that follow.
The effective potential of the Schr\"odinger operator
$\mathcal{L}_S + kW + k^2 Q$ thus encodes dynamical information about the irreversibility of the underlying dynamics.
Strikingly, detailed balance turns out to be more
consequential for Stratonovich-type functionals compared to the density-type, 
as it leads to qualitatively different concentration behavior
(see below and Corollary~\ref{cor:db_subgaussian}).

\subsubsection*{Dissipativity-restoring bound on the norm for all times}
Since the exponential tilting via the Stratonovich integral destroys 
the contraction property of the underlying Markov semigroup, the 
corresponding operator norm may grow in $t$.
A key step now is to identify this precise growth rate which we establish by restoring
dissipativity through an explicit shift of the generator.
Indeed, shifting $\mathcal{L}_k^J$ by the supremum of its quadratic form $\Lambda^J(k)$---by 
Lemma~\ref{lem:quad_identity} and the Rayleigh-Ritz theorem we identify
$\Lambda^J(k)$ as the principal eigenvalue of the self-adjoint operator $\widetilde{\mathcal{L}}_k^J$---restores
the dissipative property. By the Lumer-Phillips theorem, this yields 
the upper bound $\|P_t^{k,J}\|_{L^2(\mu)}\leq\e^{t\Lambda^J(k)}$ for all $t\geq0$
(Proposition~\ref{prop:fk_norm_bound}). Following up with a Cauchy-Schwarz argument, we bound
the moment generating function with a prefactor $N_\nu$ that encodes the finite-time information about the initial condition which is lost in the limit $t\to\infty$ (Corollary~\ref{cor:mgf_bound}).
Subsequent optimization of the Chernoff bound gives the general concentration inequality
(Theorem~\ref{thm:conc_current}) in the form of Eq.~\eqref{eq:conc_sketch_intro}
where $I^J$ is the Legendre transform of $\Lambda^J(k)$. 
Our strategy relies on $\Lambda^J(k)$, i.e., the principal eigenvalue of the \emph{symmetrized} rather than the full tilted generator, which gives us insights into what is lost in this step.
We prove $\lambda_{\max}(\mathcal{L}_k^J)\leq\Lambda^J(k)$ together with 
conditions for equality (the antisymmetric part has to exactly annihilate the principal eigenfunction of the symmetrized generator) and identify explicit cases for the dynamics and observables
where this holds (Proposition~\ref{prop:tightness}).
These insights allow us to characterize when our non-asymptotic bounds are sharp to exponential order, i.e., in the sense of large deviation theory (Remark~\ref{rem:ldt_sharpness}).

\subsubsection*{Explicit bounds via a Poincar\'e inequality}
To make $\Lambda^A(k)$, and therefore its Legendre transform, explicit we assume that $\mathcal{L}_S$ obeys a Poincar\'e inequality with constant $\lambda_\mathrm{gap}^{-1}$. 
In the established density case, the potential is linear in $k$ and with a parametrization
of the test function $f=(1+\varepsilon g)/\sqrt{1+\varepsilon^2}$ the variational problem for $\Lambda^\rho(k)$
reduces to a supremum of a concave quadratic function in $\varepsilon$, such 
that the resulting bound takes a sub-gamma form $\tilde\sigma^2k^2/[2(1-ck)]$ whose Legendre transform is
explicit~\cite{boucheron2013concentration}. 
For current observables the new quadratic term $k^2\langle Q,f^2\rangle_\mu$---present even under detailed balance---prevents this direct approach, and we follow two different strategies.

In the first approach, we bound the quadratic term uniformly $\langle Q,f^2\rangle_\mu\leq\|Q\|_{L^\infty(\mu)}$
and via an additional estimate absorb the resulting term $k^2\|Q\|_{L^\infty(\mu)}$ directly into
the variance proxy. This has the advantage that it preserves the sub-gamma structure at the price of the supremum norm (Theorem~\ref{thm:explicit_current_1}).
The second strategy treats the quadratic $Q$ differently 
by centering it and controlling $Q-\E^\mu[Q]$ via the Poincar\'e inequality. 
The price of this sharper estimate is that the resulting optimization 
for the Legendre transformation is now a \emph{quintic} in $k$ with no closed-form
root available.
We solve this issue by relaxing the restricted tilt (i.e., optimization range)
$k < k_{\max}^\ast$ where $k_{\max}^\ast$ solves an explicit quadratic equation.
This trick then restores a sub-gamma form with the mean $\E^\mu[Q]$ replacing
$\|Q\|_{L^\infty(\mu)}$, however, now with a larger exponential scale parameter
$c_{J,2}\geq c_{J,1}$ (Lemma~\ref{lem:bernstein_current} and Theorem~\ref{thm:explicit_current_2}).
The two bounds are genuinely different since whenever the second bound sharpens the Gaussian regime, it does so at the cost of an earlier onset of the exponential tail (Remark~\ref{rem:approach_comparison}).
This trade-off, like the appearance of $Q$ itself, has
\emph{no} counterpart for density-type functionals. 
Moreover, since $W$ couples to the 
irreversible part of the dynamics $\vs$,
our results show fundamental qualitative differences
in the fluctuations of Stratonovich-type functionals for systems
obeying or breaking detailed balance.
Under detailed balance $W\equiv0$
and the entire exponential tail of the bound for $\overline{J}_t$ vanishes,
such that the bound is sub-Gaussian with a variance proxy
of $2\| U\cdot DU\|_{L^\infty(\mu)}/t$ for all
$t>0$ and deviations $a>0$,
and the bound no longer contains the spectral gap $\lambda_{\mathrm{gap}}$
(Corollary~\ref{cor:db_subgaussian} and Remark~\ref{rem:db_bounds}).
Notably, this has no counterpart in the density case,
where the exponential tail parameter $c_\rho = \|V - \E^\mu[V]\|_{L^\infty(\mu)}/\lambda_{\mathrm{gap}}$ does not vanish under detailed balance.
This different concentration behavior for
current observables has far-reaching (physical) 
implications for the uncertainty quantification in (thermodynamic) 
inferences of additive functionals (see companion manuscript).

Importantly, the most intricate part of our proof lies not only in establishing these concentration inequalities, but in recognizing the underlying structural properties (e.g., on the tilted current generator), whose validity is precisely what allows the argument to go through. We believe that the techniques and overall strategy developed here may serve as a foundation for further non-asymptotic concentration results for stochastic-integral functionals, as we discuss in Section~\ref{sec:conclusion}.

\subsubsection*{Summary of main results}
Concretely, we establish the following results.

\begin{itemize}
    \item For general ergodic (possibly non-reversible) diffusions we prove the non-asymptotic
    concentration inequality in the form of Eq.~\eqref{eq:conc_sketch_intro} 
    for Stratonovich currents (Theorem~\ref{thm:conc_current}, with left-tail and two-sided
    versions in Corollary~\ref{cor:two_sided_current}), which is expressed
    through the Legendre transform $I^J$ of the principal eigenvalue $\Lambda^J(k)$ of the symmetrized tilted generator~\eqref{eq:sym_contrast_intro}. 
    For reversible dynamics, where $\js\equiv0$, the bound becomes
    purely sub-Gaussian for all $t>0$ and $a>0$, without requiring
    a Poincar\'e inequality (Corollary~\ref{cor:db_subgaussian}).
    Further, using spectral arguments, 
    we give conditions
    under which $\Lambda^J(k)$
    corresponds to the scaled cumulant generating function,
    $\Lambda^J(k) = \lim_{t\to\infty}\tfrac{1}{t}\log\E\bigl[\e^{k J_t}\bigr]$, and hence becomes sharp up to exponential order
    in the sense of large deviation theory (Proposition~\ref{prop:tightness}).
    For comparison, we re-derive the corresponding concentration bound for density-type 
    functionals and adapt it to our diffusion setting following~\cite{wu2000deviation} in order to
    make the contrast with current-type functionals explicit.

    \item  Under the assumption of a Poincar\'e inequality on the underlying dynamics (i.e., a geometrically ergodic process) we obtain fully explicit sub-gamma (and further
    Bernstein-type) bounds,
    \begin{align}
      \mathbb{P}^\nu\left(\overline{J}_t(U) - \E^\mu[\overline{J}_t] \geq a\right)
      \leq
      \left\|\frac{\dd\nu}{\dd\mu}\right\|_{L^2(\mu)}
      \exp\left[-t\frac{\tilde\sigma_{J}^2}{c_{J}^2}h\left(\frac{c_{J} a}{\tilde\sigma_{J}^2}\right)\right],
    \end{align}
    with $h(u) = 1 + u - \sqrt{1+2u}$. 
    Using two different strategies (see above), we derive two sets of
    effective parameters $\tilde\sigma_{J,i}^2$ and $c_{J,i}$, $i\in\{1,2\}$, that depend
    on the spectral gap $\lambda_\mathrm{gap}$ of the
    reversible part of the dynamics and on norms of the vector field $U$
    (Theorems~\ref{thm:explicit_current_1}  and~\ref{thm:explicit_current_2}).
    Consistently, reversibility gives $c_{J,1}=0$ in Approach 1
    and recovers Corollary~\ref{cor:db_subgaussian}.
    We re-derive the analogous density-type
    bound in the same framework, following~\cite{cattiaux2008deviation}, to highlight the qualitative differences. 

\item  The same analysis further yields explicit finite-time upper bounds on the variance (and by extension the asymptotic variance) of both functionals 
in terms of the steady-state current $\js(x)$ and density $\ps(x)$ (Proposition~\ref{prop:variance_explicit}). In particular, $\mathrm{Var}_\mu(\overline\rho_t)
    \leq \frac{2\mathrm{Var}_\mu(V)}{\lambda_\mathrm{gap}t}$ and 
\begin{align}
\mathrm{Var}_\mu(\overline J_t)
    &\leq \frac{1}{t}\left[2\langle U, DU \rangle_\mu
      +\frac{2\mathrm{Var}_\mu([\js/\ps]\cdot U)}{\lambda_\mathrm{gap}}\right].
\end{align}

\item
By inverting the explicit concentration bounds, we develop a non-asymptotic framework 
for uncertainty quantification which, to the best of our knowledge, is new
for both classes of additive functionals.
  When estimating the stationary mean $\E[\overline{A}_t]$ from a single time-average
  of length $t$, or,  equivalently from the sample mean $\widehat{A}_{n,t} = n^{-1}\sum_{i=1}^{n}\overline{A}_t^{(i)}$ over $n$
  independent realizations, we construct finite-time
  confidence intervals such that for every $n \geq 1$ and $t>0$ the stationary mean falls
  within $\widehat{A}_{n,t} \pm r_A(\alpha,n,t)$ with probability of at least $1-\alpha$,
  where (Theorem~\ref{thm:ci} and Corollary~\ref{cor:ci_sample_mean})
  \begin{align}
  r_A(\alpha,n,t)
  =
  \sqrt{\frac{2\tilde\sigma_A^2}{nt}\left(\log\frac{2}{\alpha} + n\log N_\nu\right)}
  +
  \frac{c_A}{nt}\left(\log\frac{2}{\alpha} + n\log N_\nu\right),
  \quad
  N_\nu \equiv \left\|\frac{\dd\nu}{\dd\mu}\right\|_{L^2(\mu)}.
  \end{align}
Moreover, solving $r_A(\alpha,n,t)\leq\varepsilon$ leads to the requirement
\begin{align}
  nt \geq \Theta_A(\varepsilon)\left(\log\frac{2}{\alpha} + n\log N_\nu\right),
 \qquad
  \Theta_A(\varepsilon)
  \equiv
  \frac{1}{2\varepsilon^2}
  \left(\sqrt{\tilde\sigma_A^2} + \sqrt{\tilde\sigma_A^2 + 2 c_A \varepsilon}\right)^{2} ,
\end{align}
which fixes the minimal sample size $n_{{\rm min},A}$ 
and the minimal trajectory length $\tminA$
\begin{align}
n_{{\rm min},A}(t,\varepsilon,\alpha)
=
\left\lceil
\frac{\Theta_A(\varepsilon)\log(2/\alpha)}{t - \Theta_A(\varepsilon)\log N_\nu}
\right\rceil, 
\qquad  t_{\min,A}(n, \varepsilon, \alpha)
  =  \frac{\Theta_A(\varepsilon)}{n}\left(\log\frac{2}{\alpha} + n\log N_\nu\right),
\end{align}
  required to estimate the stationary mean with tolerance 
  $\varepsilon$ at confidence $1-\alpha$ from $n$ time-series of
  length $t$ (Theorems~\ref{thm:tmin} and~\ref{thm:min_sample}). For a single trajectory ($n=1$) the same argument gives the minimal trajectory length $\tminA(\varepsilon,\alpha) =\Theta_A(\varepsilon)\log(2N_\nu/\alpha)$
  to reach tolerance $\varepsilon$.      
\end{itemize}

\paragraph{Relation to previous work}
Non-asymptotic deviation inequalities of \emph{density-type} functionals
(Eq.~\eqref{eq:density_intro}) of ergodic, possibly non-reversible Markov processes were established
in~\cite{wu2000deviation} via the symmetrized Dirichlet form and the
Lumer-Phillips theorem, and were subsequently extended through functional and transportation-cost inequalities~\cite{cattiaux2008deviation,gao2014bernstein,guillin2009transportation}
(see \cite{Girotti2023}
for quantum Markov processes).
Related Bernstein-type inequalities for ergodic Markov chains were obtained via 
spectral methods in~\cite{lezaud2001chernoff,paulin2015concentration} 
and extended to the hypocoercive setting (including non-asymptotic confidence intervals) in \cite{Birrell2025}
all, however, only for density-type observables.
Non-asymptotic concentration bounds
for time-asymmetric observables recently appeared for
discrete-state Markov jump processes \cite{bakewell2023general,bakewell2025bounds}
but remain elusive for continuous-state It\^{o} diffusions.
The asymptotic (large-time) description is typically formulated in terms of
large deviation theory~\cite{donsker1975asymptotic1, donsker1975asymptotic2,donsker1976asymptotic, donsker1983asymptotic,Ferr2020} 
(see also \cite{ReyBellet2015, ReyBellet2016} 
for improved convergence in irreversible Langevin samplers)
and for current-type functionals specifically 
it includes the level $2.5$
large deviation description~\cite{maes2008steady,barato2015formal,chetrite2015nonequilibrium,hoppenau2016level}
and macroscopic fluctuation theory~\cite{bertini2015flows,bertini2002macroscopic}. 
The thermodynamic significance of current fluctuations, and the thermodynamic uncertainty relation in
particular, is the subject of an extensive literature in stochastic
thermodynamics~\cite{barato2015,gingrich2016dissipation,horowitz2020thermodynamic,pietzonka2018universal,dechant2020,dieball2023direct}.
The concentration results developed here advance the existing work
in the following aspects.

First, the results are non-asymptotic, holding for every fixed
$t>0$, in contrast to large-deviation results, which hold only as
$t\to\infty$. Second, they control the full distribution through
exponential tail bounds, whereas the thermodynamic uncertainty relation constrains only the variance.
Third, they apply to current-type functionals
where results for It\^{o} diffusions are limited to density-type functionals. 
Fourth, the resulting uncertainty quantification---explicit confidence intervals and
minimal trajectory lengths and minimal sample sizes---is new for both classes of additive functionals.
Fifth, whereas the thermodynamic uncertainty relation bounds current fluctuations (i.e., their variances)
from below by the dissipation, our bounds on the full tail probability 
bound fluctuations from above and add
to the recent literature of so-called inverse thermodynamic uncertainty 
relations~\cite{bakewell2023general,bakewell2025bounds,vo2025inverse}. 

\paragraph{Organization of the manuscript}
The remainder of the manuscript is organized as follows.
Section~\ref{sec:setup} fixes the setting, states the assumptions on the dynamics, 
and introduces the additive functionals of interest in more detail.
In Section~\ref{sec:feynman_kac} we develop the theory of Feynman-Kac semigroups and their
associated tilted generators for both classes of observables.
We further establish their connection to Chernoff's inequality through a bound on the operator norm
that serves as the basis for all subsequent results.
In Section~\ref{sec:conc_bound_general} we prove the general concentration
inequalities from the spectral properties of the symmetrized tilted
generators, and in Section~\ref{sec:conc_poincare} we further prove explicit
sub-gamma and Bernstein-type bounds under a Poincar\'e inequality,
together with resulting upper bounds on the variance
and asymptotic variance of the additive functionals.
In Section~\ref{sec:uq} we develop a framework for uncertainty quantification based on non-asymptotic confidence intervals and minimal observation times and sample sizes. 

\section{Setup: dynamics and additive functionals}
\label{sec:setup}
In this section, we first introduce the setting by recalling
the basic framework of It\^o diffusions from the operator-theoretic point of view. Next, in Section~\ref{sec:setup_observables} we define and give further background on additive functionals, i.e.~density- and (Stratonovich) current-type functionals. 

\subsection{Underlying Markov dynamics}
\label{sec:setup_dynamics}
Throughout $(X_t)_{t\ge0}$ is an ergodic time-homogeneous It\^o diffusion on the
continuous state space $E=\mathbb{R}^d$, 
with invariant measure $\mu$, that is 
realized as the (strong) solution
of the stochastic differential equation
\begin{align}
  \dd X_t = b(X_t)\dd t + \sigma\dd W_t,
  \label{eq:sde_setup}
\end{align}
with (Lipschitz) drift vector $b:\mathbb{R}^d\to\mathbb{R}^d$, constant matrix-valued additive noise amplitude
$\sigma\in\mathbb{R}^{d\times m}$, and $(W_t)_{t\ge0}$ a standard $m$-dimensional
Wiener process defined on a filtered probability space 
$(\Omega,\mathcal{F},(\mathcal{F}_t)_{t\ge0},\mathbb{P})$
where $\mathcal{F}_t=\sigma(W_s:s\le t)$ denotes the natural filtration, i.e., the smallest $\sigma$-field for which all random variables $W_s$ for $s\leq t$ are measurable
and $X_t$ is a $\mathcal{F}_t$-adapted c\`adl\`ag process.
We write $(\mathbb{P}^x)_{x\in \mathbb{R}^d}$ for the family
of laws started deterministically from $X_0=x$ and for a general
initial distribution $\nu$ we set $\mathbb{P}^\nu=\int_{\mathbb{R}^d}\mathbb{P}^x \dd\nu(x)$. 
The transition semigroup of $(X_t)_{t\ge0}$ extends to a
strongly continuous contraction semigroup $(P_t)_{t\ge0}$ on the real Hilbert space
$L^2(\mu) \equiv L^2(\mathbb{R}^d,\mu)$~\cite{bakry2014analysis,ethier2009markov}, where $\mu$ is an invariant probability
measure with smooth, strictly positive stationary density $\dd\mu=\ps(x) \dd x$. 
The corresponding infinitesimal generator of $(P_t)_{t\ge0}$ is denoted by
$(\mathcal{L},D(\mathcal{L}))$ and the $L^2(\mu)$ adjoint is given by $\mathcal{L}^\dagger$.
In our setting the generator takes
the explicit form, for $f\in D(\mathcal{L})$,
\begin{align}
  \mathcal{L} = b(x)\cdot\nabla + \nabla\cdot D\nabla,
  \label{eq:generator_setup}
\end{align}
with constant, symmetric, positive-definite diffusion $d\times d$ matrix $D\equiv \tfrac12\sigma\sigma^\top$. 
For the stationary process we take $X_0\sim\mu$ with $\E^\mu$ denoting the corresponding expectation
and $\mathrm{Var}_\mu$ the stationary variance.
Moreover, we write $\langle f,g\rangle_\mu=
\int_{\mathbb{R}^d} f g \dd\mu$ for the inner product on $L^2(\mu)$,
$\|\cdot\|_{L^2(\mu)}$ for the induced norm, and $\|\cdot\|_{L^\infty(\mu)}$ for the sup-norm.

\subsubsection{Reversible-irreversible drift decomposition}\label{d_decomp}
Since the dynamics of Eq.~\eqref{eq:sde_setup} is generally non-reversible (i.e., violates detailed balance), 
$\mathcal{L}$ is in general not self-adjoint on $L^2(\mu)$, and we throughout use its symmetric-antisymmetric 
decomposition $\mathcal{L}=\mathcal{L}_S+\mathcal{L}_A$ on $L^2(\mu)$, where 
$\mathcal{L}_S=\tfrac12(\mathcal{L}+\mathcal{L}^\dagger)$ and
$\mathcal{L}_A=\tfrac12(\mathcal{L}-\mathcal{L}^\dagger)$
\cite{pavliotis2014stochastic,jiang2004mathematical,komorowski2012fluctuations,qian2013decomposition}. 
For It\^o diffusions this decomposition 
splits the drift into a gradient and a non-gradient part,
\begin{align}
  b(x) = D\nabla\log\ps(x) + \vs(x),
  \qquad
  \vs(x) \equiv \frac{\js(x)}{\ps(x)},
  \label{eq:drift_split}
\end{align}
where $\vs(x)$ is the so-called
\emph{local mean velocity} and $\js(x) \equiv b(x)\ps(x) - D\nabla\ps(x)$ denotes the (divergence-free) \emph{stationary probability current},
and yields,
\begin{align}
\mathcal{L}_S f=\ps^{-1}\nabla\cdot(\ps D\nabla f),
\quad\quad
\mathcal{L}_A f=\vs\cdot\nabla f.
\label{eq:markov_gen_decomposition}
\end{align}
The $L^2(\mu)$-self-adjoint generator $\mathcal{L}_S$ corresponds to the reversible part of
the underlying dynamics with gradient drift $D\nabla\log\ps$ and $\mathcal{L}_A$
is the skew-adjoint generator that maintains non-equilibrium steady-states.
The process is reversible---equivalently, at equilibrium and satisfies detailed
balance---if and only if $\mathcal{L}_A=0$, i.e., $\js\equiv0$.

\subsubsection{Dirichlet form and spectral gap}
The symmetric part $\mathcal{L}_S$ generates a reversible Markov semigroup with the
same invariant measure $\mu$ and is the operator associated with the symmetrized
\emph{Dirichlet form} 
\begin{align}
  \mathcal{E}(f,g)\equiv-\langle\mathcal{L}_S f,g\rangle_\mu,
  \qquad
  \mathcal{E}(f,f)=\langle\nabla f,D\nabla f\rangle_\mu
  =\frac{1}{2}\int|\sigma^\top\nabla f|^2\dd\mu\ge0,
  \label{eq:dirichlet_setup}
\end{align}
whose closure is denoted by $(\mathcal{E},D(\mathcal{E}))$
and corresponds to a symmetric 
Markov semigroup $(P_t^S)_{t\ge0}$
on $L^2(\mu)$~\cite{wu2000deviation, bakry2014analysis}.
Since the antisymmetric part $\mathcal{L}_A$ is 
skew-adjoint, it does not contribute to the quadratic form
and the Dirichlet form $\mathcal{E}(f,f)$ associated with
$\mathcal{L}_S$ is the same as the one associated with $\mathcal{L}$, i.e., 
$\mathcal{E}(f,f)= -\langle\mathcal{L}_S f,f\rangle_\mu=-\langle\mathcal{L}f,f\rangle_\mu$.
Moreover,  $\mathcal{E}(f,f)$ can be used in a variational definition of the \emph{spectral gap} 
of $-\mathcal{L}_S$, via, for $f\in D(\mathcal{E})$,
\begin{align}
  \lambda_\mathrm{gap}
  \equiv
  \inf_{f}
  \frac{\mathcal{E}(f,f)}{\mathrm{Var}_\mu(f)}
  \geq 0.
  \label{eq:gap_variational}
\end{align}
It therefore corresponds to 
the magnitude of the smallest non-zero eigenvalue of $-\mathcal{L}_S$ when the spectrum
has a discrete part and sets the $L^2(\mu)$ relaxation rate of the symmetrized semigroup by
$\|\e^{t\mathcal{L}_S}(f-\langle f\rangle_\mu)\|_{L^2(\mu)}
\leq\e^{-\lambda_\mathrm{gap}t}\|f-\langle f\rangle_\mu\|_{L^2(\mu)}$.
At the same time, it guarantees exponential decay of stationary
correlations at rate $\lambda_{\mathrm{gap}}$ also for the irreversible
dynamics, since the antisymmetric part $\mathcal{L}_A$ only accelerates
relaxation in 
$L^2(\mu)$~\cite{hwang1993accelerating,lelievre2013optimal,duncan2016variance,ReyBellet2016,ReyBellet2015,Duncan2017,Abdulle2019}.
Additionally, by Eq.~\eqref{eq:gap_variational},
$\lambda_\mathrm{gap}$ is the largest constant for which the spectral-gap inequality
$\lambda_\mathrm{gap}\mathrm{Var}_\mu(f)\le\mathcal{E}(f,f)$ holds for all
$f\in D(\mathcal{E})$. More generally, $\mu$ satisfies 
a \emph{Poincar\'e inequality}~\cite{bakry2014analysis} 
(see also \cite{Boissard2014,Menz2014,Schlichting2019} and references therein for recent developments)
with constant $C_P$ if, for $f\in D(\mathcal{E})$,
\begin{align}
  \mathrm{Var}_\mu(f)\le C_P\mathcal{E}(f,f),
  \label{eq:poincare}
\end{align}
which holds for some finite $C_P$ if and only if $\lambda_\mathrm{gap}>0$ and the
best constant is $C_P=\lambda_\mathrm{gap}^{-1}$~\cite{bakry2014analysis}. 
We remark that throughout $\lambda_\mathrm{gap}$ refers to the spectral gap of the symmetric part $\mathcal{L}_S$
and not to the full generator $\mathcal{L}$ which generally has a complex spectrum.

\begin{assumption}[Dynamics]
\label{ass:dynamics}
On the state space $\mathbb{R}^d$ (or a bounded subdomain), the 
diffusion~\eqref{eq:sde_setup} and its invariant probability measure $\mu$ satisfy
the following assumptions.
\begin{enumerate}
  \item We consider ergodic dynamics. In particular, $\mu$ is the unique invariant probability measure of
    $(X_t)_{t\ge0}$, i.e., any $f\in L^2(\mu)$ with $P_tf=f$ for all $t\geq 0$
    is constant $\mu$-almost everywhere 
    (see e.g., \cite{ethier2009markov,bakry2014analysis}). 
  \item The semigroup $(P_t)_{t\ge0}$ is strongly continuous on
    $L^2(\mu)$, the Dirichlet form~\eqref{eq:dirichlet_setup} is closable, $\mathcal{L}$
    and $\mathcal{L}^\dagger$ admit a common core
    $\mathcal{C}\subset D(\mathcal{L})\cap D(\mathcal{L}^\dagger)$ that contains the
    constant functions, and the stationary density $\ps$ is smooth and strictly
    positive.
  \item We consider additive noise, i.e., the noise matrix $\sigma$ is constant with
    $\sigma\sigma^\top>0$, such that $D=\tfrac12\sigma\sigma^\top$ is constant and uniformly elliptic.
  \item The symmetric part $\mathcal{L}_S$ satisfies a Poincar\'e inequality~\eqref{eq:poincare} with  optimal
    constant $C_\mathrm{P}=\lambda_\mathrm{gap}^{-1}$ where $\lambda_\mathrm{gap}$
    is the spectral gap of $\mathcal{L}_S$.
\end{enumerate}
\end{assumption}

\begin{remark}[Sufficient conditions for Poincar\'e inequality]
\label{rem:poincare_sufficient}
A Poincar\'e inequality holds under several standard conditions. 
For the isotropic case where $D=I$ and $\ps \propto\e^{-\Phi}$, the Bakry-\'Emery curvature condition
$\nabla^2\Phi\ge\kappa I$ with $\kappa>0$ yields $\lambda_\mathrm{gap}\ge\kappa$
\cite{bakry2014analysis,pavliotis2014stochastic}. 
More generally, the existence of a Lyapunov function $W$, such that
$\mathcal{L}W\leq -cW+C\mathbb{1}_K$ for some $W\geq 1$, constants $c,C>0$ and closed set $K$,
implies a Poincar\'e inequality, 
which is equivalent to geometric ergodicity 
and and for the reversible part a positive gap is equivalent to $L^2(\mu)$ exponential ergodicity
(see e.g., \cite{bakry2008rate,down1995exponential,meyn1993stability})
\end{remark}

\subsection{Density- and Stratonovich-type additive functionals}
\label{sec:setup_observables}
With the dynamics now in place, we turn to the dynamical observables
obtained by time-averaging a single path over the observation window $[0,t]$.
In particular, we study two distinct classes, the \emph{empirical density} and the
\emph{empirical current}.

\subsubsection{Density-type functional}
\begin{definition}[Density-type functional]
\label{def:empirical_density}
For a bounded measurable function $V\in L^\infty(\mathbb{R}^d)$ the
\emph{density -type functional} (or "empirical density") is defined as
\begin{align}
  \overline{\rho}_t(V)
  \equiv \frac{1}{t}\int_0^t V(X_s)\dd s,
  \label{eq:def_density}
\end{align}
with $\rho_t(V)=t\overline{\rho}_t(V)$ the corresponding time-accumulated observable.
\end{definition}
Notably, the integral in Eq.~\eqref{eq:def_density} is an ordinary, pathwise
Lebesgue integral and therefore has vanishing quadratic variation. The
statistics of such density-type functionals are a classical subject in the
literature. For Brownian motion, the distribution of
occupation times goes back to L\'evy's arcsine law~\cite{levy1940}, and
Kac's formula~\cite{kac1949distributions} reduced the computation of
their generating function to a Schr\"odinger-type spectral problem.
Occupation- and local-time functionals are by now standard objects in
probability~\cite{lamperti1958,knight1969,mansuy2008,yen2013local} and
in statistical physics~\cite{majumdar2007brownian}, where they appear in problems ranging from
spin glasses~\cite{majumdar2002exact} and non-equilibrium coarsening
systems~\cite{dornic1998} 
to anomalous diffusion~\cite{dhar1999,de2001}
and the blinking of quantum dots~\cite{margolin2005}. They also emerge
in the context of fluorescent imaging~\cite{agmon2011} and in time-average
statistical mechanics specifically, the local-time fraction serves as a
propagator for time-averaged observables~\cite{lapolla2018unfolding,lapolla2020spectral,lapolla2019}.
For reversible dynamics the large time asymptotics are governed by the general Donsker-Varadhan
large deviation principle~\cite{donsker1975asymptotic1, donsker1975asymptotic2,donsker1976asymptotic,donsker1983asymptotic,touchette2009large,touchette2018introduction,hoppenau2016level},
with explicit large deviation results obtained for simple 
examples~\cite{nyawo2018dynamical,angeletti2016diffusions,du2020dynamical},
while at finite times the variance, and correlations of bounded additive
functionals can be obtained from the eigenspectrum of the (potentially non-self-adjoint) tilted
generator~\cite{lapolla2020spectral}. Beyond occupation and local
times, density-type functionals have been applied---depending on the
choice of $V$---to diverse phenomena such as fluctuating
interfaces~\cite{majumdar2004exact}, financial options in the
Black-Scholes framework and Sinai-type random
chains~\cite{yor2001exponential,oshanin2013}, the statistics of
temperature records, and elastic
interfaces~\cite{majumdar2007brownian}, as well as models of turbulent
flow~\cite{baule2006}.

For a normalized differentiable window function $V_z^h(x)\geq0$ with coarse-graining scale
$h$ centered at $z\in\mathbb{R}^d$, the \emph{empirical density} yields a coarse-grained
local density of the underlying process~\cite{dieball2022mathematical,dieball2022coarse}. 
More generally, $\overline{\rho}_t(V)$ is the time average of $V$ along the
trajectory and is therefore an estimator for the stationary average $\E^\mu[V]$
\cite{dieball2022coarse, dieball2022mathematical}.
Such ensemble-averages are often the target of
molecular dynamics simulations and Markov chain Monte Carlo methods,
where time averages along a trajectory serve as estimators of
steady-state averages of (otherwise computationally inaccessible) 
high-dimensional
distributions~\cite{pavliotis2014stochastic,stoltz2010free,lelievre2016partial,leimkuhler2016computation,leimkuhler2015molecular,duncan2016variance}.

\subsubsection{Stratonovich current-type functional}
\begin{definition}[Current functional]
\label{def:empirical_current}
For a bounded, continuously differentiable vector field with bounded
derivatives, $U\in C^1_b(\mathbb{R}^d;\mathbb{R}^d)$, the \emph{current-type functional} (or empirical
current) is the Stratonovich integral
\begin{align}
  \overline{J}_t(U)
  \equiv \frac{1}{t}\int_0^t U(X_s)\circ\dd X_s,
  \label{eq:def_current}
\end{align}
with $J_t(U)=t\overline{J}_t(U)$ the corresponding time-accumulated observable,
where $\circ$ denotes Stratonovich integration and $U(X_s)\circ\dd X_s=\sum_i U_i(X_s)\circ\dd X^i_s$. 
\end{definition}

\begin{remark}[Motivation for Stratonovich functionals]
\label{rem:strato_motivation}
The convention in Eq.~\eqref{eq:def_current}
as a Stratonovich-type integral
\emph{is fixed by the physics of the problem},
since a physical current nominally ought to be anti-symmetric under time-reversal, i.e., for a fixed path 
$(X_s)_{0\le s\le t}$ it holds by the ``mid-point'' rule that
\begin{align*}\overline{J}_t(U)[(X_s)_{0\le s\le t}]
=
-\overline{J}_t(U)[(X_{t-s})_{0\le s\le t}].
\end{align*}
For a normalized differentiable window function $V_z^h(x)$,
Stratonovich integrals further generalize the notion of a continuity equation for
additive functionals to individual sample paths. With $\partial_{z_i}V^h_z(x)=-\partial_{x_i}V^h_z(x)$ 
and $\dd V^h_z(X_s)=\nabla V^h_z(X_s)\circ \dd X_s$ 
due to the Stratonovich convention it holds
\begin{align*}
  -\sum_{i}^{}\partial_{z_i}\bigl[t\overline{J}_t(V^h_z\hat e_i)\bigr]
  &=\sum_{i}^{}\int_0^t(\partial_{x_i}V^h_z)(X_s) \hat e_i\circ \dd X_s
  = \int_0^t(\nabla_x V^h_z)(X_s)\circ \dd X_s
  \nonumber
  \\
  &=V^h_z(X_t) -  V^h_z(X_0)
   =\partial_t \int_0^{t}V^h_z(X_s) \dd s,
\end{align*}
and the path-wise
continuity equation,
$\partial_t\bigl[t\overline{\rho}_t(V^h_z)\bigr]
 =-\sum_{i=1}^{d}\partial_{z_i}\bigl[t\overline{J}_t(V^h_z\hat e_i)\bigr],$ 
holds.
It\^o integrals lack both of
these properties~\cite{seifert2012stochastic, seifert2025stochastic,dieball2022coarse, dieball2022mathematical}. 
Historically, the Stratonovich interpretation 
is typically preferred in physical applications
since it arises in the limiting procedure from colored noise to white noise in the Langevin equation via the Wong-Zakai theorem \cite{Wong1965,Arnold1990}
and the stochastic chain rule takes its classical form~\cite{gardiner1985handbook,pavliotis2014stochastic}.
\end{remark}

The value of an additive functional at time $t$ is fixed by the entire
path $(X_s)_{0\le s\le t}$, whose increments are correlated through the
stochastic dynamics. Both $\overline{\rho}_t(V)$ and $\overline{J}_t(U)$
are therefore random variables with intricate, generically non-Gaussian
fluctuations, strongly affected by the mixing properties of the
underlying process as well as by the observable $V$ or $U$. Notably, in
practical applications typically only $V(X_t)$ and $U(X_t)$ are
observable, whereas neither the underlying process $X_t$ nor the precise
functions $U,V$ are generally known.
We are particularly interested in vector fields that are \emph{not} gradients of a potential.

Stratonovich-type functionals are more recent and much less studied than their
density-type counterparts. However, for suitable choices 
of the vector field $U$ they realize the generalized currents central to non-equilibrium statistical mechanics and stochastic thermodynamics.
For example, the constant choice $U=e_i$ gives the mean displacement along direction $e_i$ over $[0,t]$,
$t^{-1}(X^i_t-X^i_0)$, e.g.~of a probe in a
single-molecule experiment~\cite{touchette2018introduction}. The choice
$U=f$ measures the rate of work done by an external force $f$ on $X_t$ over the same
interval~\cite{van2003stationary,douarche2006work,engel2009asymptotics}
(closely related to the dissipated heat~\cite{qian2001nonequilibrium,van2004extended}), 
$U=D^{-1}\vs$ yields the steady-state entropy production rate~\cite{seifert2012stochastic,bo2019functionals}.
Moreover, given the normalized differentiable window function
$V_z^h\geq0$ centered at $z$ with $\int_{\mathbb{R}^d} V^h_z dx=1$
from above, the choice $U=V^h_z \hat n$ (with $\hat n \in \mathbb{R}^d$)
yields 
a coarse-grained estimator $\overline{J}_t(V_z^h \hat n)$ of the local steady-state current $\js$
along $\hat n$ at $z$ with scale $h$ \cite{dieball2022mathematical,dieball2022coarse}.

The asymptotic large deviation regime is captured by the intermediate
level 2.5 rate function for the joint empirical density and
current~\cite{maes2008steady,barato2015formal,chetrite2015nonequilibrium,hoppenau2016level}
and by macroscopic fluctuation
theory~\cite{bertini2015flows,bertini2002macroscopic}. A large-deviation
description for the empirical current alone, obtained from the joint rate function via the
contraction principle~\cite{touchette2009large}, is notoriously
involved and explicit expressions derived via the Gärtner-Ellis
theorem~\cite{dembo1998,den2008,touchette2009large,touchette2018}, 
require solving a non-trivial spectral problem
for a non-self-adjoint tilted operator (see Section~\ref{sec:tilted_generators}). This approach remains
limited to a very small number of 
simple models often based on low-dimensional stochastic differential equations~\cite{mehl2008large,fischer2018large,gupta2017,
verley2014unlikely,tsobgni2016large,du2023dynamical,angeletti2016diffusions}.
Therefore, one typically resorts
to either a numerical solution of this 
spectral problem or to dedicated simulation methods such as cloning or
importance sampling~\cite{giardina2006direct,lecomte2007numerical,ferre2018adaptive,touchetteharris,angeli2019rare}
and machine learning 
tools~\cite{oakes2020deep,rose2021reinforcement,das2021reinforcement}.
Moreover, since large-deviation
theory concerns asymptotic timescales, estimating the rate function over
its full range from finite experimental data proves to be a daunting task~\cite{ciliberto2004,martinez2016,rohwer2015convergence}.

At finite times, exact results for the variance and correlations of
current fluctuations are available via stochastic
calculus~\cite{dieball2022mathematical,dieball2022coarse},
and the thermodynamic uncertainty relation (TUR) and its
refinements~\cite{barato2015,gingrich2016dissipation,pietzonka2017,horowitz2020thermodynamic,koyuk2020thermodynamic,dieball2023direct,stutzer2026stochastic}
impose a fundamental constraint on the variance, with important applications
in the context of thermodynamic inference~\cite{Seifert2019,dieball2023direct,horowitz2020thermodynamic}.
Notably, these applications hinge of the feasibility to adequately sample the second moment (or higher order 
moments~\cite{dechant2020,Dechant2021,Wampler2021,Manikandan2022,Ray2023}) of the current functional, which is a priori very challenging to control. Moreover, inference of free energy differences
from measurements of the non-equilibrium work (a current-type functional) based 
on the Jarzynski~\cite{jarzynski1997nonequilibrium} 
or Crooks relation~\cite{crooks1999} 
require sufficient sampling of the tails of the work distribution~\cite{collin2005verification,alemany2015free,ribezzi2014free,Pohorille2010,Hummer2001,Zuckerman2002,Jarzynski2006,Gore2003,engel2009asymptotics}. 

Required but missing from this picture is a non-asymptotic control of the
\emph{full tail} of $\overline{J}_t(U)$ at finite $t$, beyond the
variance information of the thermodynamic uncertainty relation (TUR)
and its inverse (iTUR).
Moreover, this raises the physically intriguing complementary question whether and how the physical properties of the process $X_t$ 
---in particular the free-energy dissipation---bound the tail of the distribution from above,
complementary to the way the TUR limits the variance of currents from below
or the iTUR from above.

\begin{remark}[Accessibility from time-series]
Time-integrated observables have the practical benefit that there are directly
accessible from measured time-series in a ``model-free'' manner without any knowledge of the underlying 
dynamics.
For a measured trajectory $(X_{i\Delta t})_{i=0,1,\ldots,N}$, sampled at (sufficient) time-resolution
$\Delta t$, the integrals characterizing $\overline{\rho}_t$ and $\overline{J}_t$ are
practically evaluated via Riemann sums where Stratonovich-type functionals use
increments with mid-point convention.
In particular, varying the choice of $V$ and $U$ (i.e., probing different physical observables)
does not require a new measurement and different (physical) properties can be probed
from the same time-series
(for further details see e.g., \cite{dieball2022coarse,dieball2022mathematical,dieball2024non}).
\end{remark}

\subsubsection{Stationary expectations of additive functionals}
\label{sec:functional_means}
Since our concentration inequalities will assess finite-time deviations
of $\overline{\rho}_t$ and $\overline{J}_t$ \emph{around} 
their respective stationary expectations, it is useful to state them here once
(see also~\cite{dieball2022coarse, dieball2022mathematical}).

\begin{proposition}[Stationary expectation values]\label{prop:mean}
Let Assumptions~\ref{ass:dynamics} hold and consider
additive functionals of density-type $\overline{\rho}_t$ (Definition~\ref{def:empirical_density})
and current-type $\overline{J}_t$ (Definition~\ref{def:empirical_current}).
For every $t>0$ the stationary (time-independent) means
are given by  
\begin{align}
\E^\mu\bigl[\overline{\rho}_t(V)\bigr] &= \int_{\mathbb{R}^d} V \dd \mu = \E^\mu[V],
\\
\E^\mu\bigl[\overline{J}_t(U)\bigr] &= \int_{\mathbb{R}^d} \vs \cdot U \dd\mu = \E^\mu[\vs \cdot U ] =\E^\mu[W],
\end{align}
with local mean velocity $\vs$ from Eq.~\eqref{eq:drift_split}
and $W\equiv \vs\cdot U$.
\end{proposition}

\begin{proof}
First note that at stationarity $X_s\sim\mu$ for all $s$.
For the density functional we directly have,
\begin{align}
  \E^\mu\bigl[\overline{\rho}_t(V)\bigr]
  = \frac{1}{t}\int_0^t \E^\mu[V(X_s)] \dd s
  = \E^\mu[V(X)]
  = \int_{\mathbb{R}^d} V\dd\mu 
  = \E^\mu[V].
\end{align}
For the current functional, we first account for the Stratonovich correction via
\begin{align}
  U(X_s)\circ\dd X_s = U(X_s)\cdot\dd X_s + [\nabla\cdot(DU)](X_s) \dd s ,
\end{align}
where we consider additive noise, i.e.~constant $D$, and the correction 
term is understood as $\sum_{ij}D_{ij}\partial_i U_j = \nabla\cdot (DU)$ by the symmetry
of $D$. Inserting the stochastic differential equation~\eqref{eq:sde_setup},
the It\^o integral against the Wiener increment has zero mean and we obtain
\begin{align}
  \E^\mu\bigl[\overline{J}_t(U)\bigr]
  = \int_{\mathbb{R}^d} U\cdot b \dd\mu 
  + 
  \int_{\mathbb{R}^d}\nabla\cdot (DU)\dd\mu.
\end{align}
For the second term we use the product rule for divergences, i.e., 
$\nabla \cdot (DU\ps)= \ps\nabla\cdot (DU) + DU\cdot \nabla\ps$
and integrate, with $\dd\mu = \ps\dd x$, to obtain
\begin{align}
0=\int_{\mathbb{R}^d} \nabla \cdot (DU\ps) \dd x
= \int_{\mathbb{R}^d} \nabla\cdot (DU) \ps \dd x  + \int_{\mathbb{R}^d} DU\cdot\nabla\ps \dd x.
\end{align}
Re-arranging, using that $D$ is symmetric, and $\nabla\ps /\ps = \nabla\log\ps$ gives
\begin{align}
  \int_{\mathbb{R}^d} \nabla\cdot (DU) \ps \dd x
  = -\int_{\mathbb{R}^d} DU\cdot\nabla\ps \dd x
  = -\int_{\mathbb{R}^d} U\cdot D\nabla\log\ps \dd\mu.
\end{align}
With the drift decomposition in Eq.~\eqref{eq:drift_split} we therefore have
\begin{align}
  \E^\mu\bigl[\overline{J}_t(U)\bigr]
  = \int_{\mathbb{R}^d} U\cdot(b - D\nabla\log\ps)\dd\mu
  = \int_{\mathbb{R}^d} U\cdot\vs \dd\mu = \E^\mu[\vs\cdot U] =\E^\mu[W],
\end{align}
which completes the proof.
\end{proof}

\section{Feynman-Kac semigroups, tilted generators, and a first Cram\'er-Chernoff bound}
\label{sec:feynman_kac}
Our strategy for proving concentration inequalities
for density- $\overline{\rho}_t(V)$
and current-functionals $\overline{J}_t(U)$
rests on a dynamic version of the Cram\'er-Chernoff method,
which bounds the tail
probability of a random variable by controlling the moment generating function or an appropriate upper bound on it (see e.g.~\cite{boucheron2013concentration})
For an additive path-observable $A_t$ (density- or current-type), 
applying Chernoff's inequality to the accumulated
functional $A_t = t \overline{A}_t$ gives
\begin{align}
  \mathbb{P}^\nu\left(\overline{A}_t - \E^\mu[\overline{A}_t] \geq a\right)
  &= \mathbb{P}^\nu\left(A_t \geq t\bigl(\E^\mu[\overline{A}_t] + a\bigr)\right)
  \nonumber\\
  &\leq \inf_{k>0} \e^{-k t(\E^\mu[\overline{A}_t]+a)}\E^\nu\bigl[\e^{k A_t}\bigr],
  \label{eq:chernoff_functional}
\end{align}
where $\E^\mu[\overline{A}_t]$ is the stationary expectation with respect to the invariant
measure $\mu$ (independent of $t$), $a>0$ the
deviation, and the infimum over the (tilt) parameter $k$ gives the tightest bound for a given $t$
and $a$. 
Hereby, $\E^\nu\bigl[\e^{k A_t}\bigr]$ is the moment generating function
of $A_t$ for the Markov process started from initial measure $\nu$, i.e., $X_0 \sim \nu$. 
Note that by design we  center $\overline{A}_t$ by its stationary mean and \emph{not} with respect to the initial measure, 
which will result in a prefactor later on that accounts for the discrepancy
between $\nu$ and $\mu$.

The central object that links the moment generating function to the
operator-theoretic framework 
of Markov processes~\cite{bakry2014analysis} is the
\emph{Feynman-Kac semigroup}, a family of linear operators on $L^2(\mu)$
whose action encodes exponential functionals (i.e., the moment generating function) of the process.
In the remainder of this section we define the Feynman-Kac
semigroups for both types of observables, establish their basic properties, and derive
their infinitesimal generators---the so-called \emph{tilted generators}---which describe
the weighted (i.e., tilted) dynamics at the infinitesimal level.
Since the tilted generators are in general not self-adjoint (in the case of Stratonovich-type functionals \emph{even if the underlying dynamics obeys detailed balance}), we later symmetrize the
tilted generators in $L^2(\mu)$ and introduce the corresponding principal eigenvalue
$\Lambda^A(k)$. This eigenvalue gives an upper bound on the operator norm of the
Feynman-Kac semigroup, which is the key spectral input to all bounds that follow.

\subsection{Feynman-Kac semigroups}
\label{sec:fk_semigroup}
We begin with the classical Feynman-Kac semigroup associated with density-type observables,
which can be viewed as a weighted version 
of the standard Markov semigroup $(P_t)_{t \geq 0}$~\cite{bakry2014analysis}.

\begin{definition}[Feynman-Kac semigroup for density observables]
\label{def:fk_density}
Let $V : \mathbb{R}^d \to \mathbb{R}$ be a bounded measurable function
and $k \in \mathbb{R}$.
The \emph{Feynman-Kac semigroup} associated with $V$
is the family of operators $(P_t^{k,\rho})_{t \geq 0}$
on $L^2(\mu)$ defined by
\begin{align}
  \bigl(P_t^{k,\rho} f\bigr)(x)
  \equiv
  \E^x\left[
    f(X_t)
    \exp\left(k \int_0^t V(X_s)\dd s\right)
  \right],
  \label{eq:fk_density}
\end{align}
for $f \in L^2(\mu)$, $x \in \mathbb{R}^d$, $t \geq 0$.
\end{definition}

Although Kac originally considered only density-type observables \cite{kac1951}, an 
analogous semigroup is obtained by considering instead a Stratonovich integral in the exponent~\cite{chetrite2013nonequilibrium,chetrite2015nonequilibrium,touchette2018,dieball2023feynman}. 
The corresponding infinitesimal generator turns out to have a
qualitatively different mathematical structure due to the stochastic
nature of the integral, as we discuss in more detail in
Section~\ref{sec:tilted_current}.

\begin{definition}[Feynman-Kac semigroup for current observables]
\label{def:fk_current}
Let $U\in C^1_b(\mathbb{R}^d;\mathbb{R}^d)$ be a bounded, continuously
differentiable vector field with bounded derivatives, and $k\in\mathbb{R}$.
The \emph{Feynman-Kac semigroup} associated with $U$ is the family
$(P_t^{k,J})_{t\geq 0}$ on $L^2(\mu)$ defined by
\begin{align}
  \bigl(P_t^{k,J} f\bigr)(x)
  \equiv
  \E^x\left[
    f(X_t)\exp\left(k \int_0^t U(X_s)\circ\dd X_s\right)
  \right],
  \label{eq:fk_current}
\end{align}
for $f \in L^2(\mu)$, $x \in \mathbb{R}^d$, $t \geq 0$, where $\circ$ denotes
Stratonovich integration and $U(X_\tau)\circ\dd X_\tau=\sum_i U_i(X_\tau)\circ\dd X^i_\tau$.
\end{definition}
\subsubsection{Connection to the moment generating function and a first bound}
\label{subsec:fk_mgf}
Having introduced both Feynman-Kac semigroups, we now
make their connection to the moment generating function explicit. Let
$\nu\ll\mu$ with $\dd\nu/\dd\mu\in L^2(\mu)$ be an initial probability measure that is absolutely continuous with respect to $\mu$. 
Then, the moment generating function of the accumulated observable $A_t$ admits the
representation
\begin{align}
  \E^\nu\bigl[\e^{kA_t}\bigr]
  &= \int_{\mathbb{R}^d} \E^x\bigl[\e^{kA_t}\bigr]\dd\nu(x) \nonumber\\
  &= \int_{\mathbb{R}^d} \bigl(P_t^{k,A}\mathbf{1}\bigr)(x)\frac{\dd\nu}{\dd\mu}(x)\dd\mu(x)
     \nonumber\\
  &= \biggl\langle \frac{\dd\nu}{\dd\mu}, P_t^{k,A}\mathbf{1}
     \biggr\rangle_{L^2(\mu)},
  \label{eq:mgf_fk}
\end{align}
where $A\in\{\rho,J\}$, the second equality uses
$(P_t^{k,A}\mathbf 1)(x)=\E^x[\e^{kA_t}]$ together with the
Radon-Nikodym derivative $\dd\nu=\tfrac{\dd\nu}{\dd\mu}\dd\mu$, and $\mathbf 1$
denotes the constant function equal to one.

Applying the Cauchy-Schwarz inequality and using $\|\mathbf 1\|_{L^2(\mu)}=1$
(since $\mu$ is a probability measure) directly yields a first bound,
\begin{align}
  \E^\nu\bigl[\e^{kA_t}\bigr]
  \leq
  \Bigl\| \frac{\dd\nu}{\dd\mu} \Bigr\|_{L^2(\mu)}\bigl\| P_t^{k,A} \bigr\|_{L^2(\mu)},
  \label{eq:mgf_cs}
\end{align}
where $\|P_t^{k,A}\|_{L^2(\mu)}$ denotes the operator norm on $L^2(\mu)$ and
$\|\dd\nu/\dd\mu\|_{L^2(\mu)}$ encodes the dependence on the initial
distribution. Importantly, this gives a \emph{finite-time prefactor} (i.e., a initial condition correction)
that distinguishes 
the non-asymptotic concentration bound from a corresponding asymptotic 
approach as done, e.g.\ in ~large deviation theory
where the correction due to the initial condition gets lost as $t\to\infty$
(see Remark~\ref{rem:ldt_sharpness}).
In particular, Eq.~\eqref{eq:mgf_cs} reduces the problem of bounding the moment
generating function of additive path observables, and hence, their tail
probabilities via Chernoff's inequality~\eqref{eq:chernoff_functional},
to a single quantity in terms of the operator norm $\|P_t^{k,A}\|_{L^2(\mu)}$.
Establishing a bound on this norm is therefore the next central technical task, and it
is what the remainder of this section builds toward.

\subsubsection{Comparison with Markov semigroups}
The Feynman-Kac semigroups $(P_t^{k,\rho})_{t\geq 0}$ and $(P_t^{k,J})_{t\geq 0}$
have the same formal semigroup structure 
of the Markov semigroup $(P_t)_{t\geq 0}$.
However, they differ since Markov semigroups are \emph{contractions} 
(see e.g., \cite{bakry2014analysis, ethier2009markov}), i.e., for $f\in L^2(\mu)$, $t\geq 0$,
\begin{align}
  \|P_t f\|_{L^2(\mu)} \leq \|f\|_{L^2(\mu)},
\end{align}
which follows from positivity, the conservation property 
$P_t\mathbf 1=\mathbf 1$, and invariance of $\mu$
(i.e., $\int P_t g\dd\mu=\int g\dd\mu$).
Indeed, the exponential tilting destroys the conservation property,
when is $A_t$ non-trivial, $(P_t^{k,A}\mathbf 1)(x)=
\E^x[\e^{kA_t}]\neq 1$, and the constant function $\mathbf 1$ 
is not conserved.
By evaluating the operator 
norm with $\mathbf 1$ and using
$\|\cdot\|_{L^1(\mu)}\leq\|\cdot\|_{L^2(\mu)}$ we 
obtain $\|P_t^{k,A}\|_{L^2(\mu)} \geq \E^\mu[\e^{kA_t}]$, i.e., 
the operator norm bounded by the stationary moment generating function
from below, and a contraction bound such as
$\|P_t^{k,A}\|_{L^2(\mu)}\leq 1$ is in general lost and can even grow in $t$.
The missing contraction property is therefore a crucial 
challenge in bounding the operator norm, and motivates a more detailed study of the tilted generator, in particular, its spectral properties in 
the next section, which will allow us to control the growth rate of the norm.

\subsection{Tilted generators}
\label{sec:tilted_generators}
Each Feynman-Kac semigroup is strongly
continuous on $L^2(\mu)$ and therefore
admits an infinitesimal generator (see e.g., \cite{bakry2014analysis, pavliotis2014stochastic, ethier2009markov}).
Hereby, the parameter $k$ that appears in 
Definitions~\ref{def:empirical_density} and~\ref{def:empirical_current}
exponentially \emph{tilts} the path measure by $\e^{kA_t}$, which re-weights 
trajectories according to the value of the observable.
The corresponding generators $\mathcal{L}_k^A$ of the resulting semigroups, with $A\in\{\rho,J\}$,
are accordingly called \emph{tilted generators}. 
Notably, at $k=0$ we have no tilting and
$\mathcal{L}_0^A=\mathcal{L}$, such that $\mathcal{L}_k^A$ can be considered deformation
of the Markov generator that, for $k\neq0$, differs from it by terms encoding the respective path observable.
Just as $\mathcal{L}$ governs the
dynamics of the Markov semigroup via the
backward equation $\partial_t P_t f = \mathcal{L} P_t f$ \cite{bakry2014analysis, pavliotis2014stochastic, ethier2009markov},
the tilted generator $\mathcal{L}_k^A$ 
governs the evolution of the Feynman-Kac semigroup at the infinitesimal level via 
\begin{align}
  \partial_t P_t^{k,A} f
  = \mathcal{L}_k^A P_t^{k,A} f,
  \qquad f \in D(\mathcal{L}_k^A).
  \label{eq:backward_fk}
\end{align}

\subsubsection{Tilted generator for density observables}
\label{sec:tilted_density}
We start our discussion on tilted generators for density observables.
The derivation follows the textbook Feynman-Kac formula, see e.g.~\cite{pavliotis2014stochastic,bakry2014analysis,kac1949distributions,dieball2023feynman}, and the tilt enters $\mathcal{L}$ as a simple additive potential.

\begin{proposition}[Tilted generator for density observables]
\label{prop:tilted_gen_density}
The semigroup $(P_t^{k,\rho})_{t \geq 0}$
is a strongly continuous semigroup on $L^2(\mu)$
with infinitesimal generator
\begin{align}
  \mathcal{L}_k^\rho \equiv  b\cdot\nabla + \nabla\cdot D\nabla + kV = \mathcal{L} + kV,
  \label{eq:tilted_gen_density}
\end{align}
where $\mathcal{L}$ is the generator of the It\^o diffusion~\eqref{eq:sde_setup}
and $kV$ denotes the multiplication operator
$f\mapsto kV \cdot f$ and $D(\mathcal{L}_k^\rho) = D(\mathcal{L})$.
\end{proposition}

The simplicity of this result follows as a direct consequence since
$\rho_t(V)=\int_0^t V(X_s)\dd s$ is an ordinary Lebesgue rather than a
stochastic integral.

\subsubsection{Tilted generator for current observables}
\label{sec:tilted_current}
The current functional, unlike the density
case, is typically not treated in the classical Feynman-Kac theory.
Here the exponent of $P_t^{k,J}$ contains the Stratonovich-type integral
$\int_0^t U(X_s)\circ\dd X_s$, whose non-trivial quadratic variation carries
over to the resulting generator and introduces new terms that act on the gradient structure of $\mathcal{L}$ in Eq.~\eqref{eq:generator_setup}.
The tilted generator for Stratonovich-type functionals can
be proven directly using It\^o's formula~\cite{dieball2023feynman},
the Cameron-Martin-Girsanov theorem~\cite{barato2015formal,chetrite2015nonequilibrium,chetrite2013nonequilibrium} or via functional calculus~\cite{dieball2023feynman}.

\begin{proposition}[Tilted generator for current observables]
\label{prop:tilted_gen_current}
Let $k\in\mathbb{R}$ and let Assumptions~\ref{ass:dynamics} hold, in particular, 
assume additive noise $D=\sigma\sigma^\top/2$. The semigroup
$(P_t^{k,J})_{t \geq 0}$ is a strongly continuous semigroup on $L^2(\mu)$ whose
infinitesimal generator acts on $f \in D(\mathcal{L})$ as
\begin{align}
  \mathcal{L}_k^J
  &=
  b \cdot (\nabla + kU)
  +
  (\nabla + kU) \cdot D(\nabla + kU),
  \label{eq:tilted_gen_current_short}
\end{align}
or equivalently, expanding for $f \in D(\mathcal{L})$ yields
\begin{align}
  \mathcal{L}_k^J f
  &=
  \mathcal{L}f
  +
  k\bigl(b \cdot U + \nabla \cdot DU\bigr) f + 2kU \cdot D \nabla f
  +
  k^2(U \cdot DU)f.
  \label{eq:tilted_gen_current_long}
\end{align}
\end{proposition}

\begin{remark}[Structural comparison of tilted generators]
\label{rem:structural_comparison}
Propositions~\ref{prop:tilted_gen_density} and~\ref{prop:tilted_gen_current}
reveal two fundamental structural differences between the density and current
tilted generators.
For the density functional, i.e.~$\mathcal{L}_k^\rho=\mathcal{L}+kV$, the tilt appears as an
additive perturbation that adds the multiplication operator $kV$ to the Markov
generator while leaving its differential structure unchanged. 
The current tilt,
however, involves a shift in the gradient $\nabla\to\nabla+kU$ in both
the drift and diffusion parts, and therefore deforms the full second-order
operator, in particular by introducing the first-order term $2k U\cdot D\nabla$, rather than merely adding a potential.
Moreover, $\mathcal{L}_k^J$ carries a non-negative term
$k^2(U\cdot DU)$ that is absent in the tilted generator for density-type functionals. These differences have important consequences for the approach to proving the corresponding concentration inequalities.
\end{remark}

\subsubsection{Symmetrization of the tilted generator}
\label{sec:symmetrization_tilted}
The tilted generators $\mathcal{L}_k^\rho$ and $\mathcal{L}_k^J$ are generally 
not self-adjoint in $L^2(\mu)$. The density generator $\mathcal{L}_k^\rho=\mathcal{L}+kV$
is non-self-adjoint whenever $\mathcal{L}$ is not, i.e.~whenever the underlying Markov
process is non-reversible ($\js\neq 0$, see~\cite{pavliotis2014stochastic,jiang2004mathematical}). For current observables,
$\mathcal{L}_k^J$ is non-self-adjoint for every $k\neq 0$ even under detailed
balance ($\js=0$), owing to the first-order term $2kU\cdot D\nabla$ introduced by
the stochastic integral (see Remark~\ref{rem:structural_comparison}).
Although non-self-adjointness is a priori not an issue for the construction of the
Feynman-Kac semigroups, it does obstruct the spectral analysis by which we bound
their operator norm $\|P_t^{k,A}\|_{L^2(\mu)}$. This analysis involves the
quadratic form $\langle \mathcal{L}_k^A f, f\rangle_{L^2(\mu)}$
(Section~\ref{sec:norm_bound}), which, for the non-self-adjoint $\mathcal{L}_k^A$,
does not reflect the spectrum of $\mathcal{L}_k^A$. We therefore pass to the
symmetrization $\widetilde{\mathcal{L}}_k^A$, a self-adjoint operator with the same
quadratic form as $\mathcal{L}_k^A$, now governed by its eigenvalue spectrum. A
short lemma below makes this identity precise and justifies working with
$\widetilde{\mathcal{L}}_k^A$ in place of $\mathcal{L}_k^A$.

\begin{definition}[Symmetrized tilted generators]
\label{def:sym_tilted}
For $A \in \{\rho, J\}$,
let $(\mathcal{L}_k^A)^\dagger$ denote
the $L^2(\mu)$ adjoint of $\mathcal{L}_k^A$.
The \emph{symmetrized tilted generator} is
\begin{align}
  \widetilde{\mathcal{L}}_k^A
  \equiv
  \frac{1}{2}\bigl[\mathcal{L}_k^A + (\mathcal{L}_k^A)^\dagger\bigr],
  \label{eq:sym_tilted_def}
\end{align}
which is symmetric, and self-adjoint in $L^2(\mu)$ (on the common core, see Assumptions~\ref{ass:dynamics})
by construction.
\end{definition}

\begin{lemma}[Quadratic form identity]
\label{lem:quad_identity}
Let $k \in \mathbb{R}$ and let $f \in D(\mathcal{L}_k^A)\cap D((\mathcal{L}_k^A)^\dagger)$
be real-valued. Then
\begin{align}
  \bigl\langle \mathcal{L}_k^A f, f \bigr\rangle_{L^2(\mu)}
  =
  \bigl\langle \widetilde{\mathcal{L}}_k^A f, f \bigr\rangle_{L^2(\mu)},
  \label{eq:quad_identity}
\end{align}
with $\widetilde{\mathcal{L}}_k^A$ as in Definition~\ref{def:sym_tilted}, i.e.,
the quadratic form of $\mathcal{L}_k^A$ is determined entirely by its
self-adjoint part.
\end{lemma}

\begin{proof}
Decompose $\mathcal{L}_k^A
= \widetilde{\mathcal{L}}_k^A
+ \widehat{\mathcal{L}}_k^A$,
where $\widehat{\mathcal{L}}_k^A
= \tfrac{1}{2}(\mathcal{L}_k^A
- (\mathcal{L}_k^A)^\dagger)$
is anti-symmetric in $L^2(\mu)$.
For real-valued $f$ in a real Hilbert space $L^2(\mu;\mathbb{R})$,
anti-symmetry gives
$\langle\widehat{\mathcal{L}}_k^A f,
f\rangle_{L^2(\mu)}= -\langle\widehat{\mathcal{L}}_k^A f,f\rangle_{L^2(\mu)}$,
hence $\langle\widehat{\mathcal{L}}_k^A f,f\rangle_{L^2(\mu)} = 0$,
and Eq.~\eqref{eq:quad_identity} follows.
\end{proof}

Lemma~\ref{lem:quad_identity} shows that,
although $\mathcal{L}_k^A$ is the natural 
generator when initially defining the Feynman-Kac semigroup $P_t^{k,A}=\e^{t\mathcal{L}_k^A}$,
the symmetrization preserves the tilted quadratic form and it is therefore natural 
to work directly with $\widetilde{\mathcal{L}}_k^A$ instead, which we now compute explicitly
for the setting of diffusion processes.\footnote{We remark that Lemma~\ref{lem:quad_identity} only gives the diagonal
identity $\bigl\langle \mathcal{L}_k^A f, f \bigr\rangle_{L^2(\mu)}= \bigl\langle \widetilde{\mathcal{L}}_k^A f, f \bigr\rangle_{L^2(\mu)}$, and the full bilinear form ($f\neq g$) is not preserved, since the antisymmetric part contributes 
to the off-diagonal.}
In particular we use the drift decomposition established earlier in Eq.~\eqref{eq:drift_split}.

\begin{proposition}[Symmetric-antisymmetric decomposition of tilted generators]
\label{prop:sym_tilted_explicit}
Let $\mathcal{L}_S = \ps^{-1}\nabla\cdot(\ps D\nabla)$ and $\mathcal{L}_A = \vs\cdot\nabla$
be the symmetric and antisymmetric parts of the Markov generator
from Eq.~\eqref{eq:markov_gen_decomposition}.
For $f \in
D(\mathcal{L})$, the tilted generator for \emph{density observables} has symmetric and
antisymmetric parts, 
\begin{align}
  \widetilde{\mathcal{L}}_k^\rho f= \mathcal{L}_S f + kV f,
  \qquad
  \widehat{\mathcal{L}}_k^\rho f= \mathcal{L}_A f,
  \label{eq:sym_tilted_density}
\end{align}
and the tilted generator for \emph{current observables} decomposes as
\begin{align}
  \widetilde{\mathcal{L}}_k^J f
  & = \mathcal{L}_S f + k(\vs\cdot U) f + k^2(U\cdot DU)f\nonumber\\
  \widehat{\mathcal{L}}_k^J f
  &= \mathcal{L}_A f + 2k(DU)\cdot\nabla f + k\ps^{-1}\nabla\cdot(\ps DU) f,  
  \label{eq:sym_tilted_current}
\end{align}
where $\widetilde{\mathcal{L}}_k^A = \tfrac12\bigl[\mathcal{L}_k^A + (\mathcal{L}_k^A)^\dagger\bigr]$
and $\widehat{\mathcal{L}}_k^A = \tfrac12\bigl[\mathcal{L}_k^A - (\mathcal{L}_k^A)^\dagger\bigr]$
denote the symmetric and antisymmetric parts in $L^2(\mu)$ for $A \in \{\rho, J\}$
from Definition~\ref{def:sym_tilted}.
\end{proposition}

\begin{proof}
The symmetric-antisymmetric decomposition of the tilted generators
follows from Definition~\ref{def:sym_tilted} by computing the respective
$L^2(\mu)$ adjoints. 
First, we note that by using Eq.~\eqref{eq:markov_gen_decomposition} the adjoint of the Markov generator
is given by $\mathcal{L}^\dagger = \mathcal{L}_S - \mathcal{L}_A$,
and that terms appearing as multiplication by a real function are self-adjoint. 

For the density observable we have $\mathcal{L}_k^\rho = \mathcal{L} + kV$ where $kV$ is self-adjoint.
The $L^2(\mu)$ adjoint is thus obtained as $(\mathcal{L}_k^\rho)^\dagger
= \mathcal{L}_S - \mathcal{L}_A + kV$. 
Consequently, the symmetrized density tilted generator reads $\widetilde{\mathcal{L}}_k^\rho = \mathcal{L}_S + kV$,
and the antisymmetric part is given by $\widehat{\mathcal{L}}_k^\rho = \mathcal{L}_A$ and 
completes the proof of Eq.~\eqref{eq:sym_tilted_density}.

For the current observable the tilted generator, applied to a test function $f$, reads
\begin{align}
  \mathcal{L}_k^Jf
  = \mathcal{L}f + k(b\cdot U + \nabla\cdot DU)f + 2k(DU)\cdot\nabla f + k^2(U\cdot DU)f,
\end{align}
where we write $2kU\cdot D\nabla = 2k(DU)\cdot\nabla$ by the symmetry of $D$.
Here, the second and last terms appear as multiplication operators which are self-adjoint.
To treat the remaining first-order term we use, for a smooth
vector field $Y$, that integration by parts against $\dd\mu = \ps\dd x$ (boundary terms
vanishing) gives, for a test function $f$,
\begin{align}
  (Y\cdot\nabla)^\dagger f
  = -Y\cdot\nabla f - \ps^{-1}\nabla\cdot(\ps Y)f.
  \label{eq:first_order_adjoint}
\end{align}
Therefore, with $Y = 2k DU$ the adjoint now reads
\begin{align}
  (\mathcal{L}_k^J)^\dagger f
   = & \mathcal{L}_S f - \mathcal{L}_A f
    + k(b\cdot U + \nabla\cdot DU) f
    - 2k(DU)\cdot\nabla f
    \nonumber\\ &- 2k\ps^{-1}\nabla\cdot(\ps DU) f
    + k^2(U\cdot DU) f.
\end{align}
The antisymmetric tilted generator of the current observable is therefore given by
\begin{align}
  \widehat{\mathcal{L}}_k^J  f
  = \mathcal{L}_A f + 2k(DU)\cdot\nabla f + k\ps^{-1}\nabla\cdot(\ps DU) f,
\end{align}
and its symmetric counterpart then reads
\begin{align}
  \widetilde{\mathcal{L}}_k^J f
  = \mathcal{L}_S f + k(b\cdot U + \nabla\cdot DU) f
  - k\ps^{-1}\nabla\cdot(\ps DU)f
    + k^2(U\cdot DU) f.
\end{align}
Lastly, we expand $\ps^{-1}\nabla\cdot(\ps DU) = \nabla\cdot DU + (D\nabla\log\ps)\cdot U$
such that the $\nabla\cdot DU$ terms cancel, and using the drift decomposition~\eqref{eq:drift_split} via
$b - D\nabla\log\ps = \vs$ 
allows us to write the term linear in $k$ as $k(\vs\cdot U)$
which gives Eq.~\eqref{eq:sym_tilted_current} and completes the proof.
\end{proof}

We now give a Dirichlet representation of tilted quadratic forms that will prove useful below.
 
\begin{corollary}[Dirichlet representation of the tilted quadratic forms]
\label{cor:dirichlet_tilted}
Let $f\in D(\mathcal{L})$ be real-valued and let
$\mathcal{E}(f,f)=-\langle\mathcal{L}_S f,f\rangle_\mu=\langle\nabla f,D\nabla f\rangle_\mu$
be the symmetrized Dirichlet form of Eq.~\eqref{eq:dirichlet_setup}. Then the
quadratic forms of the tilted generators admit the representations
\begin{align}
  \langle \mathcal{L}_k^\rho f, f\rangle_\mu
  &= -\mathcal{E}(f,f) + k\langle V, f^2\rangle_\mu,
  \label{eq:dirichlet_tilted_density}
  \\[2pt]
  \langle \mathcal{L}_k^J f, f\rangle_\mu
  &= -\mathcal{E}(f,f) + k\langle \vs\cdot U, f^2\rangle_\mu
     + k^2\langle U\cdot DU, f^2\rangle_\mu,
  \label{eq:dirichlet_tilted_current}
\end{align}
where $\langle g, f^2\rangle_\mu = \int gf^2\dd\mu$.
\end{corollary}

\begin{proof}
By Lemma~\ref{lem:quad_identity}
$\langle\mathcal{L}_k^A f,f\rangle_\mu=\langle\widetilde{\mathcal{L}}_k^A f,f\rangle_\mu$.
Substituting the explicit symmetrized
generators~\eqref{eq:sym_tilted_density} and~\eqref{eq:sym_tilted_current}, using
$\langle\mathcal{L}_S f,f\rangle_\mu=-\mathcal{E}(f,f)$ gives the stated forms.
\end{proof}

\subsubsection{Self-adjointness of tilted generators}
\label{sec:selfadjoint_tilted}
After introducing the explicit symmetric-antisymmetric decomposition of the tilted generator $\mathcal{L}_k^A$
in Proposition~\ref{prop:sym_tilted_explicit}, we can now determine the conditions under which
the tilted generators become self-adjoint.
In particular, $\mathcal{L}_k^A$ is self-adjoint in $L^2(\mu)$ if and only if its
antisymmetric part vanishes, i.e.,  $\widehat{\mathcal{L}}_k^A = 0$.

\paragraph{Density observables}
The antisymmetric part $\widehat{\mathcal{L}}_k^\rho = \mathcal{L}_A = \vs\cdot\nabla$ is
independent of $k$, which means that $\mathcal{L}_k^\rho$ is self-adjoint for \emph{every} $k$ if and
only if $\vs = 0$, i.e., the dynamics is reversible. In the density case self-adjointness is thus inherited
from the underlying dynamics, and in the reversible case the symmetrization equals the
full tilted generator.

\paragraph{Current observables}
The antisymmetric part $\widehat{\mathcal{L}}_k^J f = (\vs + 2kDU)\cdot\nabla f +
k\ps^{-1}\nabla\cdot(\ps DU)f$ acting on a test function $f$
depends on $k$ and, by the argument used in Eq.~\eqref{eq:first_order_adjoint}, vanishes for all test functions
if and only if the vector field $\vs + 2kDU$ vanishes. For $k \neq 0$ this requires
\begin{align}
  \vs = -2kDU
  \quad\Longleftrightarrow\quad
  \js = -2k\ps DU
  \quad\Longleftrightarrow\quad
  U = -\frac{1}{2k}D^{-1}\vs.
  \label{eq:selfadjoint_condition}
\end{align}
For $k = 0$,  we find $\widehat{\mathcal{L}}_0^J = \vs \cdot \nabla = \mathcal{L}_A$
and $\mathcal{L}_0^J=\mathcal{L}$, i.e., self-adjointness is obtained if and only if
the reversibility condition $\vs = 0$ holds.
Notably, for reversible dynamics with $k \neq 0$,
Eq.~\eqref{eq:selfadjoint_condition} requires that $DU = 0$, hence $U = 0$ since the diffusion coefficient $D$ is positive
definite. 
Consequently, for any non-trivial $U \neq 0$ the current operator $\mathcal{L}_k^J$ is
\emph{not} self-adjoint even when the dynamics is reversible. 
In the irreversible case, the self-adjointness condition is
achieved by the special choice $U = -\tfrac{1}{2k}D^{-1}\vs$.\footnote{Interestingly, the condition for
self-adjointness of the tilted current generator is proportional to the choice
that gives the steady-state entropy production rate $U_\Sigma = D^{-1}\vs$, i.e., $U=-\tfrac{1}{2k}U_\Sigma$.} 
Substituting $\vs = -2kDU$ into Eq.~\eqref{eq:sym_tilted_current} gives
$\vs\cdot U = -2k(U\cdot DU)$ and the the symmetric operator therefore reads
$\widetilde{\mathcal{L}}_k^J = \mathcal{L}_S - k^2(U\cdot DU)$. For a
general $U$, however, $\mathcal{L}_k^J$ remains non self-adjoint.

\subsection{Bounding the Feynman-Kac norm}
\label{sec:norm_bound}
We now return to the task of bounding the operator norm $\|P_t^{k,A}\|_{L^2(\mu)}$
first identified in Eq.~\eqref{eq:mgf_cs}.
As established in Section~\ref{sec:fk_semigroup},
the Feynman-Kac semigroup is not a contraction on $L^2(\mu)$.
Our strategy to recover a controlled norm
bound is to identify the precise rate at
which this norm can grow, by exploiting
the dissipativity structure of the tilted
generator $\mathcal{L}_k^A$ and recovering a contraction semigroup in the sense of Lumer-Phillips.

\subsubsection{Bound on operator norm via restoring dissipativity}
We characterize a Hilbert space version of dissipativity for linear operators with the following definition (for further background see e.g., \cite{pazy1983semigroups}).
\begin{definition}[Dissipative operator]\label{def:diss}
Let $H$ be a real-valued Hilbert space with inner product $\langle\cdot,\cdot\rangle$.
The linear operator $T : D(T) \subset H \to H$ is called \emph{dissipative} if and only if,
for all $x\in D(T)$,
\begin{align}
\langle Tx,x\rangle_H \leq 0,
\label{eq:diss_recall}
\end{align}
\end{definition}

In our context we therefore seek the
smallest constant $C$ for which
$T\equiv\mathcal{L}_k^A-C$ becomes dissipative.
Indeed, it is straightforward to identify this quantity as the supremum of the
corresponding quadratic form of $\mathcal{L}_k^A$ which motivates the following definition.

\begin{definition}[Dissipativity constant $\Lambda^A(k)$]
\label{def:Lambda}
For $A \in \{\rho, J\}$ and $k \in \mathbb{R}$ define
\begin{align}
  \Lambda^A(k)
  \equiv
  \sup\left\{
    \bigl\langle 
    \mathcal{L}_k^A f, f \bigr\rangle_{\mu}
    \mathrel{\Big|} f \in D(\mathcal{L}_k^A),\|f\|_{L^2(\mu)} = 1
  \right\},
  \label{eq:Lambda_def_unsym}
\end{align}
where the supremum is over the real Hilbert space $L^2(\mu)$.
\end{definition}
Thus, by Definitions~\ref{def:diss} and~\ref{def:Lambda} we can now directly recover dissipativity
via the shifted operator $\mathcal{L}_k^A-\Lambda^A(k)$ by construction.
By the Lumer-Phillips theorem this operator therefore generates a contraction semigroup
and gives the following operator-norm bound.
\begin{proposition}[Feynman-Kac norm bound]
\label{prop:fk_norm_bound}
For $A \in \{\rho, J\}$, $k \in \mathbb{R}$,
and all $t \geq 0$,
\begin{align}
  \|P_t^{k,A}\|_{L^2(\mu)}
  \leq
  \e^{t\Lambda^A(k)}.
  \label{eq:fk_norm_bound}
\end{align}
\end{proposition}

\begin{proof}
By Definition~\ref{def:Lambda},
$\langle(\mathcal{L}_k^A-\Lambda^A(k))f,f\rangle_{L^2(\mu)}\le0$ for every
$f\in D(\mathcal{L}_k^A)$, i.e., the operator $\mathcal{L}_k^A-\Lambda^A(k)$ is dissipative on
$L^2(\mu)$.
As a dissipative generator, it now generates a contraction semigroup by the
Lumer-Phillips theorem~\cite{pazy1983semigroups,engel2006short}, such that 
$\|\e^{t(\mathcal{L}_k^A-\Lambda^A(k))}\|_{L^2(\mu)}\le1$. 
Lastly, removing the shift gives the operator-norm bound
\begin{align}
  \|P_t^{k,A}\|_{L^2(\mu)}
  =\e^{t\Lambda^A(k)}\bigl\|\e^{t(\mathcal{L}_k^A-\Lambda^A(k))}\bigr\|_{L^2(\mu)}
  \le\e^{t\Lambda^A(k)},
\end{align}
and completes the proof.
\end{proof}

\begin{corollary}[Bound on the moment generating function]
\label{cor:mgf_bound}
For $A\in\{\rho,J\}$, $k\in\mathbb{R}$, $t\ge0$, and any initial law $\nu\ll\mu$,
\begin{align}
  \E^\nu\bigl[\e^{kA_t}\bigr]
  \leq\left\|\frac{\dd\nu}{\dd\mu}\right\|_{L^2(\mu)}
  \e^{t\Lambda^A(k)}.
  \label{eq:mgf_final_bound}
\end{align}
\end{corollary}
\begin{proof}
Combine the Cauchy-Schwarz bound~\eqref{eq:mgf_cs} with the norm
bound~\eqref{eq:fk_norm_bound}.
\end{proof}

\subsubsection{Spectral interpretation}
To obtain concentration inequalities for additive functionals, it remains to evaluate 
$\Lambda^A(k)$ which---as a quadratic form of a non-self-adjoint titled generator---has 
no direct relation to the (in general complex) spectrum of $\mathcal{L}_k^A$.
However, the symmetrization of Section~\ref{sec:symmetrization_tilted} resolves this, since by the quadratic-form
identity, it equals the supremum over the self-adjoint $\widetilde{\mathcal{L}}_k^A$,
which admits the following spectral interpretation.

\begin{corollary}[Spectral interpretation of $\Lambda^A(k)$]
\label{cor:Lambda_spectral}
By Lemma~\ref{lem:quad_identity},
\begin{align}
  \Lambda^A(k)
  \equiv
  \sup\left\{
    \bigl\langle\widetilde{\mathcal{L}}_k^A f, f\bigr\rangle_{L^2(\mu)}
    |
    f \in D(\mathcal{L}_k^A), \|f\|_{L^2(\mu)} = 1
  \right\}.
  \label{eq:Lambda_via_sym}
\end{align}
Since $\widetilde{\mathcal{L}}_k^A$
is self-adjoint in $L^2(\mu)$ by construction
(see Def.~\ref{def:sym_tilted}),
the Rayleigh-Ritz theorem guarantees 
that the supremum of the quadratic form is precisely the principal (maximal) eigenvalue
of $\widetilde{\mathcal{L}}_k^A$, i.e., 
\begin{align}
  \Lambda^A(k)
  =
  \lambda_{\max}\bigl(\widetilde{\mathcal{L}}_k^A\bigr).
  \label{eq:Lambda_is_eigenvalue}
\end{align}
\end{corollary}

Specifically, the problem of bounding the tail probabilities of $A_t$ is thereby reduced 
to a spectral problem that involves computing---or further bounding 
(see Section~\ref{sec:conc_poincare})---the principal 
eigenvalue $\Lambda^A(k)$ of the symmetrized tilted generator. 
Moreover, the prefactor $\|\dd\nu/\dd\mu\|_{L^2(\mu)}$
encodes the dependence on the initial measure $\nu\ll\mu$.
With this in place we can now derive our first general concentration inequalities.

\begin{proposition}[Variational formula for $\Lambda^A(k)$]
\label{prop:Lambda_variational}
With the Dirichlet form $\mathcal{E}(f,f)=-\langle\mathcal{L}_S f,f\rangle_\mu$,
\begin{align}
  &\Lambda^\rho(k)
  =
  \sup\left\{k\langle V,f^2\rangle_\mu-\mathcal{E}(f,f)
    \mathrel{\Big|} f \in D(\mathcal{L}_k^\rho),\|f\|_{L^2(\mu)} = 1
  \right\},
  \label{eq:Lambda_rho_dirichlet}
  \\
  &\Lambda^J(k)
  =
  \sup\left\{k\langle \vs\cdot U,f^2\rangle_\mu
        + k^2\langle U\cdot DU,f^2\rangle_\mu-\mathcal{E}(f,f)
    \mathrel{\Big|} f \in D(\mathcal{L}_k^J),\|f\|_{L^2(\mu)} = 1
  \right\}.
  \label{eq:Lambda_J_dirichlet}
\end{align}
\end{proposition}
\begin{proof}
Combine the spectral form (Corollary~\ref{cor:Lambda_spectral}) with the Dirichlet
representation of the quadratic forms (Corollary~\ref{cor:dirichlet_tilted}).
\end{proof}

Using the variational form it is easy to verify that $\Lambda^A(k)$ 
with $A\in\{\rho, J \}$, is a convex function in $k$ since the $k^2$ term is non-negative.
Moreover, at $k=0$ it reduces to the supremum over $-\mathcal{E}(f,f)\leq 0$
which is the principal eigenvalue of the symmetric part
$\mathcal{L}_S$ of the dynamics, i.e., $\Lambda^A(0)=0$
and $f=\mathbf{1}$ the corresponding eigenfunction.
Consequently, taking the derivative with respect to $k$ at $k=0$ with $f=\mathbf{1}$, therefore
gives $\partial_k\Lambda^\rho(0)=\E^\mu[V]=\E^\mu[\overline{\rho}_t]$
and $\partial_k\Lambda^J(0)=\E^\mu[\vs\cdot U] = \E^\mu[\overline{J}_t]$, respectively, where we used Proposition~\ref{prop:mean}.

\subsubsection{Tightness for the principal eigenvalue}
\label{sec:manfred_eigenvalue_bound}
The bound on the Feynman-Kac semigroup norm in Section~\ref{sec:norm_bound} controls the
moment generating function for all times $t>0$ via 
$\Lambda^A(k)$, i.e., the principal eigenvalue of the
symmetrized tilted generator (Corollary~\ref{cor:Lambda_spectral}). The actual 
growth rate it bounds, however, is the principal eigenvalue 
$\lambda_{\max}^A(k) \equiv 
\lambda_{\max}(\mathcal{L}_k^A)$ of the
\emph{full} tilted generator, which corresponds to the 
true scaled cumulant generating function that governs the regime $t\to\infty$
(see e.g., \cite{touchette2018introduction,chetrite2015nonequilibrium} for further details).

While the symmetrized operator $\widetilde{\mathcal{L}}_k^A$ is self-adjoint such that
$\Lambda^A(k)$ corresponds to its maximal eigenvalue characterized by a variational principle,
the generator $\mathcal{L}_k^A$ is not self-adjoint in general
(see Section~\ref{sec:selfadjoint_tilted}) and its (full) spectral problem in general involves
distinct left and right eigenfunctions. 
Next, we show that $\lambda_{\max}^A(k) \leq\Lambda^A(k)$ and discuss conditions for which equality holds.

\begin{proposition}[Bound on principal eigenvalues]
\label{prop:tightness}
Consider $A\in\{\rho,J\}$ and $k\in\mathbb{R}$.
Let $\psi_k>0$ and $r_k>0$ denote the eigenfunctions corresponding to the principal eigenvalues $\Lambda^A(k)$ (Definition~\ref{def:Lambda})
and $\lambda_{\max}^A(k)$ of the symmetrized 
(Proposition~\ref{prop:sym_tilted_explicit})
and full tilted generator (Propositions~\ref{prop:tilted_gen_density} and~\ref{prop:tilted_gen_current})
respectively, i.e., 
$\widetilde{\mathcal{L}}_k^A\psi_k=\Lambda^A(k)\psi_k$
and $\mathcal{L}_k^A r_k = \lambda_{\max}^A(k) r_k$.
Then, it holds that 
\begin{align}
  \lambda_{\max}^A(k) \leq\Lambda^A(k),
\end{align}
with equality if and only if $\widehat{\mathcal{L}}_k^A\psi_k=0$, i.e.,
the principal eigenfunction of the symmetrized tilted generator is annihilated by the antisymmetric part.
Equivalently, equality is obtained if and only if $\psi_k$ 
is the eigenfunction corresponding to $\lambda_{\max}^A(k)$.
\end{proposition}

\begin{proof}
To prove the bound for principal eigenvalues we consider the spectral problem of the respective 
dominant eigenvalues
\begin{align}
\mathcal{L}_k^A r_k &= \lambda^A_{\max}(k) r_k,
\\
\widetilde{\mathcal{L}}_k^A \psi_k &= \Lambda^A(k) \psi_k,
\end{align} 
with short-hand notation $\lambda^A_{\max}(k)\equiv\lambda_{\max}(\mathcal{L}_k^A)$.
From this, it follows that $\langle r_k,\mathcal{L}_k^A r_k\rangle_\mu = \lambda_{\max}^A(k)\| r_k \|^2$
and 
$\langle \psi_k, \widetilde{\mathcal{L}}_k^A \psi_k \rangle_\mu = \Lambda^A(k) \| \psi_k \|^2 $.
Therefore, 
\begin{align}
\lambda^A_{\max}(k)
  =\frac{\langle r_k,\mathcal{L}_k^A r_k\rangle_\mu}{\|r_k\|^2}
  =\frac{\langle r_k,\widetilde{\mathcal{L}}_k^A r_k\rangle_\mu}{\|r_k\|^2}
  \leq \sup_{f} \frac{\langle f,\widetilde{\mathcal{L}}_k^A f\rangle_\mu}{\|f\|^2}
  =\Lambda^A(k)
  =\frac{\langle \psi_k,\widetilde{\mathcal{L}}_k^A \psi_k \rangle_\mu}{\|\psi_k\|^2},
\end{align}
where the first equality follows from Lemma~\ref{lem:quad_identity}
and the second to last equality uses the variational characterization 
of $\Lambda^A(k)$ from Definition~\ref{def:Lambda}.
Notably, equality of the dominant eigenvalues,
$\lambda^A_{\max}(k)=\Lambda^A(k)$, 
forces that $r_k$ maximizes the Rayleigh quotient.
Since $\widetilde{\mathcal{L}}_k^A $ is self-adjoint and $\Lambda^A(k)$
is a simple eigenvalue, the maximizer corresponds to 
the unique dominant eigenfunction , hence $r_k=\psi_k$.
This implies that the antisymmetric part of the tilted generator annihilates
the ground state of the symmetrized tilted generator, $\widehat{\mathcal{L}}_k^A\psi_k=0$,
which follows from 
\begin{align}
\widehat{\mathcal{L}}_k^A\psi_k=(\mathcal{L}_k^A - \widetilde{\mathcal{L}}_k^A )\psi_k
= \lambda^A_{\max}(k) \psi_k - \Lambda^A(k)\psi_k = 0,
\end{align}
using that $\psi_k=r_k$ is the eigenfunction of both operators 
with now equal dominant eigenvalue. 
Conversely, the reverse direction that $\widehat{\mathcal{L}}_k^A\psi_k=0$ implies 
$\lambda^A_{\max}(k)=\Lambda^A(k)$
follows from 
\begin{align}
\mathcal{L}^A_k \psi_k
= \widetilde{\mathcal{L}}_k^A\psi_k + \widehat{\mathcal{L}}_k^A\psi_k
= \Lambda^A(k) \psi_k + \widehat{\mathcal{L}}_k^A\psi_k.
\end{align}
This means that if $\widehat{\mathcal{L}}_k^A\psi_k=0$ then $\psi_k>0$ is an eigenfunction of $\mathcal{L}^A_k$ 
with corresponding eigenvalue $\Lambda^A(k)$, i.e., $\lambda^A_{\max}(k) =\Lambda^A(k)$.
\end{proof}

By Proposition~\ref{prop:tightness} the bound on the eigenvalue $\lambda^A(k)$ is tight if
and only if the antisymmetric part annihilates the ground state,
$\widehat{\mathcal{L}}_k^A\psi_k = 0$. It holds trivially when
the antisymmetric operator vanishes completely, $\widehat{\mathcal{L}}_k^A = 0$, which is
the self-adjoint case of Section~\ref{sec:selfadjoint_tilted}. 
However, more generally, it holds
whenever $\widehat{\mathcal{L}}_k^A \neq 0$ but $\psi_k$ lies in its kernel, i.e., the operator
still acts on non-trivial other functions but annihilates just the ground state $\psi_k$.

\paragraph{Density functionals}
For additive functionals of density-type the antisymmetric part reads
$\widehat{\mathcal{L}}_k^\rho = \mathcal{L}_A = \vs\cdot\nabla$, which is independent of $k$
and of $V$. Equality of the dominant eigenvalues is thus achieved if and only if
$\vs\cdot\nabla\psi_k = 0$. For reversible dynamics this holds directly since $\vs = 0$. In
the irreversible case ($\vs \neq 0$) a constant eigenfunction $\psi_k = \mathbf{1}$ is directly
annihilated by the antisymmetric part, $\vs\cdot\nabla\mathbf{1} = 0$,
without giving conditions on the dynamics or $V$.
Moreover, to check when $\psi_k=\mathbf{1}$ still
corresponds to the symmetric ground state we compute
\begin{align}
  \widetilde{\mathcal{L}}_k^\rho \mathbf{1} = \Lambda^\rho(k) \mathbf{1} = kV(x)\mathbf{1},
  \label{eq:density_eigval_flatground}
\end{align}
where $\widetilde{\mathcal{L}}_k^\rho = \mathcal{L}_S + kV$ 
with $\mathcal{L}_S\mathbf{1} = 0$.
This shows that 
Eq.~\eqref{eq:density_eigval_flatground} takes the form of an eigenvalue equation,
$\widetilde{\mathcal{L}}_k^\rho \mathbf{1} = \Lambda^\rho(k) \mathbf{1}$
with  $\Lambda^\rho(k) = kV$,
only if $V(x)=\mathrm{const.}$ for all $x$ and $k \neq 0$
and in this case it holds that $\Lambda^\rho(k)=kV$.

\paragraph{Current functionals}
For additive functionals of Stratonovich-type the antisymmetric part
is given by $\widehat{\mathcal{L}}_k^J=(\vs+2k DU)\cdot\nabla+k \ps^{-1}\nabla\cdot(\ps DU)$
and now depends on $k$. 
A useful ansatz to check for conditions on the dynamics is again 
$\psi_k = \mathbf{1}$ 
since at $k=0$ the tilted generator reduces to 
$\mathcal{L}_0^J = \mathcal{L}$ where the constant function is
the ground state of $\mathcal{L}_S$, (i.e., $\psi_0 =\mathbf{1}$)
and also $\mathcal{L}$ since $\mathcal{L}_A\mathbf{1}=0$.
The question for equality of the principal eigenvalues thus 
becomes whether we can identify conditions on the dynamics and $U$ such that
$\psi_k=\mathbf{1}$ holds for $k>0$.

The first term of $\widehat{\mathcal{L}}_k^J$ annihilates constants, i.e., $(\vs+2k DU)\cdot\nabla\mathbf{1}=0$,
and we obtain
\begin{align}
\widehat{\mathcal{L}}_k^J \mathbf{1}=k \ps^{-1}\nabla\cdot(\ps DU)\mathbf{1},
\end{align} 
such that $\widehat{\mathcal{L}}_k^J\mathbf{1}=0$ requires $\ps^{-1}\nabla\cdot(\ps DU)=0$.
Moreover, to identify the ansatz $\psi_k=1$ as the symmetric ground state, using ~$\mathcal{L}_S\mathbf{1}=0$,
\begin{align}
\widetilde{\mathcal{L}}_k^J\mathbf{1}= [k(\vs\cdot U)+k^2(U\cdot D U)]\mathbf{1}
\end{align}
must be an eigenvalue problem.
Therefore, the prefactor $[k(\vs\cdot U)+k^2(U\cdot D U)]$ must be
constant in $x$ for $k\neq 0$, which holds only if
$(\vs\cdot U)=\mathrm{const.}$ and $(U\cdot  D U)=\mathrm{const.}$ which is fixed by the dynamics (through $\vs$)
and $U$.
Under these two conditions the principal eigenvalue of the symmetrized tilted generator 
reads $\Lambda^J(k)=k(\vs\cdot U)+k^2(U\cdot D U)$.
If it also holds that $\ps^{-1}\nabla\cdot(\ps DU)=0$ from above,
then equality holds, i.e., $\lambda_{\max}^J(k) = \Lambda^J(k)$.
Notably, this is achieved whenever $\ps$, $\js$, $D$, and $U$ are constant, which crucially does not
require that $\js=0$.

\section{General concentration inequalities}
\label{sec:conc_bound_general}
We now combine the results of the previous sections into concentration inequalities
for the empirical density and current. 
In particular, we require the dynamical version of the Cram\'er-Chernoff
inequality~\eqref{eq:chernoff_functional} and the bound of the moment generating function 
in Eq.~\eqref{eq:mgf_final_bound}. 
After optimizing over the tilt
parameter $k$ we obtain exponential tail bounds that hold at every finite time $t>0$ and
every deviation $a>0$, rather than only in the large-time limit.
Concretely, the exponential bound is expressed in terms of the 
\emph{Cram\'er transform} which is obtained as the
Legendre transform of the principal eigenvalues
$\Lambda^\rho(k)$ and $\Lambda^J(k)$, respectively.

\subsection{Concentration bound for density observables}
\label{sec:conc_density}
\begin{theorem}[A general concentration inequality for density observables]
\label{thm:conc_density}
Let $(X_t)_{t \geq 0}$ be an ergodic diffusion with invariant measure $\mu$, let
$\overline{\rho}_t(V)$ denote the empirical density according to Definition~\ref{def:empirical_density} with 
ergodic (time-independent) mean $\E^\mu[\overline{\rho}_t]$, and
let $\nu \ll \mu$ with $\dd\nu/\dd\mu \in L^2(\mu)$. 
Then for all $t > 0$ and
$a > 0$,
\begin{align}
  \mathbb{P}^\nu\left(\overline{\rho}_t(V) - \E^\mu[\overline{\rho}_t] \geq a\right)
  \leq
  \left\|\frac{\dd\nu}{\dd\mu}\right\|_{L^2(\mu)}
  \exp\bigl[-t I^\rho(\E^\mu[\overline{\rho}_t] + a)\bigr],
  \label{eq:conc_density_right}
\end{align}
where the Cram\'er transform
\begin{align}
I^\rho(s) \equiv \sup_{k \in \mathbb{R}}\bigl[k s - \Lambda^\rho(k)\bigr],
\label{eq:cramer_trafo_general_density}
\end{align}
is defined as the Legendre conjugate of the principal eigenvalue $\Lambda^\rho(k)$~\eqref{eq:Lambda_rho_dirichlet}.
\end{theorem}

\begin{proof}
Applying Chernoff's inequality to the centered variable
$\overline{\rho}_t - \E^\mu[\overline{\rho}_t]$ gives, for any $k>0$,
\begin{align}
  \mathbb{P}^\nu\left(\overline{\rho}_t - \E^\mu[\overline{\rho}_t] \geq a\right)
  \leq \e^{-k t(\E^\mu[\overline{\rho}_t]+a)}\mathbb{E}^\nu\bigl[\e^{k\rho_t}\bigr],
  \label{eq:chernoff_density}
\end{align}
where $\rho_t = t \overline{\rho}_t$.
By the bound on the moment generating function (Corollary~\ref{cor:mgf_bound}), we have
$\E^\nu\bigl[\e^{k\rho_t}\bigr]\leq\|\dd\nu/\dd\mu\|_{L^2(\mu)}\e^{t\Lambda^\rho(k)}$
where $\Lambda^\rho(k)$ is the
principal eigenvalue of $\widetilde{\mathcal{L}}_k^\rho = \mathcal{L}_S
+ kV$. 
Optimizing over $k$ then yields
\begin{align}
  \mathbb{P}^\nu\left(\overline{\rho}_t - \E^\mu[\overline{\rho}_t] \geq a\right)
  &\leq
  \left\|\frac{\dd\nu}{\dd\mu}\right\|_{L^2(\mu)}
  \inf_{k>0}\exp\bigl[t(\Lambda^\rho(k) - k(\E^\mu[\overline{\rho}_t]+a))\bigr]
  \nonumber\\
  &=
  \left\|\frac{\dd\nu}{\dd\mu}\right\|_{L^2(\mu)}
  \exp\Bigl(-t\sup_{k>0}\bigl[k(\E^\mu[\overline{\rho}_t]+a) - \Lambda^\rho(k)\bigr]\Bigr).
  \label{eq:before_legendre}
\end{align}
Since $\Lambda^\rho$ is convex with $\Lambda^\rho(0)=0$ and
$(\Lambda^\rho)'(0)= \E^\mu[\overline{\rho}_t]$ for $a>0$ the optimizer satisfies
$(\Lambda^\rho)'(k^*)=\E^\mu[\overline{\rho}_t]+a>\E^\mu[\overline{\rho}_t]$ hence $k^*>0$.
Introducing the Cram\'er transform $I^\rho(s) \equiv \sup_{k \in \mathbb{R}}\bigl[k s - \Lambda^\rho(k)\bigr]$,
the restricted supremum therefore coincides with the full Legendre transform $I^\rho(\E^\mu[\overline{\rho}_t]+a)$
and completes the proof.
\end{proof}

\begin{corollary}[Left tail and two-sided bounds for density observables]
\label{cor:two_sided_density}
Under the assumptions of
Theorem~\ref{thm:conc_density},
for all $t>0$ and $a>0$, the left tail is bounded via
\begin{align}
  \mathbb{P}^\nu\left(\overline{\rho}_t - \E^\mu[\overline{\rho}_t] \leq -a\right)
  \leq
  \left\|\frac{\mathrm{d}\nu}{\mathrm{d}\mu}\right\|_{L^2(\mu)}
  \exp\bigl[-tI^\rho(\E^\mu[\overline{\rho}_t] - a)\bigr].
  \label{eq:conc_density_left}
\end{align}
Consequently, the two-sided
deviation probability satisfies
\begin{align}
  \mathbb{P}^\nu\left(\bigl|\overline{\rho}_t - \E^\mu[\overline{\rho}_t]\bigr|\geq a\right)
  \leq
  2\left\|\frac{\dd\nu}{\dd\mu}\right\|_{L^2(\mu)}
  \exp\bigl[-t\min\bigl(I^\rho(\E^\mu[\overline{\rho}_t] + a),I^\rho(\E^\mu[\overline{\rho}_t] - a)\bigr)\bigr].
  \label{eq:conc_density_two_sided}
\end{align}
\end{corollary}

\begin{proof}
For the left tail we directly apply Theorem~\ref{thm:conc_density}
to $-\overline{\rho}_t(V)$ (or $-V$) now with mean $-\E^\mu[\overline{\rho}_t]$.
From Eq.~\eqref{eq:Lambda_rho_dirichlet} we see that replacing $V$ by $-V$ 
maps $\Lambda^\rho(k) \mapsto \Lambda^\rho(-k)$
and therefore $I^\rho(s) \mapsto I^\rho(-s)$ by symmetry
of the Legendre transform under sign change.
The two-sided result is then obtained by a union bound which gives Eq.~\eqref{eq:conc_density_two_sided}.
\end{proof}

\subsection{Concentration bound for current observables}
\label{sec:conc_current}
We now turn to current observables which are expressed as 
additive functionals of Stratonovich type.
Our argument follows the density case
(Theorem~\ref{thm:conc_density}) by using the Cram\'er-Chernoff method together
with the bound on the moment-generating function of Eq.~\eqref{eq:mgf_final_bound}, which reduces the
problem of bounding the tail behavior to the principal eigenvalue $\Lambda^J(k)$. 
However, we highlight that the current-specific structure of
$\Lambda^J(k)$---i.e., the terms $k(\js\cdot U)/\ps$ and $k^2(U\cdot DU)$---give rise to a different Cram\'er transform $I^J$.

\begin{theorem}[General concentration inequality for empirical currents]
\label{thm:conc_current}
Let $(X_t)_{t \geq 0}$ be an ergodic diffusion with invariant measure $\mu$, let
$\overline{J}_t(U)$ be the empirical current in the sense of Definition~\ref{def:empirical_current} with 
ergodic mean $\E^\mu[\overline{J}_t] = \int \js \cdot U \dd x$, and let $\nu \ll \mu$
with $\dd\nu/\dd\mu \in L^2(\mu)$. Then for all $t > 0$ and $a > 0$,
\begin{align}
  \mathbb{P}^\nu\left(\overline{J}_t(U) - \E^\mu[\overline{J}_t] \geq a\right)
  \leq
  \left\|\frac{\dd\nu}{\dd\mu}\right\|_{L^2(\mu)}
  \exp\bigl[-tI^J\bigl(\E^\mu[\overline{J}_t] + a\bigr)\bigr],
  \label{eq:conc_current_right}
\end{align}
where the Cram\'er transform
\begin{align}
  I^J(s) \equiv \sup_{k \in \mathbb{R}}\bigl[k s - \Lambda^J(k)\bigr],
  \label{eq:cramer_curr_def}
\end{align}
is defined as the Legendre conjugate of the principal eigenvalue $\Lambda^J(k)$~\eqref{eq:Lambda_J_dirichlet}.
\end{theorem}

\begin{proof}
The argument follows Theorem~\ref{thm:conc_density} with $A=J$.
By Chernoff's inequality, for any $k>0$,
\begin{align}
  \mathbb{P}^\nu\left(\overline{J}_t(U) - \E^\mu[\overline{J}_t] \geq a\right)
  \leq \e^{-kt(\E^\mu[\overline{J}_t] + a)}\mathbb{E}^\nu\bigl[\e^{kJ_t(U)}\bigr].
  \label{eq:chernoff_current}
\end{align}
with $J_t = t \overline{J}_t$.
By the bound on the moment generating function (Corollary~\ref{cor:mgf_bound}), we have
$\E^\nu[\e^{k J_t}]\leq\|\dd\nu/\dd\mu\|_{L^2(\mu)}\e^{t\Lambda^J(k)}$
where $\Lambda^J(k)$ is the
principal eigenvalue of $\widetilde{\mathcal{L}}_k^J = \mathcal{L}_S
+ k(\vs\cdot U) + k^2(U\cdot DU)$. 
Optimizing over $k$ then yields
\begin{align}
  \mathbb{P}^\nu\left(\overline{J}_t(U)-\E^\mu[\overline{J}_t] \geq a\right)
  &\leq
  \left\|\frac{\dd\nu}{\dd\mu}\right\|_{L^2(\mu)}
  \inf_{k>0}\exp\bigl[t(\Lambda^J(k) - k(\E^\mu[\overline{J}_t]+a))\bigr]
  \nonumber\\
  &=
  \left\|\frac{\dd\nu}{\dd\mu}\right\|_{L^2(\mu)}
  \exp\Bigl(-t\sup_{k>0}\bigl[k(\E^\mu[\overline{J}_t]+a) - \Lambda^J(k)\bigr]\Bigr).
\end{align}
Since $\Lambda^J$ is convex with $\Lambda^J(0)=0$ and
$(\Lambda^J)'(0)=\E^\mu[\overline{J}_t]$ for $a>0$ the optimal $k^\ast$ satisfies
$(\Lambda^J)'(k^*)=\E^\mu[\overline{J}_t]+a$ hence $k^*>0$.
Introducing the Cram\'er transform $I^J(s) \equiv \sup_{k \in \mathbb{R}}\bigl[k s - \Lambda^J(k)\bigr]$,
the restricted supremum therefore equals the full Legendre transform $I^J(\E^\mu[\overline{J}_t]+a)$ 
and completes the proof.
\end{proof}

\begin{corollary}[Two-sided bound for empirical currents]
\label{cor:two_sided_current}
Under the assumptions of Theorem~\ref{thm:conc_current}, for all $t>0$ and $a>0$,
\begin{align}
  \mathbb{P}^\nu\left(\overline{J}_t(U) - \E^\mu[\overline{J}_t] \leq -a\right)
  &\leq
  \left\|\frac{\dd\nu}{\dd\mu}\right\|_{L^2(\mu)}
  \exp\bigl[-tI^J\bigl(\E^\mu[\overline{J}_t] - a\bigr)\bigr],
  \label{eq:conc_current_left}
\end{align}
and consequently
\begin{align}
  \mathbb{P}^\nu\left(\bigl|\overline{J}_t(U) - \E^\mu[\overline{J}_t]\bigr| \geq a\right)
  \leq
  2\left\|\frac{\dd\nu}{\dd\mu}\right\|_{L^2(\mu)}
  \exp\Bigl[-t\min\bigl(I^J(\E^\mu[\overline{J}_t] + a), I^J(\E^\mu[\overline{J}_t] - a)\bigr)\Bigr].
  \label{eq:two_sided_current}
\end{align}
\end{corollary}
\begin{proof}
Apply Theorem~\ref{thm:conc_current} to the observable $-U$. Since
$J_t(-U)=-J_t(U)$ and Eq.~\eqref{eq:Lambda_J_dirichlet} we get $\Lambda^J(-k)$
and the Cram\'er transform becomes $s\mapsto I^J(-s)$ which gives the
left-tail bound~\eqref{eq:conc_current_left}. A union bound
yields Eq.~\eqref{eq:two_sided_current} and completes the proof.
\end{proof}

\begin{corollary}[Sub-Gaussian concentration under detailed balance]
\label{cor:db_subgaussian}
    Let assumptions of Theorem~\ref{thm:conc_current} hold
    and further assume that the dynamics is reversible, i.e., $\js\equiv0$.
    By Proposition~\ref{prop:mean} $\E^\mu[\overline{J}_t]=0$ for every admissible $U$,
    and for all $t>0$ and $a>0$ it holds
\begin{align}
  \mathbb{P}^\nu\bigl(\overline{J}_t(U) \geq a\bigr)
  &\leq
  \left\|\frac{\dd\nu}{\dd\mu}\right\|_{L^2(\mu)}
  \exp\left[-\frac{ta^2}{4\| Q\|_{L^\infty(\mu)}}\right],
  \label{eq:subgau_DB}
\end{align}
where $Q\equiv U\cdot DU$. Therefore, $\overline{J}_t$ is sub-Gaussian with 
\emph{fully explicit} variance proxy $2\| Q\|_{L^\infty(\mu)}/t$ for all $t>0$ and $a>0$.
\end{corollary}
\begin{proof}
    Under detailed balance we have $\vs=0$ and the variational formula of $\Lambda^J$ from Proposition~\ref{prop:Lambda_variational} reduces, $k\in\mathbb{R}$, to
    \begin{align}
        \Lambda^J(k) 
        =
        \sup\left\{k^2\langle U\cdot DU,f^2\rangle_\mu-\mathcal{E}(f,f)
    \mathrel{\Big|} f \in D(\mathcal{L}_k^J),\|f\|_{L^2(\mu)} = 1
  \right\}
  \leq k^2 \| Q\|_{L^\infty(\mu)},
    \end{align}
    since $\langle Q,f^2\rangle_\mu\leq\|Q\|_{L^\infty(\mu)}$ for $\|f\|_{L^2(\mu)}=1$ and $\mathcal{E}(f,f)\geq 0$. The Cram\'er transform
    is correspondingly bounded from below
    \begin{align}
  I^J(a) \geq \sup_{k\in\mathbb{R}}\bigl[ka - k^2\|Q\|_{L^\infty(\mu)}\bigr]
  = \frac{a^2}{4\|Q\|_{L^\infty(\mu)}},
\end{align}
and yields the concentration bound.
\end{proof}
Strikingly, Corollary~\ref{cor:db_subgaussian}
requires only the non-negativity of the Dirichlet form $\mathcal{E}(f,f)$
and does \emph{not} rely on the Poincar\'e inequality, i.e., 
Eq.~\eqref{eq:subgau_DB} holds under Assumptions~\ref{ass:dynamics}
irrespective of whether a spectral gap exists. Notably, when $Q=\mathrm{const.}$ (e.g., constant $U$) we have that $\Lambda^J(k)=k^2Q$ exactly.

\begin{remark}[Asymptotic sharpness in the large deviation sense]\label{rem:ldt_sharpness}
In the language of large deviation theory 
additive functionals satisfy a large deviation principle 
(a continuous time Markovian analog of Cram\'er's Theorem 
for sums of i.i.d.~random variables \cite{touchette2009large,touchette2018introduction}) which 
states that asymptotically, the tail probability of additive functionals
decays exponentially, with
\begin{align}
\lim_{t\to\infty}-\frac{1}{t}\log\mathbb{P}\bigl(\overline{A}_t\geq a \bigr) = I^A_{\mathrm{LD}}(a),
\end{align}
where $I^A_{\mathrm{LD}}(a)$ denotes the so-called rate function.
Importantly, the G\"artner-Ellis theorem \cite{gartner1977large,ellis1984large,touchette2009large} states that the rate function
is obtained via the Legendre transform 
$I^A_{\mathrm{LD}}(a)=\sup_{k\in\mathbb{R}}\bigl[ ka - \lambda^A_{\mathrm{LD}}(k)\bigr]$, 
with the scaled cumulant generating 
function defined as $\lambda^A_{\mathrm{LD}}(k)\equiv\lim_{t\to\infty} \tfrac{1}{t}\log\E^x\bigl[\e^{kt\overline{A}_t} \bigr]$,
and assuming it exists and is differentiable.
Using the Feynman-Kac bound 
\begin{align}
\log\E^\nu\bigl[\e^{kt\overline{A}_t}\bigr] \leq
 t\Lambda^A(k)
 +  \log\left\|\frac{\dd\nu}{\dd\mu}\right\|_{L^2(\mu)},
\end{align}
we obtain the estimate
\begin{align}
\lambda_{\mathrm{LD}}^A(k)\leq \Lambda^A(k),
\end{align}
where the dependence on the initial measure $\nu$ drops out as $t\to\infty$.
Using an expansion of the Feynman-Kac semigroups with the full tilted generator $\mathcal{L}_k^A$ in 
terms of eigenfunctions $r_k^{(i)}$ and eigenvalues $\lambda_i^A(k)$
it can be shown~\cite{touchette2018introduction} that in the limit $t\to\infty$
the eigenvalue with largest real part dominates, such that it is the scaled cumulant generating
function, i.e., $\lambda_{\mathrm{LD}}^A(k)=\lambda_{\max}(\mathcal{L}_k^A)\equiv\lambda_{\max}^A(k)$.
Consequently, 
the non-asymptotic concentration bounds of Theorems~\ref{thm:conc_density} and~\ref{thm:conc_current}
are asymptotically sharp in the sense of large deviation theory
whenever the principal eigenvalues of the full and symmetrized tilted generators coincide, $\lambda_{\max}^A(k)=\Lambda^A(k)$, 
which is discussed in detail after Proposition~\ref{prop:tightness}.
\end{remark}

\section{Concentration bounds via functional inequalities}
\label{sec:conc_poincare}
The general bounds in Theorems~\ref{thm:conc_density} and~\ref{thm:conc_current} 
control the tail probabilities of  
the additive functionals $\overline{\rho}_t(V)$
and $\overline{J}_t(U)$ in terms of their respective Cram\'er transforms $I^\rho(s)$ and $I^J(s)$, 
defined as the Legendre transforms of the principal eigenvalues $\Lambda^\rho(k)$ and
$\Lambda^J(k)$. However, evaluating these eigenvalues as functions of the tilt $k$ requires
solving, for each value of $k$, the associated self-adjoint spectral problems for the symmetrized tilted generators
(Definition~\ref{def:Lambda}).
The goal of this section is to replace this 
spectral computation with a \emph{functional inequality}.
Concretely, the Poincar\'e inequality 
$\Var_\mu(f) \leq \lambda_{\textrm{gap}}^{-1}\mathcal{E}(f,f)$ (see Eq.~\eqref{eq:poincare})
allows us to derive fully explicit upper bounds on $\Lambda^\rho(k)$ and
$\Lambda^J(k)$ in terms of the spectral gap $\lambda_\mathrm{gap}$ of the
symmetric part of the underlying dynamics, $\mathcal{L}_S$, and the norms of the 
functions $V$ and $U$ that define the path observables. 
In particular, these bounds yield closed-form sub-gamma and
classical Bernstein-type concentration inequalities.

\subsection{A sub-gamma Cram\'er transform}
The explicit bounds on the principal eigenvalues  will involve
the characteristic sub-gamma structure 
$\tilde\sigma_A^2 k^2/[2(1-c_A k)]$~\cite{boucheron2013concentration}
with parameters $\tilde\sigma_A^2$ and $c_A$ 
fixed by the observable and the
spectral gap.
We therefore briefly re-state its corresponding Cram\'er transform for convenience in the following Lemma.

\begin{lemma}[Sub-gamma Cram\'er transform]
\label{lem:subgamma}
Let $\tilde\sigma^2 > 0$ and $c > 0$. 
Consider the convex function on $[0, 1/c)$
\begin{align}
\phi(k) = \frac{k^2\tilde\sigma^2}{2(1 - ck)}.
\end{align}
Its Legendre transform $\phi^\ast(a)$ is, for $a > 0$,
\begin{align}
  \phi^\ast(a)
  \equiv \sup_{0 \leq k < 1/c}\left[ka - \frac{\tilde\sigma^2 k^2}{2(1 - ck)}\right]
  = \frac{\tilde\sigma^2}{c^2}h\left(\frac{c a}{\tilde\sigma^2}\right),
  \label{eq:subgamma_legendre}
\end{align}
with supremum attained at $k^\ast = c^{-1}\bigl(1 - \sqrt{\tilde\sigma^2/(\tilde\sigma^2 + 2ca)}\bigr) > 0$
and we define the auxiliary function $h(u) \equiv 1 + u - \sqrt{1 + 2u}$.
Moreover, we have $h(u) \geq u^2/[2(1+u)]$ for $u \geq 0$, i.e.,
\begin{align}
  \phi^\ast(a) \geq \frac{a^2}{2(\tilde\sigma^2 + ca)}.
  \label{eq:subgamma_bernstein}
\end{align}
\end{lemma}

\begin{proof}
With the substitution $w \equiv 1 - ck \in (0,1]$ the bracket
in Eq.~\eqref{eq:subgamma_legendre} now reads
$c^{-1}(1-w)a - \tfrac{1}{2}c^{-2}\tilde\sigma^2(1-w)^2/w$, and setting its derivative
to zero gives, for $a > 0$,
\begin{align}
  w^\ast = \sqrt{\frac{\tilde\sigma^2}{\tilde\sigma^2 + 2ca}},
  \qquad
  k^\ast 
  = 
  \frac{1}{c}\left(1 - \sqrt{\frac{\tilde\sigma^2}{\tilde\sigma^2 + 2ca}}\right)> 0.
\end{align}
Substituting $w^\ast$ back yields
\begin{align}
  \psi^\ast(a)
  = \frac{\tilde\sigma^2 + 2ca}{2c^2}(1 - w^\ast)^2
  = \frac{1}{c^2}\left(\tilde\sigma^2 + ca - \sqrt{\tilde\sigma^2(\tilde\sigma^2 + 2ca)}\right)
  = \frac{\tilde\sigma^2}{c^2}h\left(\frac{ca}{\tilde\sigma^2}\right),
\end{align}
which gives Eq.~\eqref{eq:subgamma_legendre}. For the second statement, we use the elementary
inequality $h(u) = 1 + u - \sqrt{1+2u} \geq u^2/[2(1+u)]$ for $u \geq 0$.
Applying this bound with $u = ca/\tilde\sigma^2$ yields
\begin{align}
  \frac{\tilde\sigma^2}{c^2}h\left(\frac{ca}{\tilde\sigma^2}\right)
  \geq \frac{\tilde\sigma^2}{c^2}\frac{(ca/\tilde\sigma^2)^2}{2(1 + ca/\tilde\sigma^2)}
  = \frac{a^2}{2(\tilde\sigma^2 + ca)},
\end{align}
which is Eq.~\eqref{eq:subgamma_bernstein}.
\end{proof}

\subsection{Explicit bound for density observables}
\label{sec:explicit_density}
We begin with the density case. Recall from Proposition~\ref{prop:Lambda_variational} that the principal eigenvalue 
can be expressed in terms of the Dirichlet form, for $f\in D(\mathcal{E})$,
\begin{align}
  \Lambda^\rho(k)
  =
  \sup\left\{
    k\langle V, f^2\rangle_\mu - \mathcal{E}(f,f)
    \middle|
    \|f\|_{L^2(\mu)} = 1
  \right\}.
  \label{eq:dirichlet_density}
\end{align}
Therefore, bounding $\Lambda^\rho(k)$ amounts to controlling
the functional part 
$\langle V, f^2\rangle_\mu$ against the Dirichlet form $\mathcal{E}(f,f)$. 
The Poincar\'e inequality~\eqref{eq:poincare} provides exactly this control
and leads to the following lemma.

\begin{lemma}[Poincaré bound on $\Lambda^\rho$]
\label{lem:poincare_density}
Let Assumptions~\ref{ass:dynamics} hold, such that the symmetric part
$\mathcal{L}_S$ satisfies the Poincar\'e inequality~\eqref{eq:poincare} with
spectral gap $\lambda_\mathrm{gap}>0$ and let $\E^\mu[V]\equiv\int V\dd \mu$. Then for every $k$ with
$0\leq k < \lambda_{\mathrm{gap}}/\|V-\E^\mu[V]\|_{L^\infty(\mu)}$,
\begin{align}
  \Lambda^\rho(k)
  \leq
  k \E^\mu[V]
  + \frac{k^2 \mathrm{Var}_\mu(V)}{\lambda_\mathrm{gap} - k \|V-\E^\mu[V]\|_{L^\infty(\mu)}}.
  \label{eq:Lambda_rho_bound}
\end{align}
\end{lemma}

\begin{proof}
We start with the
Dirichlet representation of $\Lambda^\rho(k)$~\eqref{eq:dirichlet_density}.
Using $\|f\|_{L^2(\mu)}=1$,  $V$ can be re-centered, i.e., 
$\bar V \equiv V - \E^\mu[V]$ with $\E^\mu[\bar V] = 0$,
\begin{align}
  \langle V,f^2\rangle_\mu = \E^\mu[V] + \langle \bar V,f^2\rangle_\mu,
\end{align}
such that
$\Lambda^\rho(k) = k\E^\mu[V]
+ \sup_{\|f\|=1}\bigl[-\mathcal{E}(f,f)+k\langle\bar V,f^2\rangle_\mu\bigr]$.
To evaluate the remaining supremum explicitly, we parameterize $f$ in the form
\begin{align}
f=\frac{1 + \varepsilon g}{\sqrt{1+\varepsilon^2}},
\end{align}
for some $\varepsilon\geq 0$ and $g\in L^2(\mu)$ satisfying
$\|g\|_{L^2(\mu)}=1$ and $\E^\mu[g] = 0$. In particular, one computes directly that
\begin{align}
  \mathrm{Var}_\mu(f) = \frac{\varepsilon^2}{1+\varepsilon^2},
  \label{eq:var_param}
\end{align}
and, since $\E^\mu[\bar V] = 0$,
\begin{align}
  \langle \bar V, f^2\rangle_\mu = \frac{1}{1+\varepsilon^2}
  \left[2\varepsilon\langle \bar V,  g\rangle_\mu + \varepsilon^2\langle \bar V,g^2\rangle_\mu\right].
\end{align}
By the Poincar\'e inequality~\eqref{eq:poincare},
$\mathcal{E}(f,f) \geq \lambda_\mathrm{gap} \mathrm{Var}_\mu(f)$, hence
\begin{align}
\Lambda^\rho(k) - k\E^\mu[V]
&\leq
\sup_{\|f\|=1}\left[-\lambda_\mathrm{gap}\mathrm{Var}_\mu(f) + k\langle \bar V,f^2\rangle_\mu\right]
\nonumber
\\
&=
\sup_{\varepsilon\geq 0}\frac{1}{1+\varepsilon^2}
\left[2 k \varepsilon\langle \bar V, g\rangle_\mu
+ \varepsilon^2 \langle (k\bar V-\lambda_\mathrm{gap}), g^2\rangle_\mu\right].
\end{align}
Next, we bound the $\varepsilon^2$ term as
\begin{align}
\langle k\bar V - \lambda_\mathrm{gap}, g^2\rangle_\mu
\leq k\|\bar V\|_{L^\infty(\mu)} - \lambda_\mathrm{gap},
\label{eq:epsilon_sqr_bound}
\end{align}
where we use that pointwise $\bar V\leq \|\bar V\|_{L^\infty(\mu)}$ and $\|g\|_{L^2(\mu)}=1$.
Moreover, applying Cauchy-Schwarz to the cross term yields
\begin{align}
\langle \bar V,g\rangle_\mu
\leq \|\bar V\|_{L^2(\mu)}\|g\|_{L^2(\mu)} = \sqrt{\mathrm{Var}_\mu(V)},
\end{align}
using $\|\bar V\|_{L^2(\mu)}^2 = \mathrm{Var}_\mu(V)$ and $\|g\|_{L^2(\mu)}=1$.
Therefore,
\begin{align}
\Lambda^\rho(k) - k\E^\mu[V]
&\leq
\sup_{\varepsilon \geq 0}\frac{1}{1+\varepsilon^2}
\left[2 k \varepsilon \sqrt{\mathrm{Var}_\mu(V)}
+ \varepsilon^2 (k\|\bar V\|_{L^\infty(\mu)} - \lambda_\mathrm{gap})\right]
\nonumber
\\ 
&\leq
\sup_{\varepsilon \geq 0}
\left[2 k \varepsilon \sqrt{\mathrm{Var}_\mu(V)}
+ \varepsilon^2 (k\|\bar V\|_{L^\infty(\mu)} - \lambda_\mathrm{gap})\right],
\end{align}
where in the last line we drop the prefactor using $1/(1+\varepsilon^2)\leq 1$.
It remains to evaluate the supremum over $\varepsilon\geq 0$ of the quadratic polynomial
\begin{align}
P^\rho(\varepsilon) \equiv
\varepsilon^2 (k\|\bar V\|_{L^\infty(\mu)} - \lambda_\mathrm{gap})
+ 2 k \varepsilon \sqrt{\mathrm{Var}_\mu(V)}.
\end{align}
For $k\|\bar V\|_{L^\infty(\mu)} < \lambda_\mathrm{gap}$ it is concave and the supremum is
attained at
\begin{align}
  \varepsilon_0
  = \frac{k \sqrt{\mathrm{Var}_\mu(V)}}{\lambda_\mathrm{gap} - k \|\bar V\|_{L^\infty(\mu)}}.
  \label{eq:eps_star}
\end{align}
Finally, the desired bound follows from $P^\rho(\varepsilon_0)$ as
\begin{align}
  \Lambda^\rho(k)
  \leq
  k \E^\mu[V]
  + \frac{k^2\mathrm{Var}_\mu(V)}{\lambda_\mathrm{gap} - k \|\bar V\|_{L^\infty(\mu)}},
\end{align}
which completes the proof.
\end{proof}

\begin{theorem}[Concentration inequality for density observables]
\label{thm:explicit_density}
Let $(X_t)_{t\geq 0}$ be an ergodic diffusion with invariant measure $\mu$ whose
symmetric part $\mathcal{L}_S$ satisfies the Poincar\'e inequality~\eqref{eq:poincare}
with spectral gap $\lambda_\mathrm{gap}>0$ (by Assumptions~\ref{ass:dynamics}). 
Let $\overline{\rho}_t(V)$ denote the empirical density according to Definition~\ref{def:empirical_density}
with ergodic mean $\E^\mu[\overline{\rho}_t(V)]=\E^\mu[V]$,
and let $\nu\ll\mu$ with $\dd\nu/\dd\mu\in L^2(\mu)$.
Then for all $t > 0$ and $a > 0$,
\begin{align}
  \mathbb{P}^\nu\left(\overline{\rho}_t(V) - \E^\mu[\overline{\rho}_t] \geq a\right)
  \leq
  \left\|\frac{\dd\nu}{\dd\mu}\right\|_{L^2(\mu)}
  \exp\left[-t \frac{\tilde\sigma_\rho^2}{c_\rho^2} h\left(\frac{c_\rho a}{\tilde\sigma_\rho^2}\right)  \right],
  \label{eq:explicit_density}
\end{align}
where $h(u) \equiv 1 + u - \sqrt{1 + 2u}$ and we define 
\begin{align}
  \tilde\sigma_\rho^2 \equiv \frac{2\mathrm{Var}_\mu(V)}{\lambda_\mathrm{gap}},
  \qquad
  c_\rho \equiv \frac{\|V - \E^\mu[V]\|_{L^\infty(\mu)}}{\lambda_\mathrm{gap}}.
\end{align}
\end{theorem}

\begin{proof}
By Theorem~\ref{thm:conc_density} together with the Cram\'er transform~\eqref{eq:cramer_trafo_general_density}, 
it suffices to bound
$I^\rho(\E^\mu[\overline{\rho}_t] + a) = \sup_{k>0}\bigl[k(\E^\mu[\overline{\rho}_t]+a) - \Lambda^\rho(k)\bigr]$
from below. The mean directly cancels against the linear term of the Poincar\'e
bound~\eqref{eq:Lambda_rho_bound} since $\E^\mu[\overline{\rho}_t] =\E^\mu[V]$.
With  $\tilde\sigma_\rho^2$ and $c_\rho$ from above, 
we obtain the characteristic sub-gamma form, for $0 \leq k < 1/c_\rho$,
\begin{align}
  \Lambda^\rho(k) - k\E^\mu[\overline{\rho}_t]
  \leq \frac{k^2\mathrm{Var}_\mu(V)}{\lambda_\mathrm{gap} - k\|V  - \E^\mu[ V] \|_{L^\infty(\mu)}}
  = \frac{\tilde\sigma_\rho^2 k^2}{2(1 - c_\rho k)},
\end{align}
such that 
\begin{align}
I^\rho(\E^\mu[V] + a) \geq \sup_{0 \leq k < 1/c_\rho}\left[ka - 
\frac{\tilde\sigma_\rho^2 k^2}{2(1 - c_\rho k)}\right]
=\frac{\tilde\sigma_\rho^2}{c_\rho^2}h\left(\frac{c_\rho a}{\tilde\sigma_\rho^2}\right),
\end{align}
where $h(u) \equiv 1 + u - \sqrt{1 + 2u}$ and the last equality follows from Lemma~\ref{lem:subgamma} with $(\tilde\sigma^2, c)$ replaced by $(\tilde\sigma_\rho^2, c_\rho)$.
Inserting the lower bound into the Cram\'er transform of
Theorem~\ref{thm:conc_density} completes the proof.
\end{proof}

\begin{corollary}[Bernstein bound for the empirical density]
\label{cor:bernstein_density}
Under the assumptions of Theorem~\ref{thm:explicit_density}, for all $t>0$ and
$a>0$, the following simpler bound holds
\begin{align}
  \mathbb{P}^\nu\left(\overline{\rho}_t(V) - \E^\mu[\overline{\rho}_t] \geq a\right)
  \leq
  \left\|\frac{\dd\nu}{\dd\mu}\right\|_{L^2(\mu)}
  \exp\left(-\frac{ta^2}{2\bigl(\tilde\sigma_\rho^2 + c_\rho a\bigr)}\right).
  \label{eq:bernstein_density}
\end{align}
\end{corollary}

\begin{proof}
By Theorem~\ref{thm:explicit_density} and Lemma~\ref{lem:subgamma},
the inequality~\eqref{eq:subgamma_bernstein} with
$(\tilde\sigma^2, c) = (\tilde\sigma_\rho^2, c_\rho)$ gives a lower bound on the exponent
of Eq.~\eqref{eq:explicit_density} by $a^2/[2(\tilde\sigma_\rho^2 + c_\rho a)]$, which directly
gives Eq.~\eqref{eq:bernstein_density}.
\end{proof}

\begin{remark}[Gaussian and exponential regimes]
\label{rem:subgamma_density}
The exponent in Eq.~\eqref{eq:bernstein_density} interpolates between two different 
regimes depending on which term in $2(\tilde\sigma_\rho^2+c_\rho a)$ dominates,
set by
$\tilde\sigma_\rho^2=2\mathrm{Var}_\mu(V)/\lambda_\mathrm{gap}$ and 
$c_\rho=\|V-\E^\mu[V]\|_{L^\infty(\mu)}/\lambda_\mathrm{gap}$,
\begin{align}
  \frac{a^2}{2(\tilde\sigma_\rho^2+c_\rho a)}
  \approx
  \begin{cases}
    \dfrac{a^2}{2\tilde\sigma_\rho^2}, & a\ll \tilde\sigma_\rho^2/c_\rho
      \quad\text{(Gaussian)},
      \\
    \dfrac{a}{2c_\rho}, & a\gg \tilde\sigma_\rho^2/c_\rho
      \quad\text{(exponential)}.
  \end{cases}
\end{align}
For small deviations the bound is therefore sub-Gaussian with variance 
$\tilde\sigma_\rho^2/t$, while for large deviations it decays only exponentially
at a rate of $t/(2c_\rho)$. The crossover between the two regimes can be evaluated explicitly, and occurs at
$a^\ast=\tilde\sigma_\rho^2/c_\rho=2\mathrm{Var}_\mu(V)/\|V-\E^\mu[V]\|_{L^\infty(\mu)}$.
Interestingly, 
it is independent of both $t$ and the spectral gap $\lambda_{\mathrm{gap}}$ and 
is fixed only by the observable $V$.
\end{remark}

\subsection{Explicit bounds for the empirical current}
\label{sec:poincare_current}
Next, we derive explicit concentration bounds for the empirical current
$\overline{J}_t(U)$ that will take a sub-gamma and a simpler classical Bernstein form. 
As before, we begin by recalling the variational characterization of the principal eigenvalue
in terms of the Dirichlet form (Proposition~\ref{prop:Lambda_variational}),  for $f\in D(\mathcal{E})$,
\begin{align}
  \Lambda^J(k)
  = \sup\left\{
      k\langle \vs\cdot U,f^2\rangle_\mu
      + k^2\langle U\cdot DU,f^2\rangle_\mu
      - \mathcal{E}(f,f)
      \middle| \|f\|_{L^2(\mu)} = 1
    \right\}.
\end{align}
Bounding $\Lambda^J(k)$ now requires to control the linear and newly arising quadratic term against
the Dirichlet form $\mathcal{E}(f,f)$. 
The linear contribution $k\langle\vs\cdot U,f^2\rangle_\mu$ can be identified as the analogue of the 
density part $k\langle V,f^2\rangle_\mu$,
which, however, now reflects the (possibly) non-equilibrium character of the
dynamics and vanishes under detailed balance ($\vs=0$).
Notably, the novel quadratic contribution $k^2\langle U\cdot DU,f^2\rangle_\mu$ has \emph{no}
counterpart as it originates from the stochastic integral defining the current.
In the following, we treat the linear part as before and derive two different concentration bounds, depending
on how we treat the quadratic term.
Throughout we use the following shorthand notation
\begin{align}
W(x) \equiv \vs(x) \cdot U(x),
\qquad
Q(x) \equiv U(x)\cdot DU(x)\geq 0,
\end{align}
such that $\E^\mu[\overline{J}_t] = \int \vs\cdot U \dd\mu = \E^\mu[W]$.

\subsubsection{Approach 1}
\label{sec:approach1}
The first, more direct approach bounds the quadratic term uniformly,
$\langle Q,f^2\rangle_\mu\leq\|Q\|_{L^\infty(\mu)}$, and includes in a second step the resulting
$k^2\|Q\|_{L^\infty(\mu)}$ term into the variance proxy of the sub-gamma form, which in turn allows the computation of 
its Legendre transform in a closed form.

\begin{lemma}[Poincar\'e bound on $\Lambda^J(k)$ (Approach 1)]
\label{lem:Lambda_J_approach1}
Let Assumptions~\ref{ass:dynamics} hold, such that the symmetric part
$\mathcal{L}_S$ satisfies the Poincar\'e inequality~\eqref{eq:poincare} with
spectral gap $\lambda_\mathrm{gap}>0$. Then for all 
$0 \leq k < \lambda_\mathrm{gap}/\|W-\E^\mu[W]\|_{L^\infty(\mu)}$,
\begin{align}
\Lambda^J(k)
&\leq
k\E^\mu[W]
+ \frac{k^2\mathrm{Var}_\mu(W)}{\lambda_\mathrm{gap} - k\|W - \E^\mu[W]\|_{L^\infty(\mu)}}
+ k^2\|Q\|_{L^\infty(\mu)}
\\
&\leq
k\E^\mu[W]
+ \frac{k^2 \bigl(\mathrm{Var}_\mu(W) + \lambda_\mathrm{gap}\|Q\|_{L^\infty(\mu)}\bigr)}{\lambda_\mathrm{gap} - k \|W - \E^\mu[W]\|_{L^\infty(\mu)}}.
\label{eq:Lambda_J_bound1}
\end{align}
\end{lemma}

\begin{proof}
The Dirichlet representation of $\Lambda^J(k)$ reads, for $f\in D(\mathcal{E})$,
\begin{align}
  \Lambda^J(k)
  = \sup\left\{
      k\langle W,f^2\rangle_\mu
      + k^2\langle Q,f^2\rangle_\mu
      - \mathcal{E}(f,f)
      \middle| \|f\|_{L^2(\mu)} = 1
    \right\}.
    \label{eq:rq_J_proof}
\end{align}
The quadratic term is bounded directly by $\langle Q, f^2\rangle_\mu \leq
\|Q\|_{L^\infty(\mu)}$ since $\|f\|_{L^2(\mu)} = 1$. For the remaining two terms we
center the non-equilibrium term by introducing $\bar W \equiv W - \E^\mu[W]$ with
$\E^\mu[\bar W] = 0$. Therefore we find
$\langle W,f^2\rangle_\mu = \E^\mu[W] + \langle\bar W,f^2\rangle_\mu$, and use
the parametrization $f = (1 + \varepsilon g)/\sqrt{1+\varepsilon^2}$ with
$\varepsilon \geq 0$, $\E^\mu[g] = 0$, $\|g\|_{L^2(\mu)} = 1$, to obtain
\begin{align}
k\langle W, f^2\rangle_\mu
= k\E^\mu[W]
+ \frac{1}{1+\varepsilon^2}\left(2\varepsilon k\langle \bar W, g\rangle_\mu
  + \varepsilon^2 k\langle \bar W, g^2\rangle_\mu\right).
\end{align}
The centered Cauchy-Schwarz bound gives $\langle \bar W, g\rangle_\mu \leq
\|\bar W\|_{L^2(\mu)}\|g\|_{L^2(\mu)} = \sqrt{\mathrm{Var}_\mu(W)}$, while
$\langle \bar W, g^2\rangle_\mu \leq \|\bar W\|_{L^\infty(\mu)}$. Using the Poincar\'e
inequality $\mathcal{E}(f,f) \geq \lambda_\mathrm{gap}\Var_\mu(f)$
with $\Var_\mu(f)=\varepsilon^2/(1+\varepsilon^2)$, and
dropping the prefactor $1/(1+\varepsilon^2) \leq 1$, yields
\begin{align}
-\mathcal{E}(f,f) + k\langle W, f^2\rangle_\mu
\leq k\E^\mu[W]
+ \sup_{\varepsilon \geq 0}\left[
  2\varepsilon k\sqrt{\mathrm{Var}_\mu(W)}
  - \varepsilon^2\left(\lambda_\mathrm{gap} - k\|\bar W\|_{L^\infty(\mu)}\right)\right].
\end{align}
For $k\|\bar W\|_{L^\infty(\mu)} < \lambda_\mathrm{gap}$ 
the supremum over $\varepsilon$ is that of a concave quadratic polynomial 
\begin{align}
P^{J,1}(\varepsilon) 
= 
2\varepsilon k\sqrt{\mathrm{Var}_\mu(W)}
- \varepsilon^2\left(\lambda_\mathrm{gap} 
- k\|\bar W\|_{L^\infty(\mu)}\right),
\end{align}
and is maximized at
$\varepsilon_0 = k\sqrt{\mathrm{Var}_\mu(W)}/(\lambda_\mathrm{gap} - k\|\bar W\|_{L^\infty(\mu)})$.
Therefore, 
\begin{align}
-\mathcal{E}(f,f) + k\langle W, f^2\rangle_\mu
\leq k\E^\mu[W]
+ \frac{k^2\mathrm{Var}_\mu(W)}{\lambda_\mathrm{gap} - k\|\bar W\|_{L^\infty(\mu)}}.
\end{align}
Adding back the term $k^2\|Q\|_{L^\infty(\mu)}$ proves the first inequality
\begin{align}
\Lambda^J(k)
\leq
k\E^\mu[W]
+ \frac{k^2\mathrm{Var}_\mu(W)}{\lambda_\mathrm{gap} - k\|\bar W\|_{L^\infty(\mu)}}
+ k^2\|Q\|_{L^\infty(\mu)}.
\end{align}
Moreover, since $\lambda_\mathrm{gap}/(\lambda_\mathrm{gap} - k\|\bar W\|_{L^\infty(\mu)}) \geq 1$ for
the admissible range of $k$, the last term is bounded by
\begin{align}
k^2\|Q\|_{L^\infty(\mu)}
\leq
\frac{k^2\lambda_\mathrm{gap}\|Q\|_{L^\infty(\mu)}}
     {\lambda_\mathrm{gap} - k\|\bar W\|_{L^\infty(\mu)}}.
\end{align}
Combining the two expressions gives
\begin{align}
\Lambda^J(k)
\leq
k\E^\mu[W]
+ \frac{k^2\bigl(\mathrm{Var}_\mu(W) + \lambda_\mathrm{gap}\|Q\|_{L^\infty(\mu)}\bigr)}
       {\lambda_\mathrm{gap} - k\|\bar W\|_{L^\infty(\mu)}},
\end{align}
and completes the proof of the second inequality.
\end{proof}

\begin{theorem}[Explicit concentration bound for the empirical current (Approach 1)]
\label{thm:explicit_current_1}
Let $(X_t)_{t\geq 0}$ be an ergodic diffusion with invariant measure $\mu$ whose
symmetric part $\mathcal{L}_S$ satisfies the Poincar\'e inequality~\eqref{eq:poincare}
with spectral gap $\lambda_\mathrm{gap}>0$ (Assumptions~\ref{ass:dynamics}). 
Let $\overline{J}_t(U)$ denote the empirical current according to Definition~\ref{def:empirical_current}
with ergodic mean $\E^\mu[\overline{J}_t(U)]=\E^\mu[\vs\cdot U]= \E^\mu[W]$,
and let $\nu\ll\mu$ with $\dd\nu/\dd\mu\in L^2(\mu)$.
Then for all $t > 0$ and $a > 0$,
\begin{align}
  \mathbb{P}^\nu\left(\overline{J}_t(U) - \E^\mu[\overline{J}_t] \geq a\right)
  \leq
  \left\|\frac{\dd\nu}{\dd\mu}\right\|_{L^2(\mu)}
  \exp\left[-t\frac{\tilde\sigma_{J,1}^2}{c_{J,1}^2}h\left(\frac{c_{J,1} a}{\tilde\sigma_{J,1}^2}\right)\right],
  \label{eq:explicit_current_1}
\end{align}
where $h(u) \equiv 1 + u - \sqrt{1+2u}$ and we define
\begin{align}
  \tilde\sigma_{J,1}^2 \equiv
  \frac{2\mathrm{Var}_\mu(W)}{\lambda_\mathrm{gap}} + 2\|Q\|_{L^\infty(\mu)},
  \qquad 
  c_{J,1} \equiv \frac{\|W-\E^\mu[W]\|_{L^\infty(\mu)}}{\lambda_\mathrm{gap}}.
\end{align}
\end{theorem}

\begin{proof}
By Theorem~\ref{thm:conc_current} together with the Cram\'er transform~\eqref{eq:cramer_curr_def},
it suffices to bound
$I^J(\E^\mu[\overline{J}_t] + a) = \sup_{k>0}\bigl[k(\E^\mu[\overline{J}_t]+a) - \Lambda^J(k)\bigr]$
from below. The mean directly cancels against the linear term of the Poincar\'e
bound~\eqref{eq:Lambda_J_bound1} since $\E^\mu[\overline{J}_t]=\E^\mu[W]$.
With parameters $\tilde\sigma_{J,1}^2$ and $c_{J,1}$ from above,
we obtain the sub-gamma form, for $0 \leq k < 1/c_{J,1}$,
\begin{align}
  \Lambda^J(k) - k\E^\mu[\overline{J}_t]
  \leq \frac{k^2\bigl(\mathrm{Var}_\mu(W) + \lambda_\mathrm{gap}\|Q\|_{L^\infty(\mu)}\bigr)}{\lambda_\mathrm{gap} 
  - k\|W-\E^\mu[W]\|_{L^\infty(\mu)}}
  = \frac{\tilde\sigma_{J,1}^2 k^2}{2(1 - c_{J,1} k)},
\end{align}
so that
\begin{align}
I^J(\E^\mu[\overline{J}_t] + a) \geq \sup_{0 \leq k < 1/c_{J,1}}\left[ka -
\frac{\tilde\sigma_{J,1}^2 k^2}{2(1 - c_{J,1} k)}\right]
= \frac{\tilde\sigma_{J,1}^2}{c_{J,1}^2}h\left(\frac{c_{J,1} a}{\tilde\sigma_{J,1}^2}\right),
\end{align}
where $h(u) \equiv 1 + u - \sqrt{1 + 2u}$ and the last equality follows from
Lemma~\ref{lem:subgamma} with $(\tilde\sigma^2, c)$ replaced with
$(\tilde\sigma_{J,1}^2, c_{J,1})$.
Inserting the lower bound into the Cram\'er transform of
Theorem~\ref{thm:conc_current} completes the proof.
\end{proof}

\begin{corollary}[Bernstein bound for the empirical current]
\label{cor:bernstein_current_1}
Under the assumptions of Theorem~\ref{thm:explicit_current_1}, for all $t>0$ and
$a>0$,
\begin{align}
  \mathbb{P}^\nu\left(\overline{J}_t(U) - \E^\mu[\overline{J}_t] \geq a\right)
  \leq
  \left\|\frac{\dd\nu}{\dd\mu}\right\|_{L^2(\mu)}
  \exp\left(-\frac{ta^2}{2\bigl(\tilde\sigma_{J,1}^2 + c_{J,1}a\bigr)}\right).
  \label{eq:bernstein_current_1}
\end{align}
\end{corollary}

\begin{proof}
By Theorem~\ref{thm:explicit_current_1} and Lemma~\ref{lem:subgamma}, the
inequality~\eqref{eq:subgamma_bernstein} with the substitution
$(\tilde\sigma^2, c) = (\tilde\sigma_{J,1}^2, c_{J,1})$ gives a lower bound on the exponent
of Eq.~\eqref{eq:explicit_current_1} by $a^2/[2(\tilde\sigma_{J,1}^2 + c_{J,1} a)]$,
which gives Eq.~\eqref{eq:bernstein_current_1}.
\end{proof}

\subsubsection{Approach 2}
\label{sec:approach2}
In the first approach we bound the quadratic term uniformly,
$\langle Q, f^2\rangle_\mu \leq \|Q\|_{L^\infty(\mu)}$, 
which is tight only when $Q = U\cdot DU$ is essentially constant.
In the second approach we now center the quadratic term,
$\langle Q, f^2\rangle_\mu = \E^\mu[Q]
+ \langle Q - \E^\mu[Q], f^2\rangle_\mu$, and therefore extract the mean
$\E^\mu[Q]$
and control 
$Q - \E^\mu[Q]$ through the Poincar\'e inequality exactly as for the
linear term of $W$. As before this yields a sub-gamma bound on $\Lambda^J(k)$, however, now with
a variance proxy that includes the mean $\E^\mu[Q]$ in place of $\|Q\|_{L^\infty(\mu)}$. 

\begin{lemma}[Poincar\'e bound on $\Lambda^J$ (Approach 2)]
\label{lem:Lambda_J_approach2}
Let Assumptions~\ref{ass:dynamics} hold, such that the symmetric part
$\mathcal{L}_S$ satisfies the Poincar\'e inequality~\eqref{eq:poincare} with
spectral gap $\lambda_\mathrm{gap}>0$. 
Then for all $k\geq0$ with
$k \|W-\E^\mu[W]\|_{L^\infty(\mu)}
 + k^2\|Q-\E^\mu[Q]\|_{L^\infty(\mu)} < \lambda_\mathrm{gap}$,
\begin{align}
\Lambda^J(k)
\leq
k \E^\mu[W]
+ k^2 \E^\mu[Q]
+
\frac{k^2 \bigl(\sqrt{\mathrm{Var}_\mu(W)} 
+ k\sqrt{\mathrm{Var}_\mu(Q)}\bigr)^2}
{\lambda_\mathrm{gap} - k\|W - \E^\mu[W]\|_{L^\infty(\mu)} - k^2\|Q - \E^\mu[Q]\|_{L^\infty(\mu)}}.
\label{eq:Lambda_J_bound2}
\end{align}
\end{lemma}

\begin{proof}
We start with the Dirichlet representation, $f\in D(\mathcal{E})$,
\begin{align}
  \Lambda^J(k)
  = \sup\left\{
      k\langle W,f^2\rangle_\mu
      + k^2\langle Q,f^2\rangle_\mu
      - \mathcal{E}(f,f)
      \middle| \|f\|_{L^2(\mu)} = 1
    \right\}.
\end{align}
We now center both contributions, $\bar W = W - \E^\mu[W]$ and
$\bar Q = Q - \E^\mu[Q]$, and use the parametrization
$f = (1+\varepsilon g)/\sqrt{1+\varepsilon^2}$ with $\varepsilon \geq 0$,
$\E^\mu[g] = 0$, $\|g\|_{L^2(\mu)} = 1$, yielding
\begin{align}
k\langle W,f^2\rangle_\mu
&= k\E^\mu[W]
+ \frac{1}{1+\varepsilon^2}\left[
  2\varepsilon k\langle \bar W,g\rangle_\mu
  + \varepsilon^2 k\langle \bar W,g^2\rangle_\mu
\right],
\\
k^2\langle Q,f^2\rangle_\mu
&= k^2\E^\mu[Q]
+ \frac{1}{1+\varepsilon^2}\left[
  2\varepsilon k^2\langle \bar Q,g\rangle_\mu
  + \varepsilon^2 k^2\langle \bar Q,g^2\rangle_\mu
\right].
\end{align}
The Cauchy-Schwarz inequality gives
$\langle \bar W,g\rangle_\mu \leq \sqrt{\mathrm{Var}_\mu(W)}$ and
$\langle \bar Q,g\rangle_\mu \leq \sqrt{\mathrm{Var}_\mu(Q)}$, while
$\langle \bar W,g^2\rangle_\mu \leq \|\bar W\|_{L^\infty(\mu)}$ and
$\langle \bar Q,g^2\rangle_\mu\leq \|\bar Q\|_{L^\infty(\mu)}$. 
Using the Poincar\'e
inequality $\mathcal{E}(f,f) \geq \lambda_\mathrm{gap}\Var_\mu(f)$
with $\Var_\mu(f)=\varepsilon^2/(1+\varepsilon^2)$, and subsequently
dropping the prefactor $1/(1+\varepsilon^2)\leq 1$, yields
\begin{align}
\Lambda^J(k)
\leq
k\E^\mu[W] + k^2\E^\mu[Q]
+ \sup_{\varepsilon \geq 0} P^{J,2}(\varepsilon),
\end{align}
where we introduce the quadratic polynomial
\begin{align}
P^{J,2}(\varepsilon)\equiv
2\varepsilon\left(k\sqrt{\mathrm{Var}_\mu(W)} + k^2\sqrt{\mathrm{Var}_\mu(Q)}\right)
- \varepsilon^2\left(\lambda_\mathrm{gap} - k\|\bar W\|_{L^\infty(\mu)}
  - k^2\|\bar Q\|_{L^\infty(\mu)}\right).
\end{align}
For $k\|\bar W\|_{L^\infty(\mu)} + k^2\|\bar Q\|_{L^\infty(\mu)} < \lambda_\mathrm{gap}$ the
polynomial is concave and attains its maximum at
\begin{align}
\varepsilon_{0}
= 
\frac{k\sqrt{\mathrm{Var}_\mu(W)} + k^2\sqrt{\mathrm{Var}_\mu(Q)}}
       {\lambda_\mathrm{gap} - k\|\bar W\|_{L^\infty(\mu)} - k^2\|\bar Q\|_{L^\infty(\mu)}},
\end{align}
with a value of
\begin{align}
P^{J,2}(\varepsilon_0)
= \frac{k^2\left(\sqrt{\mathrm{Var}_\mu(W)} 
+ k\sqrt{\mathrm{Var}_\mu(Q)}\right)^2}{\lambda_\mathrm{gap} - k\|\bar W\|_{L^\infty(\mu)} - k^2\|\bar Q\|_{L^\infty(\mu)}}.
\end{align}
Adding the mean terms gives the claimed bound and completes the proof.
\end{proof}

Unfortunately, the Legendre transform of Eq.~\eqref{eq:Lambda_J_bound2} cannot be evaluated 
in closed form.
For the optimization over $k$ the denominator (quadratic in $k$) together with the
$\sqrt{\mathrm{Var}_\mu(Q)}$ term yields a quintic
equation for the optimal $k^\ast$ which has no closed-form solution for generic parameters.
Our strategy is therefore to relax Eq.~\eqref{eq:Lambda_J_bound2} by a further bound
that recovers the sub-gamma form, whose Cram\'er transform is explicit.

\begin{lemma}[Sub-gamma bound]
\label{lem:bernstein_current}
Let $\bar W = W - \E^\mu[W]$ and
$\bar Q = Q - \E^\mu[Q]$.
Let $k^\ast_{\max}$ be the positive root of
$k(\|\bar W\|_{L^\infty(\mu)} + k\|\bar Q\|_{L^\infty(\mu)}) = \lambda_\mathrm{gap}$,
\begin{align}
k^\ast_{\max}
=
\begin{cases}
\dfrac{-\|\bar W\|_{L^\infty(\mu)} + \sqrt{\|\bar W\|_{L^\infty(\mu)}^2 + 4\lambda_\mathrm{gap}\|\bar Q\|_{L^\infty(\mu)}}}{2\|\bar Q\|_{L^\infty(\mu)}}
& \|\bar Q\|_{L^\infty(\mu)} > 0, 
\\[6pt]
\dfrac{\lambda_\mathrm{gap}}{\|\bar W\|_{L^\infty(\mu)}}
& \|\bar Q\|_{L^\infty(\mu)} = 0,
\end{cases}
\end{align}
and define
\begin{align}
c_{J,2} \equiv \frac{1}{k^\ast_{\max}},
\qquad
\tilde\sigma_{J,2}^2 \equiv 2\E^\mu[Q]
+ \frac{2}{\lambda_\mathrm{gap}}
  \left(\sqrt{\mathrm{Var}_\mu(W)} + k^\ast_{\max}\sqrt{\mathrm{Var}_\mu(Q)}\right)^2.
\end{align}
Then for all $k \in [0, k^\ast_{\max})$,
\begin{align}
\Lambda^J(k) - k\E^\mu[W]
\leq
\frac{\tilde\sigma_{J,2}^2 k^2}{2(1 - c_{J,2}k)}.
\label{eq:bernstein_current}
\end{align}
\end{lemma}

\begin{proof}
Starting from Lemma~\ref{lem:Lambda_J_approach2}, we have
\begin{align}
\Lambda^J(k) - k\E^\mu[W]
\leq
k^2 \E^\mu[Q]
+ \frac{k^2(\sqrt{\mathrm{Var}_\mu(W)} + k\sqrt{\mathrm{Var}_\mu(Q)})^2}
{\lambda_{\mathrm{gap}} - k\|\bar W\|_{L^\infty(\mu)} - k^2 \|\bar Q\|_{L^\infty(\mu)}}.
\label{eq:start_bernstein}
\end{align}
With $c_{J,2} \equiv 1/k^\ast_{\max}$ and $0\leq k<k^\ast_{\max}$ we have
$0 \leq k c_{J,2} < 1$. Multiplying both sides of Eq.~\eqref{eq:start_bernstein} by
$(1 - k c_{J,2})/k^2 \geq 0$ gives
\begin{align}
\frac{(\Lambda^J(k) - k\E^\mu[W])(1-k c_{J,2})}{k^2}
&\leq
\E^\mu[Q](1-k c_{J,2})
\nonumber
\\
&\quad + \frac{(\sqrt{\mathrm{Var}_\mu(W)} + k\sqrt{\mathrm{Var}_\mu(Q)})^2(1-k c_{J,2})}
{\lambda_{\mathrm{gap}} - k\|\bar W\|_{L^\infty(\mu)} - k^2 \|\bar Q\|_{L^\infty(\mu)}}.
\label{eq:step2}
\end{align}
Next, we bound the two terms of the right-hand side separately. 
For the first term, $k c_{J,2} \geq 0$ directly gives
\begin{align}
\E^\mu[Q](1 - k c_{J,2}) \leq \E^\mu[Q].
\end{align}
For the second term, we claim that
for all $k \in [0, k^\ast_{\max})$,
\begin{align}
\frac{1 - k c_{J,2}}
     {\lambda_{\mathrm{gap}} - k\|\bar W\|_{L^\infty(\mu)} - k^2 \|\bar Q\|_{L^\infty(\mu)}}
\leq \frac{1}{\lambda_{\mathrm{gap}}}.
\label{eq:key_ineq}
\end{align}
By cross multiplying (both denominators are positive for $k<k^\ast_{\max}$) and dividing by
$k>0$, Eq.~\eqref{eq:key_ineq} is equivalent to stating that
$\lambda_{\mathrm{gap}} c_{J,2} \geq \|\bar W\|_{L^\infty(\mu)} + k\|\bar Q\|_{L^\infty(\mu)}$,
which holds since, by the definition of $k^\ast_{\max}$,
\begin{align}
\lambda_{\mathrm{gap}} c_{J,2}
= \frac{\lambda_{\mathrm{gap}}}{k^\ast_{\max}}
= \|\bar W\|_{L^\infty(\mu)} + k^\ast_{\max} \|\bar Q\|_{L^\infty(\mu)}
\geq \|\bar W\|_{L^\infty(\mu)} 
+ k\|\bar Q\|_{L^\infty(\mu)}.
\end{align}
Moreover, for the second term $k \leq k^\ast_{\max}$ implies
\begin{align}
\sqrt{\mathrm{Var}_\mu(W)} + k\sqrt{\mathrm{Var}_\mu(Q)}
\leq \sqrt{\mathrm{Var}_\mu(W)} + k^\ast_{\max}\sqrt{\mathrm{Var}_\mu(Q)}.
\end{align}
In combination with Eq.~\eqref{eq:key_ineq} and $1-k c_{J,2}\leq 1$ we obtain
\begin{align}
\frac{(\sqrt{\mathrm{Var}_\mu(W)} + k\sqrt{\mathrm{Var}_\mu(Q)})^2(1-k c_{J,2})}
     {\lambda_{\mathrm{gap}} - k\|\bar W\|_{L^\infty(\mu)} - k^2 \|\bar Q\|_{L^\infty(\mu)}}
\leq \frac{\left(\sqrt{\mathrm{Var}_\mu(W)} + k^\ast_{\max}\sqrt{\mathrm{Var}_\mu(Q)} \right)^2}{\lambda_{\mathrm{gap}}}.
\end{align}
Inserting both bounds into Eq.~\eqref{eq:step2} and using the definition of
$\tilde\sigma_{J,2}^2$,
\begin{align}
\frac{(\Lambda^J(k) - k\E^\mu[W])(1-k c_{J,2})}{k^2}
\leq \E^\mu[Q] + \frac{\left(\sqrt{\mathrm{Var}_\mu(W)} + k^\ast_{\max}\sqrt{\mathrm{Var}_\mu(Q)}\right)^2}{\lambda_{\mathrm{gap}}}
= \frac{\tilde\sigma_{J,2}^2}{2},
\end{align}
which gives Eq.~\eqref{eq:bernstein_current} after re-arranging and completes the proof.
\end{proof}

\begin{theorem}[Explicit concentration bound for the empirical current (Approach 2)]
\label{thm:explicit_current_2}
Let $(X_t)_{t\geq 0}$ be an ergodic diffusion with invariant measure $\mu$ whose
symmetric part $\mathcal{L}_S$ satisfies the Poincar\'e inequality~\eqref{eq:poincare}
with spectral gap $\lambda_\mathrm{gap}>0$ (Assumptions~\ref{ass:dynamics}). 
Let $\overline{J}_t(U)$ denote the empirical current according to Definition~\ref{def:empirical_current}
with ergodic mean $\E^\mu[\overline{J}_t(U)]=\E^\mu[\vs\cdot U]= \E^\mu[W]$,
and let $\nu\ll\mu$ with $\dd\nu/\dd\mu\in L^2(\mu)$.
Then for all $t > 0$ and $a > 0$,
\begin{align}
\mathbb{P}^\nu\left(\overline{J}_t(U) - \E^\mu[\overline{J}_t] \geq a\right)
\leq
\left\|\frac{\dd\nu}{\dd\mu}\right\|_{L^2(\mu)}
\exp\left[-t \frac{\tilde\sigma_{J,2}^2}{c_{J,2}^2}
  h\left(\frac{c_{J,2} a}{\tilde\sigma_{J,2}^2}\right)\right],
\label{eq:explicit_current_2}
\end{align}
where $h(u) = 1 + u - \sqrt{1+2u}$ and the parameters 
$\tilde\sigma_{J,2}^2$ and $c_{J,2}$ as in
Lemma~\ref{lem:bernstein_current}.
\end{theorem}

\begin{proof}
By Theorem~\ref{thm:conc_current} it suffices to prove a lower bound on the Cram\'er transform
\begin{align}
I^J(\E^\mu[\overline{J}_t] + a)
= \sup_{k>0}\bigl[k(\E^\mu[W]+a) - \Lambda^J(k)\bigr]
= \sup_{k>0}\bigl[ka - (\Lambda^J(k) - k\E^\mu[W])\bigr],
\end{align}
where the mean cancels against the linear term
since $\E^\mu[\overline{J}_t]=\E^\mu[W]$.
By Lemma~\ref{lem:bernstein_current}, for $0 \leq k < 1/c_{J,2}$,
\begin{align}
\Lambda^J(k) - k\E^\mu[W]
\leq \frac{\tilde\sigma_{J,2}^2 k^2}{2(1 - c_{J,2} k)}.
\end{align}
By Lemma~\ref{lem:subgamma} where we replace $(\tilde\sigma^2, c)$ by
$(\tilde\sigma_{J,2}^2, c_{J,2})$, we find
\begin{align}
I^J(\E^\mu[\overline{J}_t] + a)
\geq \sup_{0 \leq k < 1/c_{J,2}}\left[ka -
\frac{\tilde\sigma_{J,2}^2 k^2}{2(1 - c_{J,2} k)}\right]
= \frac{\tilde\sigma_{J,2}^2}{c_{J,2}^2}h \left(\frac{c_{J,2} a}{\tilde\sigma_{J,2}^2}\right).
\end{align}
Inserting this into the Cram\'er transform of Theorem~\ref{thm:conc_current} completes
the proof.
\end{proof}

\subsubsection{Discussion}
We close the section by briefly discussing the two different sub-gamma bounds
and stating the corresponding left-tail and two-sided concentration bounds.

\begin{remark}[Gaussian and exponential regimes]
\label{rem:subgamma_current}
The exponent of the simple Bernstein expression interpolates between a Gaussian and exponential
regime, depending on which term in $2(\tilde{\sigma}_J^2 +  a c_J)$ dominates, 
where we write $\tilde\sigma_J^2$ and $c_J$ for the parameters of either approach.
In particular, we find
\begin{align}
\frac{a^2}{2(\tilde\sigma_J^2 + c_J a)}
\approx
\begin{cases}
\dfrac{a^2}{2\tilde\sigma_J^2}, & a\ll \tilde\sigma_J^2/c_J \quad\text{(Gaussian)},
\\
\dfrac{a}{2c_J}, & a\gg \tilde\sigma_J^2/c_J \quad\text{(exponential)},
\end{cases}
\end{align}
meaning that $\overline{J}_t$ behaves 
sub-Gaussian with variance $\tilde\sigma_J^2/t$ for deviations smaller than
$a^\ast = \tilde\sigma_J^2/c_J$ and exponential with rate parameter $t/2c_J$ for larger deviations.
Explicitly, for Approach~1 the crossover
reads
\begin{align}
a^\ast_1 = \frac{\tilde\sigma_{J,1}^2}{c_{J,1}}
= \frac{2\mathrm{Var}_\mu(W)}{\|\bar W\|_{L^\infty(\mu)}}
+ \frac{2\lambda_\mathrm{gap}\|Q\|_{L^\infty(\mu)}}{\|\bar W\|_{L^\infty(\mu)}}.
\end{align} 
For Approach~2, using $k^\ast_{\max}$ as defined in Lemma~\ref{lem:bernstein_current} via
$\lambda_\mathrm{gap} = k^\ast_{\max}\bigl(\|\bar W\|_{L^\infty(\mu)} + k^\ast_{\max}\|\bar Q\|_{L^\infty(\mu)}\bigr)$, we obtain
\begin{align}
a^\ast_2 = \frac{\tilde\sigma_{J,2}^2}{c_{J,2}}
= 2k^\ast_{\max}\E^\mu[Q]
+ \frac{2\bigl(\sqrt{\mathrm{Var}_\mu(W)} + k^\ast_{\max}\sqrt{\mathrm{Var}_\mu(Q)}\bigr)^2}
       {\|\bar W\|_{L^\infty(\mu)} + k^\ast_{\max}\|\bar Q\|_{L^\infty(\mu)}}.
\end{align}
Notably, while in the density case the crossover $a^\ast = 2\mathrm{Var}_\mu(V)/\|\bar V\|_{L^\infty(\mu)}$ 
was independent of the spectral gap, both results for the current now depend on 
the symmetric part of the underlying dynamics via $\lambda_\mathrm{gap}$.
\end{remark}

\begin{remark}[Comparison of the two approaches]
\label{rem:approach_comparison}
We start with comparing the variance proxy that is responsible
for the sub-Gaussian behavior for small deviations.
The difference between the two approaches may be expressed as
\begin{align}
\tilde\sigma_{J,2}^2 - \tilde\sigma_{J,1}^2
= 2\bigl(\E^\mu[Q] - \|Q\|_{L^\infty(\mu)}\bigr)
+ \frac{2}{\lambda_\mathrm{gap}}\Bigl[2k^\ast_{\max}\sqrt{\mathrm{Var}_\mu(W)\mathrm{Var}_\mu(Q)}
  + (k^\ast_{\max})^2\mathrm{Var}_\mu(Q)\Bigr].
\end{align}
This means that the improvement we make by replacing
the worst case $\|Q\|_{L^\infty(\mu)}$ for the average $\E^\mu[Q]$ in the second approach,
only leads to a sharper Gaussian regime when it compensates the variance contributions
from the second term.
A comparison between the respective exponential scale parameters gives directly, by the definition of $k^\ast_{\max}$, 
\begin{align}
c_{J,2} = \frac{1}{k^\ast_{\max}}
= \frac{\|\bar W\|_{L^\infty(\mu)} + k^\ast_{\max}\|\bar Q\|_{L^\infty(\mu)}}{\lambda_\mathrm{gap}}
\geq \frac{\|\bar W\|_{L^\infty(\mu)}}{\lambda_\mathrm{gap}} = c_{J,1},
\end{align}
where equality is achieved if and only if $\|\bar Q\|_{L^\infty(\mu)} = 0$.
In approach 2, the smaller optimization range in $k$, due to the introduction of $k^\ast_{\max}$, 
therefore yields a larger scale parameter for the exponential decay and hence a weaker bound
for large deviations.
Moreover, since $c_{J,2}\geq c_{J,1}$,
whenever Approach~2 improves the variance proxy the crossover is the smaller
(Remark~\ref{rem:subgamma_current}),
$a^\ast_2 \leq a^\ast_1$, i.e., the narrower Gaussian
range comes with an earlier onset of the exponential tail.
\end{remark}

\begin{remark}[Concentration bounds at detailed balance.]
\label{rem:db_bounds}
Under detailed balance the stationary current vanishes,
$\js\equiv 0$, hence $\vs\equiv 0$ 
and correspondingly also $W = \vs\cdot U \equiv 0$. 
Consequently, we further have $\E^\mu[W] = 0$,
$\mathrm{Var}_\mu(W) = 0$, and $\|\bar W\|_{L^\infty(\mu)} = 0$, i.e.,
any current-type functional (irrespective
of the choice of $U$) has zero mean, and fluctuations of $\overline{J}_t(U)$
are purely driven by the diffusive term $Q = U\cdot DU \geq 0$. 

For the first approach, $\|\bar W\|_{L^\infty(\mu)} = 0$ directly
implies that the exponential parameter 
vanishes, $c_{J,1} = 0$, and the sub-gamma form reduces to a simpler
quadratic structure,
\begin{align}
  \Lambda^J(k) \leq \|Q\|_{L^\infty(\mu)} k^2 
  = \frac{1}{2}\tilde\sigma_{J,1}^2k^2,
  \qquad
  \tilde\sigma_{J,1}^2 = 2\|Q\|_{L^\infty(\mu)}.
\end{align}
The corresponding 
Legendre transform thus gives the canonical Gaussian expression 
$a^2/(2\tilde\sigma_{J,1}^2)$ such that, with
$\E^\mu[\overline{J}_t] = 0$, 
Theorem~\ref{thm:explicit_current_1} yields the purely sub-Gaussian bound
under detailed balance
\begin{align}
  \mathbb{P}^\nu\left(\overline{J}_t(U) - \E^\mu[\overline{J}_t] \geq a\right)
  \leq
  \left\|\frac{\dd\nu}{\dd\mu}\right\|_{L^2(\mu)}
  \exp\left(-\frac{t a^2}{4\|Q\|_{L^\infty(\mu)}}\right),
\end{align}
valid for all $t > 0$ and $a > 0$.
This recovers the results from 
the general implicit bound in Corollary~\ref{cor:db_subgaussian}
and highlights that, within our bounds and in the absence of dissipation,
current fluctuations have \emph{no}
heavy tail and concentrate at a 
rate set by the diffusive sub-Gaussian fluctuations only.
Tails heavier than Gaussian are therefore a manifestation of strictly 
non-equilibrium (i.e., irreversible) effects.

For the second approach,
centering the quadratic term
allows us to account for the potentially heterogeneity of $Q$,
which, however, introduces an artifact that does not vanish under detailed balance.
For a non-constant $Q$, i.e., when $\|\bar Q\|_{L^\infty(\mu)} > 0$,
the admissible $k$ range stays finite even when we have $W \equiv 0$,
\begin{align}
  k^\ast_{\max} = \sqrt{\frac{\lambda_\mathrm{gap}}{\|\bar Q \|_{L^\infty(\mu)}}},
  \qquad
  c_{J,2} = \sqrt{\frac{\|\bar Q\|_{L^\infty(\mu)}}{\lambda_\mathrm{gap}}}> 0,
  \qquad
  \tilde\sigma_{J,2}^2 = 2 \E^\mu[Q]
    + \frac{2\mathrm{Var}_\mu(Q)}{\|\bar Q\|_{L^\infty(\mu)}}.
\end{align}
This means that the corresponding concentration
bound keeps it exponential regime with tails that are heavier than those of a 
Gaussian with parameter $c_{J,2} > 0$ even though 
the true fluctuations are sub-Gaussian.

Notably, the difference between the two approaches
disappears whenever $Q$ is constant (e.g., $U$ is constant)
since then $\|\bar Q\|_{L^\infty(\mu)} = 0$ and $c_{J,2} = 0$, and the two approaches coincide,
$\tilde\sigma_{J,1}^2 = \tilde\sigma_{J,2}^2 = 2\E^\mu[Q] = 2\|Q\|_{L^\infty(\mu)}$. Taken together, Approach~1 thus gives the correct sub-Gaussian behavior near
equilibrium by controlling
the diffusion contribution via a uniform norm
$\|Q\|_{L^\infty(\mu)}$, while
Approach~2 replaces $\|Q\|_{L^\infty(\mu)}$ with 
the average $\E^\mu[Q]$ in the corresponding
variance proxy and thus is better when the diffusive noise 
is strongly heterogeneous, 
$\E^\mu[Q] \ll \|Q\|_{L^\infty(\mu)}$.
This improvement comes at the cost of a worse 
exponential regime $c_{J,2} \geq c_{J,1}$ (Remark~\ref{rem:approach_comparison}),
and does not recover the fact (unless $Q$ is constant) that
detailed balance fluctuations are sub-Gaussian at all deviation scales $a$.
\end{remark}

\begin{corollary}[Left-tail and two-sided bounds for current observables]
\label{cor:two_sided_current_explicit}
Let $i\in\{1,2\}$ and adopt the assumptions of Theorem~\ref{thm:explicit_current_1}
($i=1$) or Theorem~\ref{thm:explicit_current_2} ($i=2$). 
Since $Q$, $\mathrm{Var}_\mu(W)$
and $\|\bar W\|_{L^\infty(\mu)}$ are invariant under sign reversal 
of $\overline{J}_t(U)\mapsto -\overline{J}_t(U)$ (hence $U\mapsto -U$)
the parameters
$\tilde\sigma_{J,i}^2, c_{J,i}$ remain unchanged.
Applying the theorems to 
$-\overline{J}_t(U)$ therefore leads to bounds on the left tail with the same constants.
For all $t>0$, $a>0$,
\begin{align}
\mathbb{P}^\nu\left(\E^\mu[\overline{J}_t] -\overline{J}_t(U)\geq a\right)
\leq
\left\|\frac{\dd\nu}{\dd\mu}\right\|_{L^2(\mu)}
\exp\left[-t\frac{\tilde\sigma_{J,i}^2}{c_{J,i}^2}
  h\left(\frac{c_{J,i} a}{\tilde\sigma_{J,i}^2}\right)\right].
\end{align}
Correspondingly, a union bound correspondingly gives the two-sided sub-gamma bound, and its simpler Bernstein
form with the additional factor of $2$, i.e.,
\begin{align}
\mathbb{P}^\nu\left(\bigl|\overline{J}_t(U) - \E^\mu[\overline{J}_t]\bigr| \geq a\right)
&\leq
2\left\|\frac{\dd\nu}{\dd\mu}\right\|_{L^2(\mu)}
\exp\left[-t\frac{\tilde\sigma_{J,i}^2}{c_{J,i}^2}
  h\left(\frac{c_{J,i} a}{\tilde\sigma_{J,i}^2}\right)\right]
\\
&\leq
2\left\|\frac{\dd\nu}{\dd\mu}\right\|_{L^2(\mu)}
\exp\left[-\frac{ta^2}{2(\tilde\sigma_{J,i}^2 + c_{J,i}a)}\right].
\end{align}
\end{corollary}

\subsection{Finite-time upper bounds on variances}
\label{sec:variance_bounds}
The Poincar\'e bounds on $\Lambda^A(k)$ 
in Lemmas~\ref{lem:poincare_density}, \ref{lem:Lambda_J_approach1} , and~\ref{lem:Lambda_J_approach2}
do not only control the tails via the Chernoff bound, but, as they bound the respective log-moment generating functions,
they also bound the stationary variance of the additive 
functionals 
(see also \cite{bakewell2023general,bakewell2025bounds}). 
The variance is indeed controlled via the second derivative of $\Lambda^A(k)$ evaluated
at $k=0$, and the $k^2$ term therefore yields an explicit upper bound. 
Importantly, this can be achieved by using the full expression 
rather than working with the simplified sub-gamma form.

\begin{lemma}[Variance bound from log-moment generating functions]
\label{lem:variance_from_growth}
Let $A\in\{\rho,J\}$ and suppose it holds that
$\Lambda^A(k)- k\E^\mu[\overline A_t] \leq\phi(k)$ for $k\in[0,k_+)$ for $\phi$ twice
differentiable at $0$ with $\phi(0)=\phi'(0)=0$.
Then for all $t>0$,
\begin{align}
  \mathrm{Var}_\mu(\overline A_t)\leq\frac{\phi''(0)}{t}.
\end{align}
\end{lemma}

\begin{proof}
We start with the log-moment generating function of a centered functional $A_t$, 
$\psi(k)\equiv\log\E^\mu\bigl[\e^{k(A_t-\E^\mu[A_t])}\bigr]$, which is finite near $k=0$ 
and it holds that $\psi(0)=\psi'(0)=0$.
In particular, the stationary variance can be immediately identified as 
$\psi''(0)=\mathrm{Var}_\mu(A_t)=t^2\mathrm{Var}_\mu(\overline A_t)$. 
By Corollary~\ref{cor:mgf_bound} and taking $\nu=\mu$
it holds that 
$\log\E^\mu\bigl[\e^{kA_t}\bigr]\leq t\Lambda^A(k)$.
Since $\E^\mu[A_t]=t \E^\mu[\overline A_t]$ we further have
$\psi(k)\leq t\bigl[\Lambda^A(k)-\E^\mu[\overline A_t] k\bigr]\leq t\phi(k)$ on the range 
$[0,k_+)$ where the upper bound is valid. 
A comparison of the Taylor expansion around $0$ for $\psi$ and $\phi$
gives $\psi''(0)\leq t\phi''(0)$, i.e., $t^2\mathrm{Var}_\mu(\overline A_t)\leq t\phi''(0)$
and completes the proof. 
\end{proof}

\begin{proposition}[Finite-time variance upper bounds]
\label{prop:variance_explicit}
Consider the empirical density $\overline{\rho}_t(V)$ (Definition~\ref{def:empirical_density})
and empirical current $\overline J_t(U)$ (Definition~\ref{def:empirical_current})
and let Assumptions~\ref{ass:dynamics} hold.
Then for all $t>0$,
\begin{align}
  \mathrm{Var}_\mu(\overline\rho_t)
    &\leq \frac{2\mathrm{Var}_\mu(V)}{\lambda_\mathrm{gap}t},
  \label{eq:var_density}\\
  \mathrm{Var}_\mu(\overline J_t)
    &\leq \frac{1}{t}\left(2\E^\mu[Q]
      +\frac{2\mathrm{Var}_\mu(W)}{\lambda_\mathrm{gap}}\right)
      \leq \frac{1}{t}\left(2\|Q\|_{L^\infty(\mu)}
      +\frac{2\mathrm{Var}_\mu(W)}{\lambda_\mathrm{gap}}\right),
  \label{eq:var_current}
\end{align}
where $\lambda_\mathrm{gap}$ denotes the spectral gap~\eqref{eq:gap_variational} of $\mathcal{L}_S$
and $W=\vs\cdot U$ and $Q=U\cdot DU$ as before.
The corresponding asymptotic variances, 
defined by $\sigma_{A,\infty}^2\equiv\lim_{t\to\infty}t\mathrm{Var}_\mu(\overline A_t)$,
are bounded by the same constants, i.e., 
\begin{align}
  \sigma_{\rho,\infty}^2\leq\frac{2\mathrm{Var}_\mu(V)}{\lambda_\mathrm{gap}},
	\quad\quad
  \sigma_{J,\infty}^2\leq 2\E^\mu[Q]+\frac{2\mathrm{Var}_\mu(W)}{\lambda_\mathrm{gap}}
  \leq 2\|Q\|_{L^\infty(\mu)}+\frac{2\mathrm{Var}_\mu(W)}{\lambda_\mathrm{gap}}.
  \label{eq:asymptotic_variance}
\end{align}
\end{proposition}
\begin{proof}
The variance bounds follow directly from Lemma~\ref{lem:variance_from_growth}
and using the (now centered) Poincar\'e bounds of
Lemmas~\ref{lem:poincare_density},~\ref{lem:Lambda_J_approach1}, and~\ref{lem:Lambda_J_approach2}.
In particular, it holds that  $\Lambda^A(k)-k\E^\mu[\overline A_t]\le\phi(k)$
for the respective admissible range of $k$ set by the parameters and $\phi(0)=\phi'(0)=0$.

Using $\E^\mu[\overline\rho_t]=\E^\mu[V]$ for density observables, and
$\E^\mu[\overline J_t]=\E^\mu[W]$ for current observables, we identify
\begin{align}
\phi_\rho(k)&\equiv \frac{k^2\mathrm{Var}_\mu(V)}{\lambda_\mathrm{gap}-k\|\bar V\|_{L^\infty(\mu)}} 
\geq \Lambda^\rho(k)-k\E^\mu[V],
\\
\phi_{J,1}(k)&\equiv \frac{k^2\mathrm{Var}_\mu(W)}{\lambda_\mathrm{gap}-k\|\bar W\|_{L^\infty(\mu)}}
      +k^2\|Q\|_{L^\infty(\mu)}
\geq \Lambda^J(k)-k\E^\mu[W],
\\
\phi_{J,2}(k)&\equiv k^2\E^\mu[Q]
     +\frac{k^2\bigl(\sqrt{\mathrm{Var}_\mu(W)}+k\sqrt{\mathrm{Var}_\mu(Q)}\bigr)^2}
           {\lambda_\mathrm{gap}-k\|\bar W\|_{L^\infty(\mu)}-k^2\|\bar Q\|_{L^\infty(\mu)}}
\geq \Lambda^J(k)-k\E^\mu[W].
\end{align}
Taking the second derivative and evaluation at $k=0$ yields
\begin{align}
\phi_\rho''(0)
&=2 \frac{\mathrm{Var}_\mu(V)}{\lambda_\mathrm{gap}},
\\
\phi_{J,1}''(0)
&=2 \|Q\|_{L^\infty(\mu)}+2 \frac{\mathrm{Var}_\mu(W)}{\lambda_\mathrm{gap}},
\\
\phi_{J,2}''(0)
&=  2\E^\mu[Q]+2 \frac{\mathrm{Var}_\mu(W)}{\lambda_\mathrm{gap}}.
\end{align}
Since $\E^\mu[Q]\le\|Q\|_{L^\infty(\mu)}$, plugging these results into Lemma~\ref{lem:variance_from_growth}
proves the variance upper bounds and the corresponding asymptotic bounds follow by taking $t\to\infty$.
\end{proof}

\section{Application: Non-asymptotic uncertainty quantification}
\label{sec:uq}
The concentration inequalities of
Sections~\ref{sec:conc_bound_general} and~\ref{sec:conc_poincare} bound the (tail)
probability of a given deviation $a$ at a given finite observation time $t$. 
From a practical perspective, however, the question is usually posed the other way around.
Suppose a time-averaged observable $\overline{A}_t$ has been recorded along a single trajectory of length $t$ (or, say, for a sample of $n$ independent trajectories),
and one is interested in quantifying how much this measured value or the sample mean over the $n$ independent trajectories) 
deviates from the unknown true stationary mean $\E^\mu[\overline{A}_t]$
with some desired level of confidence.
Equally, one may ask how long the trajectory needs to be for
the estimate $\overline{A}_t$ to have a specified accuracy. 
In this section we invert the explicit bounds of Section~\ref{sec:conc_poincare} 
to construct non-asymptotic performance guarantees that allow us to assess and control the uncertainty 
in the inferred additive functionals at any time and any sample size.

Concretely, we introduce \emph{non-asymptotic confidence intervals} 
such that the true mean $\E^\mu[\overline{A}_t]$ lies within a 
distance or radius $r(\alpha,t)$ from the observed estimate $\overline{A}_t$, i.e., 
with a probability of at least $1-\alpha$ we find
$\E^\mu[\overline{A}_t]\in\bigl[\overline{A}_t- r(\alpha,t), \overline{A}_t +r(\alpha,t)\bigr]$
for every $t>0$ (Theorem~\ref{thm:ci}) and for a sample mean over $n$ independent realizations for every $n\ge 1$ (Proposition~\ref{prop:sample_mean} and Corollary~\ref{cor:ci_sample_mean}).
Moreover, we introduce the required \emph{minimal observation time}
as an explicit threshold $\tmin(\varepsilon,\alpha)$ such that for all $t \geq
\tmin(\varepsilon,\alpha)$ we have $r(\alpha,t) \leq \varepsilon$, i.e.,
after this time the deviation of the estimate does not
exceed a specified tolerance $\varepsilon$ with a confidence level $1-\alpha$ (Theorem~\ref{thm:tmin}).
At an operational level this makes precise of what constitutes a ``sufficiently long'' trajectory. 
We then extend the result to sample-mean deviations and define correspondingly 
"sufficiently many independent realizations" given the trajectory length (Theorem~\ref{thm:min_sample}).
$\overline{A}_t$ is a statistical estimator of
$\E^\mu[\overline{A}_t]$ and the following results provide a 
finite-time and finite-sample framework for evaluating its uncertainties that requires no asymptotic arguments or Gaussian approximations,
and no knowledge of the underlying process (i.e., not the drift $b(x)$ nor transition density $p(x,t)$) beyond a handful of (accessible) parameters.

We remark that in the following discussion we treat the empirical density and the
empirical current simultaneously. This is possible because the explicit
bounds in Section~\ref{sec:conc_poincare},
i.e.~Corollary~\ref{cor:bernstein_density} for the density and
Theorems~\ref{thm:explicit_current_1}~and~\ref{thm:explicit_current_2}
for the current, share the same abstract sub-gamma structure, and the
difference in how different additive functionals concentrate is
manifested in the effective parameters $\tilde\sigma_A^2$ and $c_A$. 
Moreover, we note that the result addressing uncertainty quantification is in fact new for both observables.
Although concentration bounds for density-type additive functionals
have been available in the literature, they have, to the best of
our knowledge, not been used to construct explicit confidence intervals
and minimal observation times before. We therefore show the
inversion from tail bounds to confidence intervals at an abstract level
once by using the parameters $\tilde\sigma_A^2$, and $c_A$ and
specialize to the respective observable later.

To this end we briefly recall the general version of the two-sided bounds. 
By Corollaries~\ref{cor:two_sided_density},~\ref{cor:two_sided_current} and~\ref{cor:two_sided_current_explicit}, for each
observable $\overline{A}_t \in \{\overline{\rho}_t(V),
\overline{J}_t(U)\}$ and corresponding parameters $\tilde\sigma_A^2$ and $c_A$ 
we have, for all $t>0$ and $a>0$,
\begin{align}
  \mathbb{P}^\nu\left(\left|\overline{A}_t -
  \E^\mu[\overline{A}_t]\right| \geq a\right)
  \leq
  2\left\|\frac{\dd\nu}{\dd\mu}\right\|_{L^2(\mu)}
  \exp\left[-t\frac{\tilde\sigma_A^2}{c_A^2}
  h\left(\frac{c_A a}{\tilde\sigma_A^2}\right)\right],
  \label{eq:uq_general}
\end{align}
with $h(u) = 1 + u - \sqrt{1+2u}$.

\subsection{Non-asymptotic confidence intervals}
\label{sec:uq_ci}
To go from a tail bound~\eqref{eq:uq_general} to a confidence
interval we have to invert the respective concentration inequality.
For a given acceptable error probability $\alpha$ and
observation time $t$ we search for the smallest deviation $a$ at which the
right-hand side of Eq.~\eqref{eq:uq_general} drops below $\alpha$. 
Importantly, the sub-gamma function $h$ can be inverted in closed form 
(see also~\cite{boucheron2013concentration}) such that the inversion 
comes at no additional cost, i.e., the confidence bound is as tight as the underlying concentration bound.
Moreover, because the bound is two-sided and symmetric in both tails, a single inversion already 
gives the full interval.
This different to non-asymptotic confidence intervals for empirical first-passage times~\cite{bebon2023controlling},
where the asymmetric (more complicated) structure of the left tail 
does not allow an explicit inversion, such that the confidence interval needs to be defined only 
implicitly and an explicit radius could be only given for the right tail.

\begin{lemma}[Inverse sub-gamma function]
\label{lem:h_inverse}
The characteristic sub-gamma function $h(u) = 1 + u - \sqrt{1+2u}$ 
is a strictly increasing
function from $[0,\infty)$ onto $[0,\infty)$ with inverse
\begin{align}
  h^{-1}(v) = v + \sqrt{2v},
  \qquad 
  v\geq 0.
  \label{eq:h_inverse}
\end{align}
Consequently, for any $v \geq 0$,
$h(u) \geq v$ if and only if $u \geq v + \sqrt{2v}$.
\end{lemma}

\begin{proof}
The monotonicity and convexity properties on $[0,\infty)$ follow directly from
$h'(u) = 1 - (1+2u)^{-1/2} \geq 0$ and $h''(u) = (1+2u)^{-3/2} > 0$,
with $h(0)=0$ and $h(u)\to\infty$ as $u\to\infty$. 
To obtain the inverse, insert
$u = v + \sqrt{2v}$ into $1+2u$, yielding
$1 + 2v + 2\sqrt{2v} = (1+\sqrt{2v})^2$.
Therefore, $h(v+\sqrt{2v}) = 1 + v + \sqrt{2v} - (1+\sqrt{2v}) = v$.
\end{proof}

\begin{theorem}[Non-asymptotic confidence intervals]
\label{thm:ci}
Under assumption~\ref{ass:dynamics}, let $\dd\nu/\dd\mu \in
L^2(\mu)$, and let $\tilde\sigma_A^2$ and $c_A$ be sub-gamma parameters for which
Eq.~\eqref{eq:uq_general} holds with $\overline{A}_t \in
\{\overline{\rho}_t(V), \overline{J}_t(U)\}$. 
Fix a confidence level $1-\alpha$, with $\alpha\in (0,1)$, and define
\begin{align}
  r_A(\alpha,t)
  =
  \sqrt{\frac{2\tilde\sigma_A^2}{t}\log\left(\frac{2N_\nu}{\alpha}\right)}
  +
  \frac{c_A}{t}\log\left(\frac{2N_\nu}{\alpha}\right),
  \quad
  N_\nu \equiv \left\|\frac{\dd\nu}{\dd\mu}\right\|_{L^2(\mu)},
  \label{eq:ci_radius}
\end{align}
where $N_\nu$ is a prefactor encoding the initial condition with $N_\mu=1$.
Then for every $t > 0$,
\begin{align}
  \mathbb{P}^\nu\left(-r_A(\alpha,t) \leq \overline{A}_t - \E^\mu[\overline{A}_t] \leq r_A(\alpha,t)\right)
  \geq 1-\alpha ,
  \label{eq:ci_twosided}
\end{align}
that is, with probability of at least $1-\alpha$, under $\mathbb{P}^\nu$, the stationary mean lies in the confidence interval
\begin{align}
  \E^\mu[\overline{A}_t]
  \in
  \left[\overline{A}_t - r_A(\alpha,t),
  \overline{A}_t + r_A(\alpha,t)\right].
  \label{eq:ci_interval}
\end{align}
Moreover, the corresponding one-sided guarantees hold with a smaller radius $r_A^{(1)}(\alpha,t)$, 
obtained from Eq.~\eqref{eq:ci_radius} by replacing $2N_\nu$ with $N_\nu$, i.e., 
\begin{align}
  \mathbb{P}^\nu\left(\overline{A}_t - \E^\mu[\overline{A}_t]
  \leq r_A^{(1)}(\alpha,t)\right)
  \geq 1-\alpha,
  \qquad
  \mathbb{P}^\nu\left(\E^\mu[\overline{A}_t] -  \overline{A}_t
  \leq r_A^{(1)}(\alpha,t)\right)
  \geq 1-\alpha.
  \label{eq:ci_onesided}
\end{align}
\end{theorem}

\begin{proof}
Using the two-sided bound of Eq.~\eqref{eq:uq_general} it suffices to show that 
the upper bound is at most equal to $\alpha$ for $a = r_A(\alpha,t)$. 
Therefore, we use the two-sided concentration inequality and set 
$2N_\nu\exp[-t(\tilde\sigma_A^2/c_A^2)h(c_A a/\tilde\sigma_A^2)]
\leq \alpha$, which is equivalent to
\begin{align}
  h\left(\frac{c_A a}{\tilde\sigma_A^2}\right)
  \geq
  \frac{c_A^2}{\tilde\sigma_A^2 t}
  \log\left(\frac{2N_\nu}{\alpha}\right)
  \equiv v.
\end{align}
By Lemma~\ref{lem:h_inverse} this holds if and only if
$c_A a/\tilde\sigma_A^2 \geq v + \sqrt{2v}$, i.e.,
\begin{align}
  a
  \geq
  \frac{\tilde\sigma_A^2}{c_A}v
  + \frac{\tilde\sigma_A^2}{c_A}\sqrt{2v}
  =
  \frac{c_A}{t}\log\left(\frac{2N_\nu}{\alpha}\right)
  + \sqrt{\frac{2\tilde\sigma_A^2}{t}
  \log\left(\frac{2N_\nu}{\alpha}\right)}
  \equiv r_A(\alpha,t),
\end{align}
where the first equality follows using the definition of $v$ from above.
By construction the event $|\overline{A}_t - \E^\mu[\overline{A}_t]| \geq
r_A(\alpha,t)$ thus has a probability of at most $\alpha$ and
which completes the proof of Eqs.~\eqref{eq:ci_twosided} and~\eqref{eq:ci_interval}.
The one-sided guarantees in Eq.~\eqref{eq:ci_onesided} follow
analogously with the one-sided bounds of
Theorems~\ref{thm:explicit_current_1}~and~\ref{thm:explicit_current_2}
and Corollary~\ref{cor:bernstein_density} 
with the prefactor $N_\nu$ instead of $2N_\nu$.
\end{proof}

\subsection{Minimal observation times}
\label{sec:uq_tmin}
The confidence intervals of Theorem~\ref{thm:ci} answer the question of
given a trajectory of length $t$, how accurate the estimate $\overline{A}_t$ is.
Here, we now answer the complementary direction of given a desired accuracy (or tolerance)
$\varepsilon$ and confidence level $1-\alpha$, how long does the trajectory need to be?
In particular, since $r_A(\alpha,t)$ is strictly decreasing in $t$, this minimal time 
is found as the solution of setting $r_A(\alpha,t) = \varepsilon$.

\begin{theorem}[Minimal observation time]
\label{thm:tmin}
Under the assumptions of Theorem~\ref{thm:ci}, we fix a desired accuracy (or tolerance)
$\varepsilon>0$ at a confidence level $1-\alpha$, with $\alpha\in(0,1)$, and define
\begin{align}
  \tminA(\varepsilon,\alpha)
  \equiv
  \frac{1}{2\varepsilon^2}
  \left(\sqrt{\tilde\sigma_A^2}+ \sqrt{\tilde\sigma_A^2 + 2 c_A \varepsilon}
  \right)^{2}
  \log\left(\frac{2N_\nu}{\alpha}\right).
  \label{eq:tmin}
\end{align}
Then $r_A(\alpha,t) \leq \varepsilon$ if and only if $t \geq
\tminA(\varepsilon,\alpha)$.
In particular, for every $t \geq \tminA(\varepsilon,\alpha)$, it holds that
\begin{align}
  \mathbb{P}^\nu\left( -\varepsilon \leq \overline{A}_t - \E^\mu[\overline{A}_t]\leq \varepsilon\right)\geq 1 - \alpha.
  \label{eq:tmin_guarantee}
\end{align}
Moreover, the minimal time $\tminA$ admits the simple two-sided estimate
\begin{align}
  \frac{2\tilde\sigma_A^2}{\varepsilon^2}\log\left(\frac{2N_\nu}{\alpha}\right)
  \leq
  \tminA(\varepsilon,\alpha)
  \leq
  \left(\frac{2\tilde\sigma_A^2}{\varepsilon^2}+ \frac{2 c_A}{\varepsilon}\right)
  \log\left(\frac{2N_\nu}{\alpha}\right),
  \label{eq:tmin_simple}
\end{align}
and the corresponding one-sided guarantees of Eq.~\eqref{eq:ci_onesided} hold at
tolerance $\varepsilon$ for all $t \geq \tminA^{(1)}(\varepsilon,\alpha)$, where
$\tminA^{(1)}(\varepsilon,\alpha)$ is obtained
from Eq.~\eqref{eq:tmin} by replacing $2N_\nu$ with $N_\nu$.
\end{theorem}

\begin{proof}
By the substitution $u \equiv \sqrt{\log(2N_\nu/\alpha)/t}$ the 
required equation $r_A(\alpha,t) = \varepsilon$ becomes a quadratic equation
$c_A u^2 + u\sqrt{2\tilde\sigma_A^2} - \varepsilon = 0$ in $u \geq 0$,
with a unique non-negative root given by
\begin{align}
  u^\ast
  = \frac{\sqrt{2\tilde\sigma_A^2 + 4c_A\varepsilon}
  - \sqrt{2\tilde\sigma_A^2}}{2c_A}
  = \frac{2\varepsilon}{\sqrt{2\tilde\sigma_A^2 + 4c_A\varepsilon}+ \sqrt{2\tilde\sigma_A^2}},
\end{align}
where the second expression follows by rationalizing. 
Therefore using $u^\ast$ in the substitution and solving for $t$
gives the minimal observation time 
\begin{align}
\tminA = \frac{\log\left(\frac{2N_\nu}{\alpha}\right)}{{u^\ast}^2} 
= \log\left(\frac{2N_\nu}{\alpha}\right)
\frac{\left(\sqrt{2\tilde\sigma_A^2+4c_A\varepsilon} + \sqrt{2\tilde\sigma_A^2}\right)^2}{4\varepsilon^2},
\end{align}
which gives Eq.~\eqref{eq:tmin}. 
The monotonicity of $r_A(\alpha,t)$ in $t$ gives the stated equivalence
$r_A(\alpha,t) \leq \varepsilon$ if and only if $t \geq
\tminA(\varepsilon,\alpha)$, and 
Eq.~\eqref{eq:tmin_guarantee} follows directly from Theorem~\ref{thm:ci}. 
The two-sided sandwich bound follows from
the elementary inequality $4x \leq (\sqrt{x}+\sqrt{x+y})^2\leq  4x+2y$,
for $x\geq0$ and $y\geq 0$. The upper bound follows from the AM-GM inequality and
the lower bound follows since $x\leq x+y$.
Taking $x=\tilde\sigma_A^2$ and $y=2c_A\varepsilon$
yields Eq.~\eqref{eq:tmin_simple}.
Corresponding one-sided statements
follow from the analogous approach using the one-sided radius
$r_A^{(1)}(\alpha,t)$ in Theorem~\ref{thm:ci}.
\end{proof}

Since under detailed balance the
exponential tail parameter of Approach~1 for current-type functionals 
vanishes (Corollary~\ref{cor:db_subgaussian} and Remark~\ref{rem:db_bounds}),
the confidence radius loses its
(heavier tail) $1/t$ correction and reads 
$r_J(\alpha,t)=\sqrt{2\tilde{\sigma}^2_{J,1} \log(2N_\nu/\alpha)/t}$.
Moreover, $\tilde{\sigma}^2_{J,1}=2 \|Q \|_{L^\infty(\mu)}$ such that
the minimal observation time becomes fully explicit and reads
\begin{align*}
    t_{\mathrm{min},J}(\varepsilon,\alpha) = \frac{4 \|Q \|_{L^\infty(\mu)}}{\varepsilon^2} \log\left(\frac{2N_\nu}{\alpha}\right).
\end{align*}
While under detailed balance this minimal trajectory length scales as $\varepsilon^{-2}$, out of equilibrium an additional
$c_{J,1}/\varepsilon$ (exponential) contribution needs to be accounted for.
With $\Theta_J(\varepsilon)=2\tilde{\sigma}^2_{J,1}/\varepsilon^2$
for detailed balance an analog simplification 
can be applied to the minimal sample size of Theorem~\ref{thm:min_sample}.

\subsection{Sample means over independent trajectories}
\label{sec:uq_ensemble}
Our preceding results quantify the uncertainty of an estimate inferred from a
\emph{single} trajectory of length $t$. 
In many settings one instead has access to a small number of
independent realizations of the dynamics (e.g., $n$ repetitions of an experiment
or $n$ parallel simulations), where each sample path is typically started from the same initial law $\nu$ 
and measured for the same duration $t$.
Hereby, writing $\overline{A}_t^{(1)},\dots,\overline{A}_t^{(n)}$ for the
corresponding $n$ independent copies of $\overline{A}_t$ under $\mathbb{P}^\nu$, 
the estimator of the stationary mean becomes the \emph{sample mean of time-averages}, i.e., 
\begin{align}
  \widehat{A}_{n,t}
  \equiv
  \frac{1}{n}
  \sum_{i=1}^{n}\overline{A}_t^{(i)}.
  \label{eq:sample_mean}
\end{align}
The relevant practical question now is, for a given accuracy and confidence level,  how many trajectories are required at a fixed observation time $t$.
Moreover, assuming we have a fixed total amount $T\equiv nt$, is it best to have a single long trajectory or many individual ones?
Indeed, the answer depends on whether
the dynamics is initialized at stationary.
As before, the cost to start away from the steady-state is encoded by the prefactor
$N_\nu = \|\dd\nu/\dd\mu\|_{L^2(\mu)}\geq 1$ with equality if and only if $\nu = \mu$.

\begin{proposition}[Concentration of the sample mean]
\label{prop:sample_mean}
Let the assumptions of Theorems~\ref{thm:explicit_density},~\ref{thm:explicit_current_1}
and~\ref{thm:explicit_current_2} hold, with the corresponding sub-gamma parameters
$\tilde\sigma_A^2$ and $c_A$ established there for $\overline{A}_t \in
\{\overline{\rho}_t(V), \overline{J}_t(U)\}$. Then for every $n \geq 1$, $t>0$, and
$a>0$, the sample mean $\widehat{A}_{n,t}$ satisfies the one-sided bound
\begin{align}
  \mathbb{P}^\nu\left(\widehat{A}_{n,t} - \E^\mu[\overline{A}_t] \geq a\right)
  \leq
  N_\nu^{n}
  \exp\left[-nt\frac{\tilde\sigma_A^2}{c_A^2}
  h\left(\frac{c_A a}{\tilde\sigma_A^2}\right)\right],
  \label{eq:sample_mean_onesided}
\end{align}
and the corresponding two-sided bound
\begin{align}
  \mathbb{P}^\nu\left(\left|\widehat{A}_{n,t} 
  - \E^\mu[\overline{A}_t]\right| \geq a\right)
  \leq
  2 N_\nu^{n}
  \exp\left[-nt\frac{\tilde\sigma_A^2}{c_A^2}
  h\left(\frac{c_A a}{\tilde\sigma_A^2}\right)\right].
  \label{eq:sample_mean_tail}
\end{align}
Here, $\mathbb{P}^\nu$ now denotes the underlying initial law of the $n$ independent trajectories. 
In particular, $\widehat{A}_{n,t}$ obeys the single-trajectory bounds
($n=1$) of Section~\ref{sec:conc_poincare}
and also Eq.~\eqref{eq:uq_general} by replacing $t \to nt$ and $N_\nu \to N_\nu^n$.
\end{proposition}

\begin{proof}
By Corollary~\ref{cor:mgf_bound}, the time-average of a single trajectory obeys the centered
bound
\begin{align}
  \E^\nu\left[\e^{kt(\overline{A}_t - \E^\mu[\overline{A}_t])}\right]
  \leq
  N_\nu \e^{t(\Lambda^A(k) - k\E^\mu[\overline{A}_t])}.
\end{align}
Bounding the symmetrized principal eigenvalue $\Lambda^A(k)$ via
Lemmas~\ref{lem:poincare_density},~\ref{lem:Lambda_J_approach1}
and~\ref{lem:Lambda_J_approach2}, for the respective observable (with the relaxation of
Lemma~\ref{lem:bernstein_current} for the second current strategy),
yields the characteristic sub-gamma form, for $0 \leq k < 1/c_A$,
\begin{align}
  \E^\nu\left[\e^{kt(\overline{A}_t - \E^\mu[\overline{A}_t])}\right]
  \leq
  N_\nu\exp\left[\frac{t\tilde\sigma_A^2 k^2}{2(1-c_A k)}\right].
  \label{eq:single_mgf}
\end{align}
Introduce $s \geq 0$ as the tilt parameter for the sample mean and set $s \equiv knt$. Due to 
independence of the individual trajectories, for $0 \leq s < nt/c_A$,
\begin{align}
  \E^\nu\left[\e^{s(\widehat{A}_{n,t} - \E^\mu[\overline{A}_t])}\right]
  =
  \prod_{i=1}^{n}
  \E^\nu\left[\e^{(s/n)(\overline{A}_t^{(i)} - \E^\mu[\overline{A}_t])}\right]
  \leq
  N_\nu^{n}\exp\left\{\frac{(\tilde\sigma_A^2/nt) s^2}{2\left[1-(c_A/nt) s\right]}\right\},
\end{align}
where each factor is bounded via Eq.~\eqref{eq:single_mgf}
with the replacement $kt = s/n$. 
Hence $\widehat{A}_{n,t}$ is sub-gamma with
variance proxy $\tilde\sigma_A^2/(nt)$, scale parameter $c_A/(nt)$, and the adapted prefactor
reads $N_\nu^{n}$. 
Applying the Cram\'er bound (Lemma~\ref{lem:subgamma}) to the sub-gamma form of the sample-mean therefore
yields the one-sided upper bound~\eqref{eq:sample_mean_onesided}. 
The corresponding left-tail bound follows analogously from the same arguments now applied to $-\overline{A}_t^{(i)}$,
and a union bound yields the two-sided results with an additional 
factor of $2$ in Eq.~\eqref{eq:sample_mean_tail}.
\end{proof}

Since Proposition~\ref{prop:sample_mean} reproduces the single-trajectory setting
via $n=1$, i.e., the bounds are obtained by the substitutions 
$t \to nt$ and $N_\nu \to N_\nu^{n}$, the uncertainty quantification results of
Sections~\ref{sec:uq_ci} and~\ref{sec:uq_tmin} analogously translate to the sample mean setting.

\begin{corollary}[Confidence interval for the sample mean]
\label{cor:ci_sample_mean}
Let the assumptions of Theorem~\ref{thm:ci} hold. Fix a confidence level $1-\alpha$, with
$\alpha \in (0,1)$, and for $n \geq 1$ define
\begin{align}
  r_A(\alpha,n,t)
  =
  \sqrt{\frac{2\tilde\sigma_A^2}{nt}\left(\log\frac{2}{\alpha} 
  + n\log N_\nu\right)}
  +
  \frac{c_A}{nt}\left(\log\frac{2}{\alpha} + n\log N_\nu\right).
  \label{eq:ci_sample_mean}
\end{align}
Then for every $n \geq 1$ and $t > 0$, it holds that
\begin{align}
  \mathbb{P}^\nu\left(-r_A(\alpha,n,t) \leq \widehat{A}_{n,t} 
  - \E^\mu[\overline{A}_t]
  \leq r_A(\alpha,n,t)\right)
  \geq 1-\alpha,
  \label{eq:ci_sample_mean_twosided}
\end{align}
that is, with probability of at least $1-\alpha$, under $\mathbb{P}^\nu$, the stationary mean
is found in the confidence interval
\begin{align}
  \E^\mu[\overline{A}_t]
  \in
  \left[\widehat{A}_{n,t} - r_A(\alpha,n,t),
  \widehat{A}_{n,t} + r_A(\alpha,n,t)\right].
  \label{eq:ci_sample_mean_interval}
\end{align}
Corresponding one-sided guarantees hold with the smaller radius
$r_A^{(1)}(\alpha,n,t)$ obtained by replacing $\log(2/\alpha)$ with $\log(1/\alpha)$
in Eq.~\eqref{eq:ci_sample_mean}.
\end{corollary}

\begin{proof}
The results follow by the substitutions $t \to nt$ and $N_\nu \to N_\nu^{n}$ in Theorem~\ref{thm:ci}
and the underlying required bound holds by Proposition~\ref{prop:sample_mean}. 
Note, that the logarithmic factor now reads
$\log(2N_\nu^n/\alpha) = \log(2/\alpha) + n\log N_\nu$.
The one-sided radius follows as before 
from the one-sided bound~\eqref{eq:sample_mean_onesided}.
Here the (single) prefactor $N_\nu^n$ replaces $\log(2/\alpha)$ with $\log(1/\alpha)$.
\end{proof}

\begin{theorem}[Minimal sample size and observation time]
\label{thm:min_sample}
Under the assumptions of Theorem~\ref{thm:ci}, 
we fix a desired accuracy (or tolerance)
$\varepsilon>0$ at a confidence level $1-\alpha$, with $\alpha\in(0,1)$, and introduce
\begin{align}
  \Theta_A(\varepsilon)
  \equiv
  \frac{1}{2\varepsilon^2}
  \left(\sqrt{\tilde\sigma_A^2} 
  + \sqrt{\tilde\sigma_A^2 + 2c_A\varepsilon}\right)^{2},
  \label{eq:Omega}
\end{align}
such that the minimal required time for a single trajectory (Theorem~\ref{thm:tmin})
reads $\tminA(\varepsilon,\alpha) = \Theta_A(\varepsilon)\log(2N_\nu/\alpha)$. 

Then $\mathbb{P}^\nu(-\varepsilon \leq \widehat{A}_{n,t} - \E^\mu[\overline{A}_t] \leq \varepsilon) \geq 1-\alpha$, i.e., 
the confidence radius stays below the tolerance level, $r_A(\alpha,n,t) \leq \varepsilon$, with 
probability of at least $1-\alpha$, if and only if

\begin{align}
  nt \geq \Theta_A(\varepsilon)\left(\log\frac{2}{\alpha} + n\log N_\nu\right).
  \label{eq:uq_full}
\end{align}
This fixes the minimal sample size and the minimal per-trajectory time as follows.

For fixed $t > \Theta_A(\varepsilon)\log N_\nu$, the
guarantee holds if and only if $n \geq n_{\min,A}$, where
\begin{align}
n_{\min,A}(\varepsilon,\alpha,t)
=\left\lceil\frac{\Theta_A(\varepsilon)\log(2/\alpha)}{t - \Theta_A(\varepsilon)\log N_\nu}
\right\rceil.
\label{eq:nmin}
\end{align}
denotes the \emph{minimal sample size}. In the case that $t \leq \Theta_A(\varepsilon)\log N_\nu$, 
no finite number of independent trajectories is sufficient. 
At stationarity $N_\mu = 1$ we recover 
    \begin{align}
      n_{\min,A}(\varepsilon,\alpha,t)
      =\left\lceil\frac{\tminA(\varepsilon,\alpha)}{t}\right\rceil,
      \label{eq:nmin_stationary}
    \end{align}
and the required sample size is simple the minimal time of a single trajectory
(Theorem~\ref{thm:tmin}) divided by the per-trajectory length, and any $t>0$ 
has a finite sample size.

For fixed $n$, Eq.~\eqref{eq:uq_full} holds if and only if $t \geq t_{\min,A}(n,\varepsilon, \alpha)$,
where
\begin{align}
  t_{\min,A}(n, \varepsilon, \alpha)
  =\frac{\Theta_A(\varepsilon)}{n}\left(\log\frac{2}{\alpha} + n\log N_\nu\right).
  \label{eq:tmin_n}
\end{align}
Notably, in the case $n=1$ this recovers $\tminA(\varepsilon,\alpha)$ from 
Eq.~\eqref{eq:tmin}.
The corresponding one-sided guarantees hold under the same conditions with
$\log(2/\alpha)$ replaced by $\log(1/\alpha)$.
\end{theorem}

\begin{proof}
By Proposition~\ref{prop:sample_mean}, the sample mean 
$\widehat{A}_{n,t}$ satisfies the two-sided bound~\eqref{eq:uq_general} 
with $t \to nt$ and $N_\nu \to N_\nu^n$.
The inversion of Eq.~\eqref{eq:uq_general} in Theorem~\ref{thm:tmin}
therefore gives $r_A(\alpha,n,t) \leq \varepsilon$ 
if and only if $nt \geq \Theta_A(\varepsilon)\log(2N_\nu^{n}/\alpha) =
\Theta_A(\varepsilon)(\log(2/\alpha) + n\log N_\nu)$, which gives Eq.~\eqref{eq:uq_full}.
The minimal sample size and observation time per trajectory follow by 
solving Eq.~\eqref{eq:uq_full} for $n$, for $t$, respectively.
\end{proof}

\section{Conclusion}\label{sec:conclusion}
\subsection{Summary} 
Time-series analysis, in particular time averaging along individual trajectories as a proxy for steady-state
averages, naturally leads to the analysis of additive functionals. Functionals of the Stratonovich-type, also known as ``generalized currents'' in the Physics literature, recently attracted particularly much attention 
in the context of thermodynamic inference.
Experimental constraints and intrinsic sampling limitations lead to a limited (often only a few) trajectories which are short compared to the mixing time of the underlying dynamics. This poses grand challenges on the inference, in particular on the control of uncertainty of estimates. Finite-time fluctuations
are therefore an intrinsic feature of time-averaged observables and asymptotic statements (ergodic theorem or large-deviation limits) alone generally cannot quantify fluctuations in practical applications.

State-of-the-art physical applications critically hinge on a sufficient sampling of higher-order moments of the additive functional, which is challenging to control with the limited information available about the microscopic dynamics. This poses a pressing need for non-asymptotic 
concentration results which for Stratonovich currents are virtually nonexistent. Due to the higher complexity of the corresponding Feynman-Kac semigroup (i.e.\ the "tilt" is not an additive potential), the analysis of the concentration behavior of Stratonovich currents turns out to be much more challenging than the well known results on "classical" Lebesgue-type additive functionals of the density type. 

In our work we prove---to the best of our knowledge for the first time---concentration inequalities for Stratonovich functionals for any bounded, sufficiently smooth vector-valued observable of a general geometrically ergodic diffusion process, including explicit sub-gamma and Bernstein-type inequalities. We contrast these results to the simpler results for functionals of the density type, which we re-derive for comparison. As a side result we further obtain explicit upper bounds on the variance of current functionals and discuss the relation of our results to large-deviation statements.   
As an application of the main theorems we construct non-asymptotic confidence intervals for individual realizations as well as small-sample means, yielding explicit minimal trajectory lengths and sample sizes that will be relevant in practical applications. 

\subsection{Thermodynamic bounds on fluctuations} A priori unexpectedly, the derived explicit concentration inequalities feature $L^2(\mu)$ and $L^\infty(\mu)$ norms of the local mean 
velocity $\vs$. From a physical perspective this is striking. Namely, the \emph{observed} steady-state entropy production rate of the 
system---the central object of stochastic thermodynamics---in fact corresponds to $\Sigma^U=\E^\mu[v_s\cdot D^{-1}v_s \mathbb{1}_{\mathrm{supp}(U)}]$ \cite{Min_1999,seifert2012stochastic,seifert2025stochastic}  and bounds the concentration
of any generalized current by replacing the parameters in Theorem~\ref{thm:explicit_current_1}
by
\begin{align}
    \tilde\sigma^2_{J,1}\leq 
    \tilde\sigma_{\mathrm{th},1}^2
  = 2\|Q\|_{L^\infty(\mu)}
    \left(1 + \frac{\Sigma^U}{\lambda_\mathrm{gap}}\right),
  \qquad
  c_{J,1}\leq
  c_{\mathrm{th},1}
= \frac{2\sqrt{\Sigma^U_\infty\|Q\|_{L^\infty(\mu)}}}{\lambda_\mathrm{gap}},
\end{align}
where 
$\Sigma^U_\infty\equiv\|\vs\cdot D^{-1}\vs \mathbb{1}_{\mathrm{supp}(U)} \|_{L^\infty(\mu)}$ 
defines the \emph{observed} maximal dissipation rate (see companion manuscript).
This is a quite deep physical insight which extends the laws of (stochastic) thermodynamics to the concentration behavior of additive functionals of generalized currents at all times, including extreme deviations in the tails.  

\subsection{Open questions} This work established concentration properties of Stratonovich additive functionals around the stationary mean for geometrically ergodic diffusion initiated in a general, sufficiently well-behaved measure. Our present results lay the foundations for generalization to the infinite-dimensional setting. Corresponding results for sub-geometrically ergodic diffusion, however, remain elusive and will require a somewhat different approach as the Poincar\'e inequality does not apply. 
Moreover, it will be interesting to investigate whether
our concentration bounds on Stratonovich-type functionals
can be extended to the hypocoervice setting 
(i.e., dynamics that do not satisfy a Poincar\'e inequality but converge exponentially fast to the stationary state in $L^2(\mu)$), 
where recently Bernstein-type concentration inequalities 
have been proven 
for density-type observables using a modified inner product \cite{Birrell2025}. 
Similarly, 
extensions to the time-inhomogeneous case \cite{Birrell2026}, 
multiplicative noise diffusions~\cite{ZimmerArxiv},
or results on the concentration around the transient, time-dependent mean remain to be established, which will likely require dedicated two-sided bounds and is thus likely to pose a substantially greater challenge.   
Lastly, it will be interesting to explore whether explicit concentration inequalities 
for Stratonovich-type observables 
can be established using different $L^2(\mu)$-functional inequalities, e.g., log-Sobolev or $F$-Sobolev inequalities \cite{wu2000deviation, cattiaux2008deviation}, 
transportation-information inequalities \cite{guillin2009transportation,gao2014bernstein},
or perturbation theory \cite{lezaud2001chernoff}.

\section{Funding} The financial support from the Studienstiftung des Deutschen Volkes (to R.~B.) and from the European Research
Council (ERC) under the European Union’s Horizon Europe
research and innovation program (Grant Agreement
No.~101086182 to A.~G.) 
is gratefully acknowledged. 

\bibliographystyle{imsart-number} 
\bibliography{bibliography}       

\end{document}